\newif\iffinal
\finaltrue	

\documentclass[letterpaper,11pt,reqno]{amsart}
\RequirePackage[utf8]{inputenc}
\usepackage[portrait,margin=3cm]{geometry}
\iffinal\else\usepackage[notref,notcite]{showkeys}\fi
\usepackage{mathrsfs}
\usepackage{hyperref}
\usepackage[foot]{amsaddr}
\usepackage{amssymb,amsthm,amsfonts,amsbsy,latexsym,dsfont}
\usepackage[textsize=small]{todonotes}       
\usepackage{graphicx}
\usepackage[numeric,initials,nobysame]{amsrefs}
\usepackage{upref,setspace}
\usepackage{xcolor,colortbl}
\usepackage{enumerate}
\usepackage{graphicx}
\usepackage{subcaption}   
\usepackage{bm}
\usepackage{tikz}
\usetikzlibrary{calc,intersections,through,backgrounds,shapes.geometric}
\usetikzlibrary{graphs}
\tikzset{every path/.style={line width=.07 cm}}
\usepackage[normalem]{ulem}

\allowdisplaybreaks

\definecolor{darkblue}{rgb}{0,0,0.6}

\def\tcr#1{\textcolor{black}{#1}}

\newenvironment{enumeratei}{\begin{enumerate}[\upshape (i)]}{\end{enumerate}}
\newenvironment{enumeratea}{\begin{enumerate}[\upshape (a)]}{\end{enumerate}}

\usepackage{bbm}

\usepackage{paralist}

\newenvironment{inparaenuma}{\begin{inparaenum}[\upshape \bfseries (a) ]}{\end{inparaenum}}

\numberwithin{equation}{section}
\numberwithin{figure}{section}
\numberwithin{table}{section}

\newtheorem{thm}{Theorem}[section]
\newtheorem{lem}[thm]{Lemma}
\newtheorem{theorem}[thm]{Theorem}

\newtheorem{corollary}[thm]{Corollary}
\newtheorem{prop}[thm]{Proposition}
\newtheorem{defn}[thm]{Definition}

\newtheorem{ass}[thm]{Assumption}

\newtheorem{lemma}[thm]{Lemma}

\theoremstyle{definition}
\newtheorem{rem}{Remark}

\newcommand{\ind}{\mathds{1}}
\newcommand{\eps}{\varepsilon}

\newcommand{\set}[1]{\left\{#1\right\}}

\newcommand{\equald}{\stackrel{\mathrm{d}}{=}}
\newcommand{\probc}{\stackrel{\mathrm{P}}{\longrightarrow}}

\newcommand{\convas}{\stackrel{\mathrm{a.s.}}{\longrightarrow}}

\newcommand{\hght}{\mathrm{ht}}

\newcommand{\stod}{\preceq_{\mathrm{st}}}

\def\qed{ \hfill $\blacksquare$}

\newcommand{\cD}{\mathcal{D}}\newcommand{\cE}{\mathcal{E}}\newcommand{\cF}{\mathcal{F}}
\newcommand{\cG}{\mathcal{G}}\newcommand{\cH}{\mathcal{H}}
\newcommand{\cL}{\mathcal{L}}
\newcommand{\cN}{\mathcal{N}}
\newcommand{\cP}{\mathcal{P}}\newcommand{\cR}{\mathcal{R}}
\newcommand{\cS}{\mathcal{S}}\newcommand{\cT}{\mathcal{T}}
\newcommand{\cV}{\mathcal{V}}
\newcommand{\cZ}{\mathcal{Z}}

\newcommand{\vA}{\mathbf{A}}\newcommand{\vB}{\mathbf{B}}

\newcommand{\vP}{\mathbf{P}}

\newcommand{\vZ}{\mathbf{Z}}

\newcommand{\ve}{\mathbf{e}}

\newcommand{\vp}{\mathbf{p}}
\newcommand{\vs}{\mathbf{s}}\newcommand{\vt}{\mathbf{t}}\newcommand{\vu}{\mathbf{u}}
\newcommand{\vv}{\mathbf{v}}\newcommand{\vx}{\mathbf{x}}

\newcommand{\mvzero}{\boldsymbol{0}}

\newcommand{\mvxi}{\boldsymbol{\xi}}

\newcommand{\fR}{\mathfrak{R}}

\newcommand{\bL}{\mathbb{L}}
\newcommand{\bN}{\mathbb{N}}
\newcommand{\bR}{\mathbb{R}}
\newcommand{\bT}{\mathbb{T}}

\newcommand{\bZ}{\mathbb{Z}}

\DeclareMathOperator{\E}{\mathbb{E}}
\DeclareMathOperator{\pr}{\mathbb{P}}

\DeclareMathOperator{\var}{Var}

 \DeclareMathOperator{\height}{ht}

 \DeclareMathOperator{\dist}{dist}

\DeclareMathOperator{\PR}{PR}

\newcommand{\convd}{\stackrel{d}{\longrightarrow}}
\newcommand{\convp}{\stackrel{P}{\longrightarrow}}
\def\RR{\mathbb{R}}

\usepackage{pdfsync}

\definecolor{aqua}{rgb}{0.0, 1.0, 1.0}
\definecolor{boo}{rgb}{1.0, 0.0, 1.0}
\definecolor{stred}{rgb}{1.0, 0.44, 0.37}

\newcommand{\yu}{{\sf Yule} }

\newcommand{\BRW}{{\sf BRW} }
\newcommand{\MBP}{{\sf MBP}}

\usepackage{accents}

\newcommand{\bbT}{\mathbb{T}}
\newcommand{\TT}{\mathcal{T}}
\newcommand{\bs}{\mathbf{s}}
\newcommand{\bt}{\mathbf{t}}

\newcommand{\convprf}{\stackrel{\mbox{\bf P-fr}}{\longrightarrow}}
\newcommand{\prob}{\mathbb{P}}
\newcommand{\bfomega}{{\boldsymbol \omega}}

\newcommand{\Zbold}{{\mathbb{Z}}}
\DeclareMathAlphabet\mathbfcal{OMS}{cmsy}{b}{n}

\newcommand{\bcF}{\mathbfcal{F}}
\newcommand{\bcP}{\mathbfcal{P}}
\usepackage{accents}
\newcommand\thickbar[1]{\accentset{\rule{.6em}{1.3pt}}{#1}}
\newcommand{\barT}{\thickbar{T}}
\newcommand{\cpst}{\cP^\star}

\def\beq{ \begin{equation} }
\def\eeq{ \end{equation} }
\def\FF{\mathcal{F}}
\def\LL{\mathcal{L}}
\def\ep{\varepsilon}
\def\mn{\medskip\noindent}

\newcommand{\chnr}[1]{\textcolor{black}{{#1}}}

\begin{document}

\title[Co-evolving dynamic networks]{Co-evolving dynamic  networks}

\date{}

\author[Banerjee]{Sayan Banerjee}
\author[Bhamidi]{Shankar Bhamidi}
\address{Department of Statistics and Operations Research, 304 Hanes Hall, University of North Carolina, Chapel Hill, NC 27599}
\email{sayan@email.unc.edu, bhamidi@email.unc.edu, zoehuang@unc.edu}
\author[Huang]{Xiangying Huang}

\begin{abstract}
Co-evolving network models, wherein dynamics such as random walks on the network influence the evolution of the network structure, which in turn influences the dynamics, are of interest in a range of domains. While much of the literature in this area is currently supported by numerics, providing evidence for fascinating conjectures and phase transitions,  proving rigorous results has been quite challenging. 
We propose a general class of co-evolving tree network models  driven by local exploration, started from a single vertex called the root.  New vertices attach to the current network via randomly sampling a vertex and then exploring the graph for a random number of steps in the direction of the root, connecting to the terminal vertex. Specific choices of the exploration step distribution lead to the well-studied affine preferential attachment and uniform attachment models, as well as less well understood dynamic network models with global attachment functionals such as PageRank scores \cite{chebolu2008pagerank}. We obtain local weak limits for such networks and use them to derive asymptotics for the limiting empirical degree and PageRank distribution. We also quantify asymptotics for the degree and PageRank of fixed vertices, including the root, and the height of the network. Two distinct regimes are seen to emerge, based on the expected exploration distance of incoming vertices, which we call the `fringe' and `non-fringe' regimes. These regimes are shown to exhibit different qualitative and quantitative properties. In particular, networks in the non-fringe regime undergo `condensation' where the root degree grows at the same rate as the network size. Networks in the fringe regime do not exhibit condensation. 
\tcr{Further, non-trivial phase transition phenomena are shown to arise for: (a) height asymptotics in the non-fringe regime, driven by the subtle competition between the condensation at the root and network growth; (b) PageRank distribution in the fringe regime, connecting to the well known power-law hypothesis.}  
In the process, we develop a general set of techniques involving local limits, infinite-dimensional urn models, related multitype branching processes and corresponding Perron-Frobenius theory, branching random walks, and in particular relating tail exponents of various functionals to the scaling exponents of quasi-stationary distributions of associated random walks. These techniques are expected to shed light on a variety of other co-evolving network models.

\end{abstract}

\maketitle

\section{Introduction}
\label{sec:int}

\subsection{Motivation}

Driven by the explosion in the amount of data on various real world networks, the last few years have seen the emergence of many new mathematical network models. Goals underlying these models include, \begin{inparaenuma}
	\item extracting unexpected connectivity patterns within the network (e.g. community detection);
	\item understanding properties of dynamics on these real world systems such as the spread of epidemics and opinion dynamics;
	\item understanding mechanistic reasons for the emergence of empirically observed properties of these systems such as heavy tailed degree distribution or the small world property;
\end{inparaenuma} 
see e.g. \cites{albert2002statistical,newman2003structure,newman2010networks,bollobas2001random,durrett-rg-book,van2009random} and the references therein. Within this vast research area, dynamic or temporal networks, namely systems that change over time, play an important role both in applications such as understanding social networks or the evolution of gene regulatory networks \cites{holme2012temporal, holme2019temporal, masuda2016guidance}. One major frontier, especially for developing rigorous understanding of proposed models, are the so-called co-evolutionary (or adaptive) networks, where specific dynamics (e.g. random walk explorations) on the network influence the structure of the network, which in turn influences the dynamics; thus both modalities (dynamics on the network and the network itself) co-evolve \cites{gross2008adaptive,aoki2016temporal,sayama2013modeling,sayama2015social}. Quoting \cite{porter2020nonlinearity+}, {\it ``$\ldots$ adaptive networks provide a promising mechanistic approach to simultaneously explain both structural features and temporal features $\ldots$ and can produce extremely interesting and rich dynamics, such as the spontaneous development of extreme states in opinion models $\ldots$ ''} Despite significant interest in such models, deriving rigorous results has been challenging. Let us describe two concrete motivations behind this paper: 

\begin{enumeratea}
\item {\bf Evolving networks driven through local exploration:} Motivated by the growth of social networks, there has been significant interest in trying to understand the influence of processes such as search engines or influence ranking mechanisms in the growth of networks. A number of papers \cites{pandurangan2002using,blum2006random,chebolu2008pagerank} have explored the dynamic evolution of networks through new nodes first exploring neighborhoods of randomly selected vertices before deciding on whom to connect. Even simple microscopic rules seemed to result in non-trivial phase transitions. 
\item {\bf Preferential attachment models driven by global attachment schemes:} Perhaps the most well known class of co-evolving network models in practice are the so-called preferential attachment class. Typically one fixes an attachment function $f_{att}:\mathbb{N} \to \bR_+$, with $f_{att}(k)$ denoting the attractiveness of a degree $k$ vertex for new vertices joining the system.   New vertices enter the system and attach to existing vertices with probability proportional to their attractiveness. Such models have been used in diverse settings including positing causal mechanisms for heavy tailed degree distributions \cite{barabasi1999emergence}, understanding robustness of networks to attacks \cites{bollobas2003mathematical,bollobas2004robustness}, and modeling retweet networks \cite{bhamidi2012twitter}. Such models require global knowledge of the network at each stage, yet use only the degree of each vertex for attachment, eschewing potentially more relevant global attractiveness functions such as centrality measures of vertices like the PageRank score (described in more detail below).  
\end{enumeratea}

\noindent {\bf Organization of the paper:} We describe the general class of models in Section \ref{sec:mod-def}. Section \ref{sec:not-constr} describes initial constructions required to state the main results. Section \ref{sec:res} contains all the main results, where we also connect the results to existing literature and conjectures. Proofs are commenced in Section \ref{sec:proofs} with individual subsections connecting functionals of the  model to continuous time branching processes, branching random walks, quasi-stationary distributions and so on, and deriving relevant results for this paper. The rest of the Sections then use the technical foundations in Section \ref{sec:proofs} to complete the proofs of the main theorems.

\subsection{Model definition}
\label{sec:mod-def}
Fix a probability mass function $\vp:= \set{p_k:k\geq 0}$ on $\bZ_+$. For the rest of the paper, let $\vZ = \set{Z_1, Z_2, \ldots}$ be an \emph{i.i.d} sequence with distribution $\vp$.  We now describe the recursive construction of a sequence of random trees $\set{\cT_n:n\geq 1}$, always rooted at vertex $\set{v_0}$, with edges pointed from descendants to their parents.  Start with two vertices $\set{v_0, v_1}$, with $\cT_1$ a rooted tree at $\set{v_0}$, an oriented edge from $v_1$ to $v_0$.  Assume for some $n\geq 1$, we have constructed $\cT_n$. Then to construct $\cT_{n+1}$:

\begin{enumeratea}
    \item New vertex $\set{v_{n+1}}$ enters the system at time $n+1$. 
    \item This new vertex selects a vertex $V_n$, \emph{uniformly at random}, amongst the existing vertices $\cV(\cT_n) = \set{v_0, \ldots, v_n} $. 
    \item  Let $\cP(v_0, V_n)$ denote the path from the root to this vertex. This new vertex traverses up this path for a random length $Z_{n+1}$ and attaches to the terminal vertex. If the \chnr{graph distance to the root,} $\dist(v_0, V_n) \leq Z_{n+1}$ then this new vertex attaches to the root $v_0$. 
\end{enumeratea}

\begin{figure}[h]
     \centering     
\begin{subfigure}[b]{0.3\textwidth}
         \centering
         \begin{center}
\begin{tikzpicture}[background rectangle/.style=
		     {draw=red!80,fill=red!10,rounded corners=2ex},
		   show background rectangle]
  \draw[gray](0,0)--(0,-1);
  \draw[gray](0,0)--(-1,-1);
  \draw[gray](-1,-1)--(-1,-2);
  \draw[gray](-1,-2)--(-1,-3);
  \draw[gray](-1,-2)--(-2,-3);
    \filldraw(0,0) circle (2pt) node[above] {$v_0$};
   \filldraw(0,-1) circle (2pt) node[right] {$v_2$};
    \filldraw(-1,-1) circle (2pt) node[right] {$v_1$};
    \filldraw(-1,-2) circle (2pt) node[right] {$v_3$};
     \filldraw(-1,-3) circle (2pt) node[right] {$v_5$};
       \filldraw[red](-2,-3) circle (2pt) node[right] {$v_4$};
       \draw[dashed, blue,->](2,-3)--(1,-2);
       \filldraw[blue](2,-3) circle (2pt) node[right] {$v_6$};
\end{tikzpicture}
\end{center}
\end{subfigure}
\hfill
\begin{subfigure}[b]{0.3\textwidth}
         \centering
         \begin{center}
\begin{tikzpicture}[background rectangle/.style=
		     {draw=red!80,fill=red!10,rounded corners=2ex},
		   show background rectangle]
  \draw[gray](0,0)--(0,-1);
  \draw[gray](0,0)--(-1,-1);
  \draw[gray](-1,-2)--(-1,-3);
 \filldraw(0,0) circle (2pt) node[above] {$v_0$};
   \filldraw(0,-1) circle (2pt) node[right] {$v_2$};
    \filldraw[blue](-1,-1) circle (2pt) node[right] {$v_1$};
    \filldraw[lightgray](-1,-2) circle (2pt) node[right] {$v_3$};
     \filldraw(-1,-3) circle (2pt) node[right] {$v_5$};
       \filldraw[red](-2,-3) circle (2pt) node[right] {$v_4$};
         \draw[->, red](-1,-2)--(-1,-1);
          \draw[->, red](-2,-3)--(-1,-2);   
          
              \draw[dashed, blue,->](2,-3)--(1,-2);
       \filldraw[blue](2,-3) circle (2pt) node[right] {$v_6$};
\end{tikzpicture}
\end{center}
     \end{subfigure}
\hfill     
\begin{subfigure}[b]{0.3\textwidth}
         \centering
     \begin{center}
\begin{tikzpicture}[background rectangle/.style=
		     {draw=blue!80,fill=blue!10,rounded corners=2ex},
		   show background rectangle]
  \draw[gray](0,0)--(0,-1);
  \draw[gray](0,0)--(-1,-1);
  \draw[gray](-1,-1)--(-1,-2);
  \draw[gray](-1,-2)--(-1,-3);
  \draw[gray](-1,-2)--(-2,-3);
    \filldraw(0,0) circle (2pt) node[above] {$v_0$};
   \filldraw(0,-1) circle (2pt) node[right] {$v_2$};
    \filldraw[blue](-1,-1) circle (2pt) node[right] {$v_1$};
    \filldraw(-1,-2) circle (2pt) node[right] {$v_3$};
     \filldraw(-1,-3) circle (2pt) node[right] {$v_5$};
       \filldraw(-2,-3) circle (2pt) node[right] {$v_4$};
          \draw[->, blue](0,-2)--(-1,-1);
        \filldraw[blue](0,-2) circle (2pt) node[right] {$v_6$};
\end{tikzpicture}
\end{center}
     \end{subfigure}
        \caption{$v_6$ is a new incoming vertex, and selects $v_4$ to start exploring the network,  with sampled number of exploration steps $Z_6=2$.}
        \label{fig:three graphs}
\end{figure}
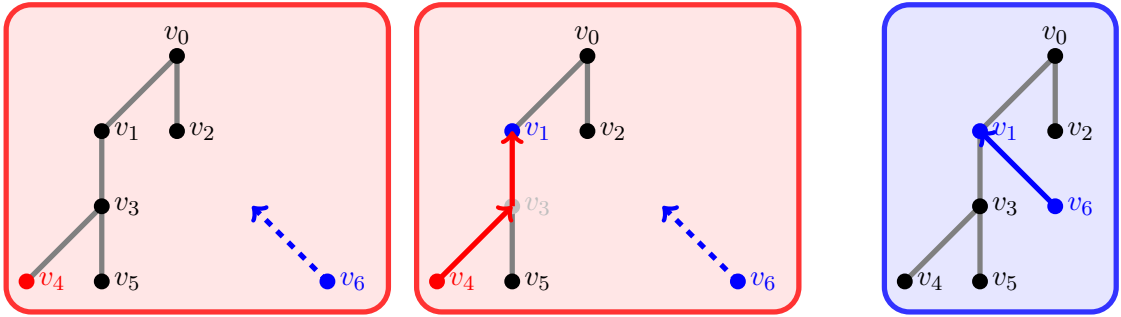

See \chnr{Figure} \ref{fig:three graphs} for a pictorial description. Thus new vertices enter the system and perform local explorations, before attaching themselves to an existing vertex. This, at first sight, simple mechanism, results in a host of important special cases depending on the choice of $\vp$ \chnr{which we describe next}. We will let $\set{\cT_n(\vp):n\geq 1}$ denote the corresponding tree process and suppress dependence on $\vp$ when this is \chnr{clear} from the context. 

\subsubsection{Random recursive tree:} If $p_0 =1$ then one obtains the random recursive tree where new vertices connect to previously existing vertices uniformly at random. See \cites{mahmoud2008polya,smythe1995survey,szymanski1987nonuniform} and the references therein for the extensive use of this model in computer science. 

\subsubsection{Affine preferential attachment:} Suppose $\vp$ is Bernoulli($p$) for some $0< p<1$, namely $p_0 = 1-p, p_1 = p$. Then one can check that the corresponding tree process has the same distribution as an affine preferential attachment model with attachment function $f_{att}(k) = k + (1-2p)/p$. 

\subsubsection{PageRank driven preferential attachment:}
In trying to understand models where vertices try to game search engines and attempt to increase their popularity by connecting to popular existing vertices in the system, one natural approach is via preferential attachment models where attractiveness of existing vertices is measured by their PageRank score (a global, as opposed to more local degree-only based attachment schemes). \chnr{The probabilisitic models of network evolution that at first sight seem to need global information on the network, but then have equivalent representations as local exploration schemes, has inspired a large body of work especially in statistical physics, under the general area of \emph{network growth with redirection} \cites{ben2010random,vazquez2003growing,krapivsky2005network,PhysRevE.70.056107,gabel2013highly,gabel2013sublinear,krapvisky2023magic,krapivsky2017emergent}.}


    \begin{defn}[PageRank scores]
    \label{def:page-rank}
	    For a directed graph $\cG = (\cV, \cE)$, the PageRank scores of vertices $v\in \cV$ with damping factor $c$, is the stationary distribution $(\fR_{v,c}: v\in \cG)$ of the following random walk. At each step, with probability $c$, follow an outgoing edge (uniform amongst available choices) from the current location in the graph while, with probability $1-c$, restart at a uniformly selected vertex in the entire graph. These scores are given by the linear system of equations:
	    \begin{equation}
	        \label{eqn:page-rank}
	        \fR_{v,c} = \frac{1-c}{n} + c\sum_{u\in \cN^{-}(v)} \frac{\fR_{u,c}}{d^{+}(u)}
	    \end{equation}
	    where $\cN^{-}(v)$ is the set of vertices with edges pointed at $v$ and $d^{+}(u)$ is the out-degree of vertex $u$. 
	\end{defn}
     At vertices with zero out-degree (e.g. the root of directed tree), the random walk stays in place with probability $c$ and jumps to a uniformly chosen vertex with probability $1-c$. The stationary probabilities at such vertices are then multiplied by $1-c$ to keep the formula \eqref{eqn:page-rank} same for all vertices. 
     \begin{defn}[PageRank driven preferential attachment {\cite{pandurangan2002using}}]
     Fix damping factor $c \in (0,1)$. Consider the following sequence of directed random trees, started with a root $v_0$ and another vertex $v_1$ with directed edge pointed to $v_0$. At each discrete time step $n\geq 2$, a new vertex $v_n$ enters the system and connects to a previously existing vertex with probability proportional to the PageRank of the existing vertex.  
     \end{defn}

Now consider the process $\set{\cT_n(\vp):n\geq 1}$, with $\vp$ as Geometric($p$), namely for $k\geq 0$, $p_k = p(1-p)^k$. Thus a new vertex selects an existing vertex at random in the current tree and then walks up the path from the selected vertex to the root, wherein at each step, it decides to attach itself to the present location with probability $p$ or move upwards with probability $1-p$.

	\begin{theorem}[{\cite[Thm 1.1]{chebolu2008pagerank}}]
	    The model with $\vp$ as Geometric($p$) has the same distribution as the PageRank driven preferential attachment model with damping factor $c = 1-p$. 
	\end{theorem}



\section{Notation and initial constructions}
\label{sec:not-constr}

\subsection{Mathematical notation}
 For $J\geq 1$ let $[J]:= \set{1,2,\ldots, J}$. If $Y$ has an exponential distribution with rate $\lambda$, write this as $Y\sim \exp(\lambda)$. Write $\bZ$ for the set of integers, $\bR$ for the real line, $\bN$ for the natural numbers and let $\bR_+:=(0,\infty)$. Write $\convas,\convp,\convd$ for convergence almost everywhere, in probability and in distribution respectively. For a non-negative function $g:\bZ_+\to [0,\infty)$ and another function $f:\bZ_+\to \bR$,
write $f(n)=O(g(n))$ when $f(n)/g(n)$ is uniformly bounded, and
$f(n)=o(g(n))$ when $\lim_{n\rightarrow \infty} f(n)/g(n)=0$. Write $f(n) = \Omega(g(n))$ if $\liminf_{n\rightarrow \infty} f(n)/g(n)>0$. 
Furthermore, write $f(n)=\Theta(g(n))$ if $f(n)=O(g(n))$ and $g(n)=O(f(n))$. For two real valued stochastic processes on the same space $\set{X_t:t\geq 0}$ and $\set{Y_t:t\geq 0}$,  denote $X_t = Y_t + o_{a.s.}(1)$ if the random variable $X_t-Y_t \rightarrow 0$ almost surely as $t \rightarrow \infty$.
We write that a sequence of events $(A_n)_{n\geq 1}$
occurs \emph{with high probability} (whp) if $\pr(A_n)\rightarrow 1$ as $n \rightarrow \infty$. We use $\stod$ for stochastic domination between two real valued probability measures. 
For a rooted finite tree $\bt$, let $|\bt|$ denote the number of vertices in $\bt$ and $\height(\bt)$ denote its height, namely, the maximal graph distance of the root from any other vertex. \chnr{For graph $\cG$, we write $\dist_{\cG}(\cdot, \cdot)$ for graph distance and in most cases suppress dependence on $\cG$ when this is clear from context. }

\subsection{Continuous time models}

Recall the description of the model in Section \ref{sec:mod-def}. For much of the proofs, we will work with the following continuous time versions of the above process. 

\begin{defn}[Continuous time versions]
\label{defn:cts-time}
    We let $\set{\cT(t, \vp):t\geq 0}$ denote the continuous time version of the above process wherein $\cT(0,\vp)$ is a tree with vertex set $\set{v_0}$ rooted at $v_0$. 
    Each existing vertex $v$ in the tree reproduces at rate one. When vertex $v$ reproduces, a random variable $Z$ following the law $\vp$ is sampled independently. 
\begin{enumeratea}
\item If $Z\leq dist(v_0,v)$, then a new vertex $\tilde{v}$ is attached to the unique vertex $u$ lying on the path between $v$ and $v_0$ that satisfies $dist(v,u)=Z$ via a directed edge from $\tilde{v}$ to $u$.
\item If $Z> dist(v_0,v)$, attach the new vertex $\tilde{v}$ to the root $v_0$ via a directed edge towards $v_0$.
\end{enumeratea} 
\smallskip
Let $\set{\cT^*(t, \vp):t\geq 0}$ denote the continuous time process which follows the dynamics as above but with (b) above replaced by, 
\begin{enumeratea}
    \item[\text{(b)}$^\prime$] If $Z> dist(v_0,v)$, nothing happens. 
\end{enumeratea}
\end{defn}
It will turn out later that the process  $\cT^*$ describes the evolution of the fringe tree below \emph{non-root} vertices. Define the stopping times
$$T_n=\inf \{ t\geq 0:|\cT(t, \vp)|=n+1\},$$
which is the birth time of the incoming vertex $v_{n}$.
\chnr{Let $\cF_t$ denote the natural sigma-field of the process up to time $t$ and  $\set{\cF_t:t\geq 0}$ the corresponding filtration.} The following is obvious from construction. 
\begin{lem}
\label{lem:cts-disc}
The processes $\set{\cT_n(\vp):n\geq 1}$ and $\set{\cT(T_n,\vp):n\geq 1}$ have the same distribution.
\end{lem}

\subsection{Fringe convergence}
The aim of this section is to formalize the notion of convergence of neighborhoods of large random trees to neighborhoods of limiting infinite discrete structures.  Local weak convergence of discrete random structures  has now become quite standard in probabilistic combinatorics see e.g. \cite{aldous-steele-obj, benjamini-schramm}.  In the case of trees, following \cite{aldous-fringe}, the theory of local graph limits has an equivalent formulation in terms of convergence of fringe distribution.

\subsection{Extended fringe decomposition}
\label{sec:fri-decomp}
 Given two rooted (unlabeled) finite trees $\bs, \bt$, say that $\bs \sim \bt $ or $\bs=\bt$ if there exists a root
preserving isomorphism between the two trees.  For $n\geq 1$, let $ \bT_n $ be the space of all rooted trees on  $n$  vertices and let $ \bbT =
\cup_{n=0}^\infty \bT_n $ be the space of all finite rooted trees. Here $\bT_0 = \emptyset $ will be used to represent the empty tree (tree on zero vertices). Fix a tree $\bt\in \bbT$ with root $\rho$ and a vertex $v$ at distance $h$ from the root.  Let $(v_0 =v, v_1, v_2, \ldots, v_h = \rho)$ be the unique path from $v$ to $\rho$. The tree $\bt$  can be decomposed as $h+1$ rooted trees $f_0(v,\bt), \ldots, f_h(v,\bt)$, where $f_0(v,\bt)$ is the tree rooted at $v$ consisting of all vertices for which there exists a path from the root passing through $v$, and for $i \ge 1$, $f_i(v,\bt)$ is the subtree rooted at $v_i$ consisting of all vertices for which the path from the root passes through $v_i$ but not through $v_{i-1}$. 
  Call the map
$F: \bbT \to \bbT^\infty$
defined by,
\[F(v, \bt) = \left(f_0(v,\bt), f_1(v,\bt) , \ldots, f_h(v,\bt), \emptyset, \emptyset, \ldots \right),\]
as the fringe decomposition of $\bt$ about the vertex $v$. Call $f_0(v,\bt)$ the {\bf fringe} of the tree $\bt$ at $v$.
For $k\geq 0$, call $F_k(v,\bt) = (f_0(v,\bt) , \ldots, f_k(v,\bt))$ the {\bf extended fringe} of the tree $\bt$ at $v$ truncated at distance $k$ from $v$ on the path to the root.

Now consider the space $\bbT^\infty$. An element $\bfomega = (\bt_0, \bt_1, \ldots) \in \bbT^\infty$ with $|\bt_i|\geq 1$ for all $ i\geq 0$,  can be thought of as a locally finite infinite rooted tree with a {\bf s}ingle  path to {\bf in}finity (thus called a {\tt sin}-tree \cite{aldous-fringe}) as follows: Identify the sequence of roots of $\set{\bt_i:i\geq 0}$ with the integer lattice $\Zbold_+ = \set{0,1,2,\ldots}$, equipped with the natural nearest neighbor edge set, rooted at $\rho=0$.  Analogous to the definition of extended fringes for finite trees, for any $k\geq 0$ write 
$F_k(0,\bfomega)= (\bt_0, \bt_1, \ldots, \bt_k)$. 
Call this the extended fringe of the tree $\bfomega$ at vertex $0$, till distance $k$, on the infinite path from $0$. Call $\bt_0 = F_0(0,\bfomega)$ the {\bf fringe} of the {\tt sin}-tree $\bfomega$. Now suppose $\prob$ is a probability measure on $\bbT^\infty$ such that if $\TT:= (\bt_0(\TT), \bt_1(\TT),\ldots)\sim \prob$, then  $|\bt_i(\TT)|\geq 1$ a.s.  $\forall~ i\geq 0$. Then $\TT$ can be thought of as an infinite {\bf random} {\tt sin}-tree.

\subsection{Convergence on the space of trees}
\label{sec:fringe-convg-def}

 Let $\TT_\infty$ be a random {\tt sin}-tree with distribution $\prob$ on $\bbT^\infty$.  
Suppose $\set{\cT_n}_{n\geq 1}$ be a sequence of {\bf finite} rooted random trees all constructed on the same probability space (for notational convenience  assume $|\cT_n|= n$, all one needs is $|\cT_n|\probc \infty$). Fix $k\geq 0$ and (non-empty) trees $\bs_0,\bs_1,\ldots, \bs_k \in \bbT$. Let $\hat{\bs} = (\bs_0,\ldots, \bs_k)\in \bbT^{k+1}$ and define the empirical proportions,
\[H_n^k(\hat{\bs}) = \frac{1}{n} \sum_{v\in \cT_n} \ind\set{F_k(v,\cT_n) = (\bs_0,\bs_1,\ldots, \bs_k)}.\]
As before ``$f_j(v,\cT_n) = \bs_j$'' implies identical up to a root preserving isomorphism. 
Consider two (potentially distinct) notions of convergence of $\set{\cT_n:n\geq 1}$:
\begin{enumeratea}
    \item \label{it:fringe-a} Fix a probability measure $\pi$ on $\bT$. Say that a sequence of trees $\set{\cT_n}_{n\geq 1}$ converges in probability, in the fringe sense, to $\pi$, if for every  rooted tree $\vt\in \bT$, $H_n^{0}(\bt) \probc \pi(\bt)$. 
    \item \label{it:fringe-b} Say that a sequence of trees $\set{\cT_n}_{n\geq 1}$ converges in probability, in the {\bf extended fringe sense}, to $\TT_{\infty}$ if for all $k\geq 0$ and $\hat{\bs} \in \bbT^{k+1}$, one has
  \[H_n^k(\hat{\bs}) \stackrel{P}{\longrightarrow} \prob\left(F_k(0,\TT_{\infty}) = \hat{\bs} \right).\]
We shall denote this convergence by $\TT_n\convprf \TT_{\infty}$ as $n\to\infty$.
\end{enumeratea}
Letting $\pi_{\cT_{\infty}}(\cdot) = \pr(F_0(0, \cT_{\infty}) = \cdot)$ denote the distribution of the fringe of $\cT_\infty$ on $\bT$, convergence in (b) above implies convergence in notion (a). Further, both notions imply convergence of functionals such as the degree distribution. For example notion (a) implies,  for any $k\geq 0$, 
\begin{equation}
\label{eqn:deg-convg-fr}
	\frac{\#\set{v\in \TT_n, \deg(v) = k+1}}{n} \convp \prob(\deg(0,\TT_{\infty})=k).
\end{equation}
Here $\deg(0,\TT_{\infty})$ denotes the number of edges connected to $0$ in its fringe $\bt_0(\TT_{\infty})$. However, this convergence gives much more information about the asymptotic properties of $\cT_n$, including convergence of global functionals \cite{bhamidi2012spectra}.

In terms of going from convergence in (a) to (b), Aldous in \cite{aldous-fringe} showed that, amongst the distributions $\pi$ that arise as potential limits of the fringe convergence in (a), there is a special sub-class of measures called fringe distributions, which automatically imply the existence of and convergence to a limit infinite {\tt sin}-tree.  For each $\vs\in \bT$, suppose the root has children $v_1, v_2, \ldots, v_d$ for some $d\geq 0$. Let $\set{f(\vs, v_i): 1\leq i\leq d}$ denote the subtrees at these vertices, with $f(\vs, v_i)$ rooted at $v_i$. For each $\vt\in \bT$, let $Q(\vs, \vt) = \sum_{i} \ind\set{f(\vs, v_i) = \vt}$.  

\begin{defn}
    \label{def:fringe}
    Call a probability distribution $\pi$ on $\bT$ a {\bf fringe distribution} if, 
    \[\sum_{\vs \in \bT} \pi(\vs) Q(\vs, \vt) = \pi(\vt), \qquad \forall~ \vt\in \bT. \]
\end{defn}
It is easy to check that the space of fringe measures is a convex subspace of the space of probability measures on $\bT$.  The next result summarizes some of the remarkable findings in \cite{aldous-fringe}, relevant to this paper.

\begin{theorem}[{\cite{aldous-fringe}}]
\label{thm:aldous}
Fix a probability measure $\pi$ on $\bT$.  Suppose $\set{\cT_n}_{n\geq 1}$ converges in probability in the fringe sense to $\pi$. 
\begin{enumeratea}
    \item $\pi$ is a fringe distribution in the sense of Definition \ref{def:fringe} iff $\sum_{\vt} \pi(\vt) \text{root-degree}(\vt) = 1$. 
    \item If $\pi$ is a fringe distribution then, convergence in probability in the fringe sense (notion \eqref{it:fringe-a}), implies the existence of a random {\tt sin}-tree $\cT_{\infty}$ such that $\set{\cT_n}_{n\geq 1}$ converges in probability in the extended fringe sense to $\cT_\infty$ (notion \eqref{it:fringe-b}). 
    \item If $\pi$ is a fringe distribution then $\height(\cT_n)\to\infty$ in probability. 
\end{enumeratea}
    
\end{theorem}

\section{Results}
\label{sec:res}
\chnr{We start by describing local weak convergence of the network model in Section \ref{sec:res-local-weak}. This gives convergence of the degree distribution to an explicit limit. Further the tail exponent of the limit degree distribution is shown to be related to a large deviation rate constant of the quasi-stationary distribution of a random walk associated to the exploration step distribution. Section \ref{sec:res-conden} describes asymptotics of the root and other fixed vertex degrees, in particular deriving necessary and sufficient conditions on $\vp$ for \emph{condensation}, namely the root obtaining a fixed density of the total number of edges. Section \ref{sec:res-height} derives asymptotics for the height, appropriately normalized,  in terms of extremal statistics of related branching random walks, showing the emergence of a phase transition.}  We conclude in Section \ref{sec:res-page-rank} with asymptotics and phase transitions for the asymptotic PageRank of fixed vertices as well as the limiting PageRank distribution.  

\subsection{Local weak convergence}
\label{sec:res-local-weak}
Recall the process $\set{\cT^*(t,\vp):t\geq 0}$ in Definition \ref{defn:cts-time}.
\begin{defn}[Fringe limit]
\label{def:fringe-limit}
     Let $\tau \sim \mathrm{Exp}(1)$ independent of the process $\cT^*$.  Let $\pi_{\vp}(\cdot)$ denote the distribution of $\cT^*(\tau,\vp)$, viewed as a random a.s. finite rooted tree in $\bT$. 
\end{defn}

\begin{ass}
\label{ass:mean}
Let $Z\sim \vp$. Assume $\E[Z] <\infty$. 
\end{ass}
\begin{theorem}[Fringe convergence]\label{thm:fringe}
\begin{enumeratea}
\item 	Under Assumption \ref{ass:mean}, the sequence of random trees $\set{\cT_n(\vp):n\geq 1}$ converges in probability, in the fringe sense, to $\pi_{\vp}(\cdot)$. Writing $D$ for the root degree of the random tree sampled using $\pi_{\vp}$, for every $k\geq 0$,
	\[\frac{1}{n+1}\sum_{v\in \cT_n(\vp)} \ind\set{\deg(v) = k+1} \probc \pr(D=k), \qquad \mbox{ as } n\to\infty. \]
\item $\pi_{\vp}(\cdot)$ is a fringe distribution, as in  Definition \ref{def:fringe}, if and only if $\E[Z]\leq 1$ where $Z\sim \vp$.
\item The root degree $D$ and the size of the limit fringe tree $|\cT^*(\tau,\vp)|$ satisfy:
\begin{enumeratei}
\item If $\E[Z] \leq 1$ then $\E[D] =1$ and $\E[|\cT^*(\tau,\vp)|] = \infty$,  
\item If $\E[Z] > 1$ then $\E[D] < 1$ and $\E[|\cT^*(\tau,\vp)|] < \infty$. 
\end{enumeratei}
\end{enumeratea}

\end{theorem}
\begin{rem}
Owing to the above Theorem, we will refer to the setting $\E[Z]\leq 1 $ as the \emph{fringe regime} whilst $\E[Z]>1$ will be referred to as the \emph{non-fringe regime}. 
\end{rem}
\noindent The next result follows from Theorem \ref{thm:fringe} and \chnr{results in} \cite{aldous-fringe} summarized in Theorem \ref{thm:aldous}.

\begin{corollary}[Convergence to limiting {\tt sin}-tree]\label{cor:sin-tree-convg}
	Assume for $Z\sim \vp$, $\E[Z]\leq 1$.
	\begin{enumeratea}
	    \item There exists a limiting infinite {\tt sin}-tree $\cT_\infty(\vp)$ such that $\set{\cT_n(\vp):n\geq 1}$ converges in probability in the extended fringe sense to $\cT_\infty(\vp)$. 
	    \item $\height(\cT_n)\probc  \infty$. 
	\end{enumeratea}

\end{corollary}

\begin{rem}
\begin{enumeratei}
\item By \cite{bhamidi2012spectra}, when $\E(Z) \leq 1$, local weak convergence above implies that even global functionals such as the spectral distribution of the adjacency matrix converge (in this case the limit spectral distribution can be shown to be non-random with an infinite set of atoms).   In Aldous' terminology \cite{aldous-fringe}, it would be interesting to derive a \emph{reduced Markov description} of this limit object. 
\item In the non-fringe case when $\E(Z) > 1$, whilst Theorem \ref{thm:fringe} shows convergence in probability in the fringe sense, \chnr{as described below in} Theorem \ref{thm:max-degree}(a), in this regime the degree of the root scales like $\sim (1-q_*)n$, and thus if one selects a vertex at random, with probability $\sim (1-q_*)$ the parent of this vertex is the root whose degree $\to\infty$. In particular, there is {\bf no convergence} in the extended fringe sense to a {\tt sin}-tree, which by definition has to be locally finite. 
\end{enumeratei}

\end{rem}

In order to derive quantitative bounds on explicit functionals such as the degree distribution and PageRank, we will need to make a few more assumptions on $\vp$. Recall that in Section \ref{sec:mod-def}, the case $p_0+p_1 =1$ corresponds to either the random recursive tree or affine preferential attachment, which have already been thoroughly analyzed in the literature. While the techniques described below recover many results for these models, we are mainly interested in settings not covered under these two models. Let $\{Z_i\}_{i\geq 1}$ be a collection of i.i.d. random variables following the law $\vp$. Consider the random walk, 
\begin{equation}
    \label{eqn:rw-def}
    S_n=S_0+\sum_{i=1}^n (Z_i-1), \qquad S_0 \in \mathbb{Z},\, n\geq 1.
\end{equation}
Define the probability generating function (pgf):
\beq\label{genfunc}
f(s):=\sum_{k=0}^\infty p_ks^k, \quad s\geq 0.
\eeq

\begin{ass}
\label{ass:pmf}
    Assume that $\vp$ satisfies the following:
\begin{inparaenuma}
\item $p_0\in (0,1)$, $p_0 + p_1 <1$. 
\item Assume that the random walk $\{S_n:n\geq 0\}$ in \eqref{eqn:rw-def} is aperiodic, i.e., $gcd\{ j: p_j>0\}=1$.
\end{inparaenuma}

\end{ass}
\begin{rem}
\tcr{Most of our results can be extended to the case $p_0+p_1=1$ by straightforward modifications of our proof techniques. See Remark \ref{rem_affine}.} 

\end{rem}

\begin{defn}
\label{def:R-q-s0}
\begin{enumeratea}
        \item Let $s_0$ be the unique positive root of $sf'(s)=f(s)$ if it exists, otherwise let $s_0$ be the radius of convergence of the power series $f(\cdot)$. Define $R:=\lim_{s\uparrow s_0} s/f(s)$.
        \item Let  $q_*$ be the unique positive solution to the fixed point equation $f(q)=q$. 
\end{enumeratea}
\end{defn}

Let us briefly describe probabilistic interpretations of these objects before stating our main results. Recall that $q_*$ denotes the extinction probability of a Galton-Watson branching process with offspring distribution $\vp$.   Next define
\beq
\label{eqn:hit-rw}
\barT_i:=\inf\{ n\geq 0: S_n=0| S_0=i\}
\eeq
 to be the hitting time of zero  of the random walk starting from state $i>0$. Then $q_*$ has the alternate interpretation $q_*= \pr(\barT_1< \infty)$. 
 Moreover, by \cite{daley1969quasi,pakes1973conditional}, under Assumptions \ref{ass:mean} and \ref{ass:pmf}, $R$ arises as the following limit and lies in the asserted interval:
 \begin{equation}
     \label{eqn:R-def}
   \lim_{n\to\infty} \left( \pr(n< \barT_1<\infty)\right)^{1/n} =   \frac{1}{R}\in (0,1].
 \end{equation}
 Standard results, connecting random walks to branching processes, imply $q_*=\pr(\barT_1<\infty) <1$ if and only if $\E[Z] >1$. The next lemma collects some classical facts about $R,s_0,q_*$ arising in the analysis of quasi-stationary distributions for random walks. We provide brief pointers to the literature for completeness. In the following \chnr{lemma}, let $r_f$ denote the radius of convergence of the pgf $f$.
\begin{lemma}
\label{lemma:prop-R-s0}
Suppose Assumptions \ref{ass:mean} and \ref{ass:pmf} hold.
\begin{enumeratea}
\item $g(s) := f(s)/s$ is strictly decreasing on $(0,s_0)$ and, if $s_0<r_f$, $g$ is strictly increasing on $(s_0,r_f)$. In particular, $\inf_{s\in (0,1)} f(s)/s \geq 1/R$. 
\item For $\E[Z]<1$, if $f(s)$ is analytic at $s=1$, then $R\in (1,1/\E[Z]), s_0\in(1,\infty)$, otherwise $R=s_0=1$.
\item For $\E[Z]=1$, $R=s_0=1$.
\item For $\E[Z]> 1$, $R>1$ and $q_*<s_0<1$.

\end{enumeratea}
\end{lemma}

\begin{proof}
Part (a) follows from the observation that $g'(s) = (sf'(s) - f(s))/s^2$, whose numerator is strictly increasing on $(0,r_f)$ as $p_0 + p_1<1$, and negative at $s=0$ as $p_0>0$. For part (b), when $f(s)$ is analytic at $s=1$, note that $sf'(s) - f(s)<0$ at $s=1$ as $\E(Z)<1$, which implies $s_0>1$. Moreover, if $s_0=\infty$, $r_f=\infty$ and $g(s)$ is strictly decreasing on $(0,\infty)$, which is a contradiction as $\lim_{s\rightarrow \infty}g(s) = \infty$ (as $p_0 + p_1<1$). Thus, $s_0 \in (1,\infty)$. Further, this also implies $f(s_0)<\infty$ and hence, using part (a) and the definition of $s_0$,
$1 < R=s_0/f(s_0) \le 1/f'(s_0) < 1/\E[Z]$. For the remaining assertions, see \cite[Lemma 1]{pakes1973conditional} and the Remark following it. 
 \end{proof}
 Note that Lemma \ref{lemma:prop-R-s0} implies that, under Assumptions \ref{ass:mean} and \ref{ass:pmf}, $s_0,f(s_0)$ are both finite and $R=s_0/f(s_0)$. Now let \chnr{the random variable} $D$ be as in Theorem \ref{thm:fringe}, following the law of the limiting degree distribution of the discrete tree network $\{\cT_n(\vp)\}_{n\geq 1}$.

\begin{theorem}[Tail of limit Degree distribution]\label{thm:deg-dist}
Under Assumptions \ref{ass:mean} and \ref{ass:pmf}:
\begin{enumeratei}
\item {\bf Fringe regime:} When $\E[Z]\leq 1$, with $R$ as in Definition \ref{def:R-q-s0},
\beq
\lim_{k\to\infty} \frac{\log \pr(D\geq k)}{\log k}=-R.
\eeq
\item {\bf Non-fringe regime:} When $\E[Z]>1$, 
\beq
-R\leq \liminf_{k\to\infty} \frac{\log \pr(D\geq k)}{\log k}\leq \limsup_{k\to\infty} \frac{\log \pr(D\geq k)}{\log k}\leq -\left(R\wedge \frac{\log q_*}{\log s_0}\right).
\eeq
\end{enumeratei}

\end{theorem}
Although the upper and lower bounds in part (ii) above are different, they can be checked to be equal in several cases even in the non-fringe regime. See Remark \ref{sharp} for a discussion.

\subsection{Condensation and fixed vertex degree asymptotics}
\label{sec:res-conden}

The next result describes asymptotics for the root degree. In particular, part (a) shows that in the non-fringe regime, there is a condensation phenomenon at the root and the root neighbors asymptotically comprise a positive limiting density  of all the vertices in the tree. 
\begin{theorem}[Root degree asymptotics]\label{thm:max-degree}
	Let $\deg(v_0, n)$ denote the degree of the root in $\cT_n(\vp)$.  Under Assumptions \ref{ass:mean} and \ref{ass:pmf}:
	\begin{enumeratea}
	    \item {\bf Non-Fringe regime:} Assume $\E[Z]> 1$. Then 
	    \[\frac{\deg(v_0, n)}{n} \stackrel{\prob}{\longrightarrow} 1-q_*>0,\]
	    where $q_*$ is defined in Definition \ref{def:R-q-s0}(b).
	    \item {\bf Fringe regime:}  Assume $\E[Z]\leq 1$. Then for any $\delta>0$,  as $n\to\infty$,
\begin{align}\label{root_order_fringe}
\frac{\deg(v_0, n)}{n^{\frac{1}{R}-\delta}} &\overset{a.s.}{\longrightarrow} \infty, \quad \text{ and }\quad \frac{\deg(v_0, n)}{n^{\frac{1}{R}}(\log n)^{1+\delta}} \overset{a.s.}{\longrightarrow} 0.
\end{align} 
	\end{enumeratea}
\end{theorem}

The next result describes the degree evolution of a fixed {\bf non-root} vertex. In particular in the non-fringe regime there is a marked difference between the evolution of the degree of the root $v_0$ and any non-root vertex.  

\begin{theorem}[Fixed vertex degree asymptotics]
    \label{thm:fixed-vertex}
    Fix $k\geq 1$ and let $\deg(v_k, n)$ denote the degree of vertex $v_k$ in $\cT_n(\vp)$. Then under Assumptions \ref{ass:mean} and \ref{ass:pmf}, for any $\delta> 0$,  
\end{theorem}
\begin{align}\label{degreeofv}
\frac{\deg(v_k,n)}{n^{\frac{1}{R}-\delta}} &\overset{a.s.}{\longrightarrow} \infty, \quad \text{ and }\quad \frac{\deg(v_k,n)}{n^{\frac{1}{R}}(\log n)^{1+\delta}} \overset{a.s.}{\longrightarrow} 0.
\end{align}

\begin{rem}
When $\vp$ is Geometric($p$) for $p\in (0,1)$ so that $\E[Z] = (1-p)/p$, the model was was first rigorously analyzed in \cite{chebolu2008pagerank} using combinatorial recursions.  The authors observed a phase transition for the expected root degree at $p=1/2$ (precisely as one transitions from $\E[Z] <1$ to $\E[Z] \geq 1$). In brief, \cite{chebolu2008pagerank} showed that the expected degree of the root in this special case satisfies $\E(\deg(v_0, \cT_n)) = \Omega(n (\log n)^{k})$ if $p\leq 1/2$ whilst $\E(\deg(v_0, \cT_n)) = O(n^{4pq}(\log n)^{k'})$ if $p > 1/2$, for some $k,k' \in \bZ$. One can check that in this case, $R = 1/(4pq)$, thus this paper also clarifies the reason for the mysterious constant $4pq$, in terms of a large deviations exponent of the hitting time of zero by the associated random walk.

\end{rem}

\subsection{Height asymptotics}
\label{sec:res-height}
\color{black}
Recall that $f(\cdot)$ denotes the pgf of $\vp$. Define
$$
\kappa(s) := \frac{ f(s)}{s\log (1/s)}, \ s \in (0,1).
$$
It is shown in Lemma \ref{kappa_properties} that the infimum of $s \mapsto \kappa(s)$ is attained at a unique point in $(0,s_0]$.
 \begin{defn}\label{kappa}
 Define,
\begin{equation}
    \label{eqn:kappa-ht}
    \kappa_0 :=\inf_{s\in(0,1)} \frac{ f(s)}{s\log (1/s)}
\end{equation}
Let $s^* \in (0,1)$ be the point where the infimum of $\frac{ f(s)}{s\log (1/s)}$ is attained, i.e., $\kappa_0 =\frac{ f(s^*)}{s^*\log (1/s^*)}$.
\end{defn}
The following theorem gives height asymptotics for our model. Interestingly, we observe a phase transition in the limiting behavior of rescaled heights in the non-fringe regime.
\begin{theorem}[Height asymptotics]\label{thm:height}
	Let $\cH_n$ denote the height of $\cT_n(\vp)$. Then, under Assumptions \ref{ass:mean} and \ref{ass:pmf}: 
	\begin{enumeratei}
	 \item {\bf Fringe regime:} When $\E[Z]\leq 1$, as $n\to\infty$,
	 \[\frac{\cH_n}{\log{n}}\overset{a.s.}{\longrightarrow} \kappa_0. \]
	 \item {\bf Non-Fringe regime:} When $\E[Z]> 1$, as $n\to\infty$,
	 $$
	 \frac{\cH_n}{\log{n}}\overset{a.s.}{\longrightarrow}
	 \begin{cases}
	  \kappa_0 &\quad \text{ if } s^*\in (0,q_*],\\
	  \frac{1}{\log(1/q_*)} &\quad \text{ if } s^*\in (q_*,s_0],
	 \end{cases}
	 $$
	 where $q_*$ is defined as in Definition \ref{def:R-q-s0} (b). 
	\end{enumeratei}
\end{theorem}
\color{black}
Three remarks are in order.


\begin{rem}[Non-triviality in the non-fringe regime]
Theorem \ref{thm:max-degree}(a) implies that in the non-fringe regime, there are $\Theta(n)$ vertices within distance one of the root whp so it is not obvious that in this case the height should diverge. Thus the result above on the height shows, that even in this case, the height scales like $\log{n}$ with an appropriate limit constant.  
\end{rem}
 \color{black}
\begin{rem}[{\bf Probabilistic interpretation of limit and phase transition}]
The limiting rescaled height and associated phase transition can be probabilistically understood via branching random walks \cite{biggins1995growth,biggins1997fast}.
\begin{defn}
\label{defn:brw}
Fix pmf $\vp$. Consider a branching random walk with individuals being born into the population in continuous time, and with spatial locations on $\bZ$,  starting with one individual at location zero with dynamics:
\begin{enumeratea}
\item Each vertex gives rise to offspring according to a rate one Poisson process.
\item Writing $\rho_v$ for the location of particle $v$, if $v$ is born to $u$ then $\rho_v = \rho_u + \zeta_{uv}$, where $\zeta_{uv}  \stackrel{d}{=} (1-Z)$, independent across parent offspring connections and times of birth, where $Z\sim \vp$.
\end{enumeratea} 
Write $\set{\BRW(t):t\geq 0}$, for the corresponding process keeping track of genealogical structure and locations.
\end{defn}
 Let $B(t)$ denote the location of the rightmost particle at time $t$. Then, it turns out, using  \cite{biggins1995growth} (see Section \ref{sec:BRW-def}), in our setting,  $\lim_{t\to\infty} B(t)/t \convas \kappa_0$. Hence, the rescaled height asymptotics agrees with that of $B(\cdot)$ when $\E[Z] \le 1$, or $\E[Z] >1$ and $s^* \in (0,q_*]$.
 
However, when $\E[Z]>1$ and $s^* \in (q_*,s_0]$, there is a competition between the reproduction rate of the root and the upper tail large deviations behavior of $B(\cdot)$, which is obtained in Lemma \ref{BRW_deviation}. An inspection of the proof reveals that, for large $t$, the height of $\cT(t,\vp)$ has the same asymptotics as the maximum of the heights of the subtrees rooted at the children of the root that are born in the time interval $[\frac{1-\ep}{(1-f'(q_*))\log(1/q_*)}, \frac{1+\ep}{(1-f'(q_*))\log(1/q_*)}]$ for small $\ep>0$. This phenomenon manifests itself through the phase transition observed in Theorem \ref{thm:height}.
\end{rem}

\begin{rem}[{\bf Explicit height computation for PageRank driven networks}]\label{pdnht}
When $\vp$ is Geometric($p$), the pgf is given by $f(s)=\frac{p}{1-qs}, \, s \in [0,1/q)$. 
Asymptotics for the height in this model were previously addressed in \cite{mehrabian2016sa} using quite different techniques. For an explicit constant $\tilde p\approx .206$, they show that for $p \in [\tilde p,1)$, $\cH_n/\log{n}$ converges to a limit constant whilst for $p\in (0,\tilde p)$, there exist constants $c_L(p) < c_U(p)$ such that whp for any $\eps> 0$, $\cH_n \in [(c_L(p) -\eps)\log{n}, (c_U(p)+\eps)\log{n}] $ as $n\to\infty$. Our result shows that, contrary to what is conjectured in \cite{mehrabian2016sa}, there is indeed a phase transition in the height asymptotics at $p = \tilde p$.
More precisely, the minimizer $s^* \in (0,1)$ in \eqref{eqn:kappa-ht} can be seen to be the unique solution to the equation 
\[(1-p)s=\frac{1+\log s}{1+2\log s}, \qquad  s\in (0,1).\]
Define $u^*$ via the relation $s^* =e^{-1/u^*}$. Then,
\[\kappa_0= pe^{1/u^*} u^* (2-u^*),\]
which matches the expression for $c_L(p)$ in \cite[Theorem 2]{mehrabian2016sa}. Moreover, for $p\in (0,1/2)$, $q_* = p/(1-p)$ and hence
\[\frac{1}{\log(1/q_*)} = \left(\log\left(\frac{1-p}{p}\right)\right)^{-1},\]
which matches $c_U(p)$ in \cite[Theorem 2]{mehrabian2016sa} for $p\in (0,\tilde p)$. Finally, the value of $p$ where the height phase transition happens in Theorem \ref{thm:height} is seen to be the unique such value which gives $s^*=q_*$, or equivalently $\kappa'(q_*)=0$. This characterizes the value as the unique solution to
$$
\log\left(\frac{1-p}{p}\right) = \frac{1-p}{1-2p}, \ p \in (0,1/2),
$$
which agrees with $\tilde p$ obtained in \cite[Theorem 2]{mehrabian2016sa}.

Our proofs thus elucidate the connections between limit constants in \cite{mehrabian2016sa} obtained through subtle combinatorial analysis and extremal statistics in branching random walks \cite{biggins1995growth}.
\end{rem}
\color{black}

\subsection{PageRank asymptotics}
\label{sec:res-page-rank}
Recall the PageRank scores $\set{\fR_{v,c}(n): v\in \cT_n(\vp)}$ in Definition \ref{def:page-rank}. For any $v\in \cT_n(\vp)$, let $P_l(v,n)$ denote the number of \emph{directed} paths of length $l$ that end at $v$ in $\cT_n(\vp)$. Since $\cT_n(\vp)$ is a directed tree, it is easy to check that the PageRank scores have the explicit formulae for any vertex $v$, 
\beq
\label{eqn:non-root-page-rank}
\fR_{v,c}(n)=\frac{(1-c)}{n}\left(1 + \sum_{l=1}^\infty c^l P_l(v,n)\right).
\eeq

For the sequel, it will be easier to formulate results in terms of the \emph{graph normalized} PageRank scores \cite{garavaglia2020local} $\set{R_{v,c}(n):v\in \cT_n(\vp)} = \set{n\fR_{v,c}(n): v\in \cT_n(\vp)}$. The first result shows a non-trivial phase transition of the PageRank scores for fixed vertices in the fringe regime. This phase transition carries over to the empirical distribution of PageRank scores which will conclude this Section. 

\begin{theorem}[PageRank asymptotics for fixed vertices]
\label{thm:page-rank-fixed}
Fix vertex $v_k$ and damping factor $c\in (0,1)$. Then under Assumptions \ref{ass:mean} and \ref{ass:pmf}, 
\begin{enumeratea}
    \item {\bf Non-fringe regime:} when $\E[Z] > 1$, 
 \begin{enumeratei}
\item When $k=0$ (root PageRank), $R_{v_0,c}(n)/n$ is bounded away from zero in probability: for any $\delta>0$, there exists $\ep>0$ such that
$$
\liminf_{n \rightarrow \infty}\pr\left(\frac{R_{v_0,c}(n)}{n} \ge c(1-c)(1-q_*) - \ep\right)\ge 1-\delta.
$$
\item For any $k \ge 1$, as $n\to\infty$,
\begin{align}\label{profv}
\frac{R_{v_k,c}(n)}{n^{\frac{1}{R}-\delta}} &\overset{a.s.}{\longrightarrow} \infty, \quad \text{ and }\quad \frac{R_{v_k,c}(n)}{n^{\frac{1}{R}}(\log n)^{1+\delta}} \overset{a.s.}{\longrightarrow} 0.
\end{align} 
\end{enumeratei}   
\item {\bf Fringe regime:} When $\E[Z]\leq 1$, 
\begin{enumeratei}
\item Fix any $c\in (0,s_0^{-1}]$ with $c<1$. Then for any $\delta>0$, and any vertex $v_k$, $k \ge 0$,
as $n\to\infty$,
\begin{align}\label{profv}
\frac{R_{v_k,c}(n)}{n^{\frac{1}{R}-\delta}} &\overset{a.s.}{\longrightarrow} \infty, \quad \text{ and }\quad \frac{R_{v_k,c}(n)}{n^{\frac{1}{R}}(\log n)^{1+\delta}} \overset{a.s.}{\longrightarrow} 0.
\end{align} 
\item Suppose $s_0>1$. Fix any $c\in( s_0^{-1},1)$ and $v_k$, $k \ge 0$. Then there exists a non-negative random variable $W_{k,c}$ with $\pr(W_{k,c} >0) >0$ such that, 
\[\frac{R_{v_k,c}(n)}{n^{cf(1/c)}}\convas W_{k,c}, \qquad \mbox{ as } n\to\infty. \]
\end{enumeratei} 
\end{enumeratea}
\end{theorem}

Next define the empirical distribution of normalized PageRank scores, 
$$\hat{\mu}_{n, \PR} := n^{-1} \sum_{v\in \cT_n(\vp)} \delta\set{R_{v,c}(n)}.$$   
General results on the implication of local weak convergence of sparse graphs on the convergence of the empirical distribution of PageRank scores was derived in \cite{garavaglia2020local,banerjee2021pagerank}. In particular, the local weak convergence in Theorem \ref{thm:fringe} coupled with \cite{garavaglia2020local,banerjee2021pagerank} leads to the following result. Recall the finite rooted random tree $\cT^*(\tau,\vp)$ from Definition \ref{def:fringe-limit} and let $P_l(\emptyset)$ denote the number of directed paths of length $l$ that end at the root denoted as $\emptyset$. Define the \emph{normalized} PageRank score at $\emptyset$ as, 
\begin{equation}\label{lwlpr}
    \cR_{\emptyset, c}(\infty) = (1-c)\left(1 + \sum_{l=1}^{\infty} c^l P_l(\emptyset)\right).
\end{equation}
\begin{corollary}[PageRank asymptotics]\label{cor:page-rank}
    Under Assumption \ref{ass:mean}, the random variable $\cR_{\emptyset, c}(\infty)$ is finite a.s. Further for every continuity point $r$ of the distribution of $\cR_{\emptyset, c}(\infty)$, 
    \[n^{-1}\sum_{v\in \cT_n(\vp)} \ind\set{R_{v,c}(n) > r} \probc \pr(\cR_{\emptyset, c}(\infty) > r).  \]
\end{corollary}
   
   The next result which, in particular, displays a qualitative phase transition of the tail exponent of the limiting empirical distribution of the PageRank scores in the fringe regime (Theorem \ref{thm:pr2}(b)), has not been observed before in the literature and continues the vein of results displayed in Theorem \ref{thm:page-rank-fixed}(b).
   
 \begin{theorem}[Tail behavior of PageRank distribution]\label{thm:pr2}
 Under Assumptions \ref{ass:mean} and \ref{ass:pmf},
 \begin{enumeratea}
 \item {\bf Non-fringe regime:} When $\E[Z]>1$, for $c\in (0,1)$,
$$-R\leq \liminf_{r\to\infty} \frac{ \log(\pr(\cR_{\emptyset, c}(\infty)\geq r))}{\log r}\leq  \limsup_{r\to\infty} \frac{ \log(\pr(\cR_{\emptyset, c}(\infty)\geq r))}{\log r}\leq -\left(R\wedge \frac{\log q_*}{\log s_0}\right).
$$
where as before,  $q_*<1$ is the solution to $f(q)=q$.
\item {\bf Fringe regime:} When $\E[Z]\leq 1$,
$$\lim_{r\to\infty} \frac{ \log(\pr(\cR_{\emptyset, c}(\infty)\geq r))}{\log r}=
\begin{cases}
-R &\text{ for $c\in(0,s_0^{-1}]$ with $c<1$},\\
-\frac{1}{cf(1/c)} & \text{ for $c\in (s_0^{-1},1)$, provided $s_0>1$}.
\end{cases}
$$
\end{enumeratea}
 \end{theorem}
 \begin{rem}[\textbf{PageRank and the power-law hypothesis}]
 Since its introduction by Brin and Page \cite{page1999pagerank}, PageRank has been largely and successfully used to identify influential nodes in a variety of network models \cite{Haveliwala2002personalization, Gyongyi04, Andersen06}. Although PageRank `looks beyond' degrees and captures more intricate local geometry around vertices, it can be computed efficiently in a distributed fashion. The well known power-law hypothesis conjectures that for real world networks with power-law (in)degree distribution, the PageRank also has a power-law distribution with the same exponent as the degree. This has been shown to hold in several \emph{static} network models like the directed configuration model \cite{chen2017generalized, olvera2019pagerank} and the inhomogeneous random digraph \cite{lee2020pagerank, olvera2019pagerank}.
 Recently, \cite{banerjee2021pagerank} showed that the power-law hypothesis is false for the affine preferential attachment model: the PageRank distribution has a strictly heavier tail than the degrees. This suggests that for \emph{dynamic} networks (evolving over time), the PageRank captures strictly more information than the empirical degree structure even at the level of large deviations.
 
 In Theorem \ref{thm:pr2}, we show for the first time a \emph{phase transition} for the limiting PageRank distribution in a random network model: in the fringe regime, the power law hypothesis holds for damping factor $c \in (0,s_0^{-1}]$, but the PageRank tail becomes heavier for larger $c$. Intuitively, when $c$ crosses a threshold, the PageRank score incorporates information from a large enough local neighborhood of each vertex so as to distinguish the extremal behavior of the limiting PageRank and degree distributions.  Identifying the class of network models for which such phase transitions occur should further quantify the efficacy and limits of the power-law hypothesis.
 \end{rem}

 \begin{rem}\label{sharp}
 Although we explicitly characterize the tail exponent for the limiting degree and PageRank distribution in the fringe regime (see Theorem \ref{thm:deg-dist}(i) and Theorem \ref{thm:pr2}(b)), only upper and lower bounds are established on this exponent in the non-fringe regime (see Theorem \ref{thm:deg-dist}(ii) and Theorem \ref{thm:pr2}(a)). However, computation of the quantities $R, q_*,s_0$ in the specific models below suggests that the upper and lower bounds actually match even in the non-fringe regime unless $\E[Z]$ is much greater than $1$. 
 \begin{enumeratei}
 \item \emph{PageRank driven preferential attachment}: This model was described in example (c) of Section \ref{sec:mod-def} (see also Remark \ref{pdnht}). In this case, $p_k = p(1-p)^k, k \ge 0$, and $\E[Z] = \frac{1}{p}-1$. Thus, the non-fringe regime corresponds to $p \in [0,1/2)$. Writing $q=1-p$, the pgf is $f(s)= p/(1-qs), \, s\in [0,1/q),$ which gives $R = 1/(4pq)$, $q_* = (p/q) \wedge 1$ and $s_0=1/(2q)$. Hence, the upper bound in the tail exponent in the non-fringe regime is given by
 $$
 -\left(R\wedge \frac{\log q_*}{\log s_0}\right) = -\left(\frac{1}{4pq}\wedge \frac{\log (q/p)}{\log (2q)}\right) = \begin{cases}
-\frac{1}{4pq} &\text{ for $p \in [p_0',1/2)$},\\
-\frac{\log (q/p)}{\log (2q)} & \text{ for $p \in (0,p_0')$}.
\end{cases},
 $$
 where $p_0' \approx 0.0616$ is obtained using Mathematica. In particular, the tail exponent for the limiting degree distribution is exactly $-1/(4pq)$ for all $p\in [p_0',1)$. For the limiting PageRank, as described in Theorem \ref{thm:pr2}(b), we see a phase transition in the fringe regime as the damping factor $c$ crosses $s_0^{-1}$. In the non-fringe regime, for $p \in [p_0',1/2)$, the tail exponent equals $-1/(4pq)$ for all values of the damping factor $c \in (0,1)$.
 
 \item \emph{Simple random walk driven attachment}: Here, we take $p_0=p=1-p_2$. Thus, at each attachment, the distance of the new vertex from the root behaves like a (biased) simple random walk with increment $\pm1$. $\E[Z] = 2(1-p)$, and thus the non-fringe regime corresponds to $p \in [0,1/2)$. In this case, writing $q=1-p$, the pgf is $f(s) = p + qs^2, \, s \ge 0,$ which gives $R = 1/(2\sqrt{pq})$, $q_*= p/q \wedge 1$ and $s_0= \sqrt{p/q}$. The upper bound in the tail exponent in the non-fringe regime is given in this case by
 $
 -\left(R\wedge \frac{\log q_*}{\log s_0}\right) = - \left(\frac{1}{2\sqrt{pq}} \wedge 2\right).
 $
 In particular, the tail exponent is exactly characterized for $p \in (\frac{1}{2}-\frac{\sqrt{3}}{4},1)$. 
 \end{enumeratei}
 The disparity between the upper and lower bounds in the non-fringe regime may appear to be an artifact of our proof techniques (see Theorem \ref{updegree}(iii) where this discrepancy appears from estimating moments of a functional of the process). However, preliminary calculations using the many-to-few formula for branching random walks \cite{harris2017many} (which gives more refined estimates) suggest that this discrepancy might be qualitative in nature due to certain rare events that affect the tail exponent when $\E[Z]\gg 1$. We will investigate this in future work.
 \end{rem}
\begin{rem}[\textbf{Affine preferential attachment}]\label{rem_affine}
As discussed before, affine preferential attachment with attachment function $f_{att}(k) = k + (1- 2p)/p$, $p \in (0,1)$, is a special case of our model, corresponding to $p_0 = 1- p, p_1 = p$. One can easily verify that in this case $s_0=\infty$ and $R=1/p$.
Although we assume $p_0 + p_1<1$, most of our proof techniques can be extended in a straightforward manner to the case $p_0+p_1=1$. Extrapolation of our results to the affine case recovers several known results which we now describe. In this case, since $\E[Z] <1$, we are always in the fringe regime. Theorem \ref{thm:deg-dist}(i) implies that the limiting degree distribution has a power-law with exponent $R=1/p$. This is well known (see e.g. \cites{bollobas2001degree,rudas2007random}).
The limiting PageRank distribution, with damping factor $c \in (0,1)$, turns out to have a heavier tail than the degrees, with exponent $1/((1-p)c + p)$, which was recently shown in \cite{banerjee2021pagerank}.
\end{rem}

\section{Proofs: Technical foundations}
\label{sec:proofs}
\subsection{Roadmap of the proofs}
The goal of this Section is to build the technical underpinnings for the proofs of the main results, as well as elucidate the connections between the functionals of the model and corresponding core areas in probability. 
\begin{enumeratei}
\item Section \ref{sec:prelim-est-embed}  deals with properties of the continuous time embeddings in Definition \ref{defn:cts-time}; these are then used in Section \ref{sec:fringe-proof} to prove Theorem \ref{thm:fringe}. 
\item Section \ref{sec:BRW-def} describes stochastic orderings between the height and extremal displacements of associated (upper and lower bounding) branching random walks. \tcr{Further, large deviations estimates are derived for branching random walks.} These results are then used in Section \ref{sec:proof-height} to prove Theorem \ref{thm:height}. 
\item While Section \ref{sec:prelim-est-embed} and \ref{sec:BRW-def} deal with direct embeddings of the process, the next few subsections describe ``non-obvious'' embeddings. Section \ref{sec:quasi} derives connections between the distance profile (in continuous time) and functionals of the quasi-stationary distribution of the random walk in \eqref{eqn:rw-def} through two infinite dimensional matrices in \eqref{eqn:matrix-def}, spectral properties of which result in the key role of the constants $R, s_0, q_*$ in the main results.
\item Analysis of truncations of the height profile leads to finite dimensional urn models and their corresponding Athreya-Karlin embeddings in finite dimensional multitype branching processes in Section \ref{sec:urn-model-def}; asymptotics of these processes, in particular as the level of truncation $K\to\infty$ needs a careful analysis of the Perron-Frobenius  eigenvalue, since the corresponding limit infinite dimensional operator is non-compact; this analysis culminates in Proposition \ref{Perroneigen}. These results form the core ingredients in obtaining lower bounds connected to evaluating power-law exponents of the degree (Theorem \ref{thm:deg-dist}) in Section \ref{sec:proof-deg-dist} and PageRank distribution (Theorem \ref{thm:pr2}) in Section \ref{sec:proof-pagerank} respectively. They are also used for lower bounds in the analysis of degree and PageRank of fixed (non-root) vertices (Theorem \ref{thm:fixed-vertex} in Section \ref{sec:proof-fixed-vertex} and Theorem \ref{thm:page-rank-fixed} in Section \ref{sec:proof-pagerank}).
\item Asymptotics of the degree and PageRank of the root necessitate the construction and analysis of an infinite dimensional multitype branching process (MTBP) in Section \ref{sec:inf-dim-mtbp}. These results play a central role in the proof of Theorem \ref{thm:max-degree} in Section \ref{sec:proof-fixed-vertex} and Theorem \ref{thm:page-rank-fixed} in Section \ref{sec:proof-pagerank}. Technical properties related to $\alpha$-recurrence and transience of kernels arising in the analysis of the MTBP are proven in Appendix \ref{apprec}.   
 \end{enumeratei}

\subsection{Size estimates on the continuous time embedding}
\label{sec:prelim-est-embed}
Recall, from Definition \ref{defn:cts-time}, the construction of the tree process in continuous time. Let $n(t) = |\cT(t,\vp)|$.

\begin{defn}[Rate $\nu$ Yule process]
\label{def:yule-process}
    Fix $\nu > 0$. A rate $\nu$ Yule process is a pure birth process $\set{Y_\nu(t):t\geq 0}$ with $Y_\nu(0)=k \in \mathbb{N}$ and where the rate of birth of new individuals is equal to $\nu$ times the size of the current population. More precisely, $\pr(Y_\nu(t+dt) - Y_\nu(t) = 1|\cF(t)):= \nu Y_\nu(t) dt + o(dt)$ and $\pr(Y_\nu(t+dt) - Y_\nu(t) \ge 2|\cF(t)):= o(dt),$ where $\set{\cF(t):t\geq 0}$ is the natural filtration of the process. Write $\set{\yu(t):t\geq 0}$ for the corresponding forest valued (tree valued if $k=1$) rate one process that keeps track of the genealogy of the process. 
\end{defn}

The following is a standard property of the Yule process. 

\begin{lem}[{\cite[Section 2.5]{norris-mc-book}}]
\label{lem:yule-prop}
Fix $t >0$ and rate $\nu > 0$ and assume $Y_{\nu}(0)=1$. Then $Y_\nu(t)$ has a Geometric distribution with parameter $p=e^{-\nu t}$. Precisely, $\text{ }\pr(Y_\nu(t) = k) = e^{-\nu t}(1-e^{-\nu t})^{k-1}, k\geq 1.$ The process $\set{Y_\nu(t)\exp(-\nu t):t\geq 0}$ is an $\bL^2$ bounded martingale and thus $\exists~ W>0 $ such that  $Y_\nu(t)\exp(-\nu t)\convas W$. Further $W \sim \mathrm{Exp}(1)$. 
\end{lem} 
This property leads directly to the next two results. 

\begin{lem}
\label{lem:yule-asymp}
Let $n(t) = |\cT(t,\vp)|$. Then $\set{n(t):t\geq 0}$ has the same distribution as a rate one  Yule process started with one individual at time zero.  By Lemma \ref{lem:yule-prop}, 
\begin{enumeratea}
\item $\set{e^{-t}n(t): t\geq 0}$ is an $\bL^2$ bounded martingale, $e^{-t} n(t) \stackrel{a.s., \bL^2}{\longrightarrow} W$ with $W\sim \mathrm{Exp}(1)$.  
\item Defining  $T_n = \inf\set{t\geq 0: n(t) = n+1}$, then $T_n -\log{n} \convas -\log{W}$ as $n\to\infty$. 
\end{enumeratea}
\end{lem}

\begin{lemma}\label{mart_power}
Let $0\leq s<t$. We have
\begin{align*} 
\E[ (e^{-t}n(t))^2|\FF_s]&\leq  (e^{-s}n(s))^2+e^{-2s}n(s),\\
\E\left[ (e^{-t}n(t))^3|\FF_s\right]&\leq 8(e^{-s}n(s))^3.
\end{align*}
\end{lemma}
\begin{proof}
Applying the generator $\LL$ of the Yule process on $(e^{-t}n(t))^2$ we get,  
\begin{align*}
\LL (e^{-t}n(t))^2&= e^{-2t} n(t)\left( (n(t)+1)^2-n(t)^2\right)-2e^{-2t}n^2(t)=e^{-2t}n(t).\\
\end{align*}
It follows that $\set{\bar M_2(t):t\geq 0}$ defined next is a martingale:
$$\bar M_2(t):=(e^{-t}n(t))^2-\int_0^t e^{-2u}n(u)du$$
Thus
\begin{align*}
\E[ (e^{-t}n(t))^2|\FF_s]&=(e^{-s}n(s))^2+\int_s^t e^{-2u} \E[n(u)|\FF_s]du\\
&\leq (e^{-s}n(s))^2+e^{-2s}n(s).
\end{align*}
The second assertion of the Lemma follows the same reasoning, starting with the application of the generator on $(e^{-t}n(t))^3$. We omit the details. 

\end{proof}

\begin{lemma}\label{diam_tail}
Consider a rate one Yule process $\set{\yu(t):t\geq 0}$ started with a single individual at $t=0$.  Let $\height(t)$ denote the corresponding height (maximal distance from the root) of the corresponding genealogical tree at time $t$. For any $\beta\geq 1$, we have $\E(\beta^{\height(t)}) \le 2e^{2\beta t} < \infty$.
\end{lemma}
\begin{proof}
Let $\cZ_n(t)$ denote the number of $n$-th generation individuals born before time $t$ and $B_n$ denote the time of the first birth  in the $n$-th generation. Then for $\theta>0$,
\begin{align}
\E(\beta^{\height(t)})&\leq \sum_{n=0}^\infty \beta^n \pr(\height(t)\geq n)=\sum_{n=0}^\infty \beta^n \pr(B_n\leq t)
\leq \sum_{n=0}^\infty \beta^n e^{\theta t} \E[ e^{-\theta B_n}]. \label{eqn:711}
\end{align}
Theorem 1 in \cite{kingman1975first} implies that $ \E[ e^{-\theta B_n}]\leq \psi(\theta)^n$ where
$$\psi(\theta)=\int_0^\infty \theta e^{-\theta t} \E[\cZ_1(t)] dt=\frac{1}{\theta}.$$
Thus $\E[ e^{-\theta B_n}]\leq \theta^{-n}$. Using this in \eqref{eqn:711} with  $\theta = 2\beta$ completes the proof. 

\end{proof}

\subsection{Branching random walks}
\label{sec:BRW-def}

Recall, the definition of $\BRW$, in Definition \ref{defn:brw}. Consider the following variations in step (b) of the dynamics:

\begin{enumeratea}
\item[(b)$^\prime$] If the prospective location of a new particle is at zero or below, it is ``reflected'' to location one. Then in terms of locations describing graph distance to the root, this is precisely the evolution of distances in $\cT$. Thus the height $\cH_{\cT(t)}$ is precisely the location of the rightmost particle.  
\item[(b)$^{\prime\prime}$] If the prospective location of a new particle is zero or below it is killed (removed from the system). This gives the distance process in $\cT^*$. As before the height $\cH_{\cT^*(t)}$ is precisely the location of the rightmost particle in this process. 
\end{enumeratea}

From the description of the dynamics, the following is obvious. 
\begin{lemma}
\label{lem:ht-domination}
Let $B(t)$ be the rightmost particle in $\BRW(t)$. One can couple $\cT, \cT^*, \BRW$ on a common probability space such that for all $t\geq 0$, $\cH_{\cT^*(t)} \leq B(t) \leq \cH_{\cT(t)} $. 
\end{lemma}

For $\BRW$, the offspring process for each individual is a rate one Poisson process and the corresponding branching process (without location information) is a Yule process. In particular, the offspring process is non-lattice and underlying branching process is supercritical. Following \cite{biggins1995growth}, write $\mu(dz, d\tau) = \sum_{k=0}^\infty p_k \delta_{1-k}(dz) \otimes d\tau$ for the mean intensity measure of the walk.  
Define for $\theta \in \bR, \phi \geq 0$ the following functionals:
$$m(\theta,\phi)=\int e^{-\theta z-\phi \tau} \mu (dz,d\tau)=\int_0^\infty \sum_{k=0}^\infty e^{-\theta(1-k)} e^{-\phi\tau} p_k d\tau=\frac{f(e^\theta)}{e^{\theta} \phi}.$$
$$\alpha(\theta)=\inf\{ \phi: m(\theta,\phi)\leq 1\}=\frac{f(e^\theta)}{e^{\theta}}.$$
\begin{equation}
    \label{eqn:alphast}
    \alpha^*(x)=\inf_{\theta<0}\{x\theta+\alpha(\theta)\}=\inf_{s\in(0,1)} \{ x\log s+\frac{f(s)}{s}\}.
\end{equation}

\color{black}
We will be using the results in \cite{biggins1995growth} that concern a very general branching random walk model, where particles are allowed to move after birth. In comparison, a particle in our branching random walk $\BRW$ performs no further movement beyond the initial displacement from its parent at birth. It would be straightforward for the curious reader to verify that $\BRW$ satisfies the mild assumptions in  \cite{biggins1995growth}, so we will refrain from repeating those detailed assumptions here.

We will rephrase Theorem 4 in \cite{biggins1995growth} into the following proposition, which will be used to prove large deviations results for $\BRW$ later.
\begin{prop}\label{Thm4}
Let $N_t[xt,\infty)$ denote the number of particles in $\BRW(t)$ that lie in $[xt,\infty)$. For all $x\neq \sup\{ y: \alpha^*(y)>-\infty\}$,
$$\frac{\log( \E[ N_t[xt,\infty)])}{t} \to \alpha^*(x) \quad \text{ as }t\to\infty.$$
\end{prop}

Recall from Definition \ref{kappa} that 
$$\kappa_0=\inf_{s\in (0,1)}\frac{f(s)}{s\log(1/s)}.$$
The following result follows from \cite[Corollary 2]{biggins1995growth}. 
\begin{prop}
\label{prop:bt-limit}
For $\BRW(t)$ the rightmost particle satisfies, 
\beq\label{Biggins}
\frac{B(t)}{t}\to \inf \{ x:\alpha^*(x)<0\} = \kappa_0 \quad \text{ almost surely}.
\eeq
\end{prop}

\begin{proof}
The convergence ${B(t)}/{t}\to \inf \{ x:\alpha^*(x)<0\}$ is proven in \cite[Corollary 2]{biggins1995growth}. To see that $ \kappa_0 = \inf \{ x:\alpha^*(x)<0\}$, note that for $s\in(0,1)$, $ x\log s+\frac{f(s)}{s}<0$ is equivalent to $x>\frac{f(s)}{s\log(1/s)}$. Thus, $x>\inf_{s\in(0,1)}\frac{f(s)}{s\log(1/s)}$ if and only if  $\alpha^*(x)<0$. 

\end{proof}

Next we state a large deviations result for $\pr(B(t)\geq xt)$ for large $t$ when $x>\kappa_0$, which is the key to our proof of Theorem \ref{thm:height}, and is interesting in its own right.
\begin{lemma}\label{BRW_deviation}
With $\kappa_0$ as in \eqref{eqn:kappa-ht} and $B(t)$ as the rightmost particle in $\BRW(t)$, for $x>\kappa_0$,
$$\pr(B(t)\geq xt) = \exp(\alpha^*(x)t + o(t)) \ \text{ as } \ t \to \infty.$$
\end{lemma}
\begin{proof}
We start with the upper bound. For fixed $x> \kappa_0$, recall the functional $N_t[xt,\infty)$ from Proposition \ref{Thm4}. Using Proposition \ref{Thm4} gives,
$$\pr(B(t)\geq xt)\leq \pr(N_t[xt,\infty)\geq 1)\leq \E[N_t[xt,\infty)]\leq \exp(\alpha^*(x)t+o(t)).$$

To prove the lower bound, we will use an argument of induction in time. For $\ep \in (0,1)$, let $\FF_{\ep t}$ denote the filtration generated by the genealogies and locations of all the particles in $\BRW$ born up till time $\ep t$.  For $v \in \BRW(\ep t)$, let $S_v$ denote the location of the particle $v$ in $\BRW(\ep t)$. For each such $v$ consider the branching random walk encoding the genealogy and location of particles born after time $\ep t$ whose most recent common ancestor in $\BRW(\ep t)$ is $v$. 
For $t' \ge 0$, let $B^v(t')$ denote the location of the rightmost particle in this branching random walk originating from particle $v$ observed $t'$ time units after $\ep t$. It is easy to see that
\begin{align*}
\pr(B(t) \geq xt \, | \, \FF_{\ep t})&=\pr( \max_{v\in \BRW(\ep t)} B^v((1-\ep)t)\geq xt \, | \, \FF_{\ep t})\\
&=\pr\left(\max_{v\in \BRW(\ep t)}\bigg\{ \frac{B^v((1-\ep)t)-S_v}{(1-\ep)t}+\frac{S_v}{(1-\ep)t}\bigg\}\geq \frac{x}{1-\ep} \, \bigg| \, \FF_{\ep t}\right).
\end{align*}

If there exists a particle $v\in \BRW(\ep t)$ such that both $S_v\geq x\ep t$ and $B^v((1-\ep)t)-S_v\geq x(1-\ep)t$, then we will have $B(t)\geq xt$. Based on this observation, we define $D_{\ep t}:=\{ v\in \BRW(\ep t): S_v\geq x\ep t\}$, noting that $D_{\ep t}$ is measurable with respect to $\FF_{\ep t}$. In addition, the collection $\{ B^v((1-\ep)t)-S_v: v\in \BRW(\ep t)\}$ comprise i.i.d. random variables, independent of $\FF_{\ep t}$, each distributed as the location of the rightmost particle in $\BRW((1-\ep)t)$. Hence,
\begin{align*}
\pr(B(t) \geq xt \, | \, \FF_{\ep t})&\geq \pr\left( \max_{ v\in D_{\ep t}} \frac{B^v((1-\ep)t)-S_v}{(1-\ep)t}\geq x \, \bigg| \, \FF_{\ep t} \right)\\
&=1-\left( 1-\pr\left(\frac{B((1-\ep)t)}{(1-\ep)t}\geq x\right)\right)^{|D_{\ep t}|}\\
&\geq \pr\left(\frac{B((1-\ep)t)}{(1-\ep)t}\geq x\right)|D_{\ep t}| \left(1 - \pr\left(\frac{B((1-\ep)t)}{(1-\ep)t}\geq x\right)\cdot |D_{\ep t}|\right),
\end{align*}
where the last line follows from the elementary inequality  $1-(1-x)^y\geq 1-e^{-xy}\geq xy(1-xy)$ for   $x\in [0,1], y\geq 0$. Taking expectations on both sides of the above bound, we have
\begin{align}\label{Bt_lb1}
\pr(B(t) \geq xt)
\geq  \pr\left(\frac{B((1-\ep)t)}{(1-\ep)t}\geq x\right) \left(\E |D_{\ep t}| - \pr\left(\frac{B((1-\ep)t)}{(1-\ep)t}\geq x\right)\cdot \E(|D_{\ep t}|^2)\right),
\end{align}
where $|D_{\ep t}|$ denotes the size of $D_{\ep t}$. Noting that $|D_{\ep t}|=N_{\ep t}[x\ep t,\infty)$, it follows directly from Proposition \ref{Thm4}  that for $x>\kappa_0$, 
$$\frac{\log (\E  |D_{\ep t}|) }{\ep t} \to \alpha^*(x)=\inf_{s\in (0,1)}\{x\log s+f(s)/s\}<0,$$
i.e., 
\beq \label{lb_sizeD}
\E  |D_{\ep t}| \geq \exp(\alpha^*(x) \ep t -o( t)).
\eeq
We claim that $\pr\left(\frac{B((1-\ep)t)}{(1-\ep)t}\geq x\right)\cdot \E(|D_{\ep t}|^2)=o(   \E |D_{\ep t}|)$ when $\ep>0$ is chosen to be sufficiently small. To see this, we use the upper bound proved earlier to get
$$
 \pr\left(\frac{B((1-\ep)t)}{(1-\ep)t}\geq x\right)\leq \exp( \alpha^*(x) (1-\ep) t+o(t)).
$$
In addition, Lemma \ref{mart_power} gives a trivial upper bound on $\E(|D_{\ep t}|^2)\leq \E[ n(\ep t)^2]\leq 2e^{2\ep t}$.
Knowing \eqref{lb_sizeD} and noting that $\alpha^*(x)<0$ for $x>\kappa_0$, we can then choose $\ep<\frac{\alpha^*(x)}{3\alpha^*(x)-2}$ so that
\begin{align*}
\pr\left(\frac{B((1-\ep)t)}{(1-\ep)t}\geq x\right)\cdot \E(|D_{\ep t}|^2)&\leq 2\exp\left( ( \alpha^*(x)(1-\ep) +2\ep )t\right)\leq \exp \left( 2\alpha^*(x)\ep t\right)=o(\E  |D_{\ep t}|). 
\end{align*}
Therefore, when $t$ is sufficiently large, combining the above with \eqref{Bt_lb1} and \eqref{lb_sizeD} leads to
\begin{align*}
\pr(B(t) \geq xt)&\geq  \frac{1}{2} \pr\left(\frac{B((1-\ep)t)}{(1-\ep)t}\geq x\right)\cdot  \E |D_{\ep t}|\\
&\geq   \frac{1}{2} \pr\left(\frac{B((1-\ep)t)}{(1-\ep)t}\geq x\right)\exp(\alpha^*(x) \ep t -o(t)),
\end{align*}
which then implies the following relation
\begin{align*}
&\liminf_{t\to\infty}\frac{1}{t}\log \pr\left(\frac{B(t)}{t}\geq x\right)\geq \alpha^*(x)\ep+(1-\ep)\liminf_{t\to\infty}\frac{1}{(1-\ep)t}\log \pr\left(\frac{B((1-\ep)t)}{(1-\ep)t}\geq x \right).
\end{align*}
This proves our desired conclusion
\begin{align*}
\pr\left(\frac{B(t)}{t}\geq x \right)&\geq \exp(\alpha^*(x)t-o(t)).
\end{align*}
\end{proof}
\color{black}

\subsection{Connection to Quasi stationary random walks}
\label{sec:quasi}
We start by clarifying the appearance of the mysterious functionals such as $R, s_0$ etc in the statement of the main results. Recall the process $\cT^*(\cdot), \cT(\cdot)$ in Definition \ref{defn:cts-time}. For $i\geq 0$, let $\cP_i(t)$ denote the number of vertices at distance $i$ to the root in $\TT^*(t)$ with $\cP_0(t) \equiv 1$ for all $t$.  Write $\bcP(t)=(\cP_0(t),\cP_1(t),\ldots)^\prime$ for the entire column vector. Let $\tilde{\cP}_i(t)$ and $\tilde{\bcP}(t) = (\tilde{\cP}_0(t),\tilde{\cP}_1(t),\ldots)^\prime$ denote the analogous objects for $\cT$.  Define the two (infinite dimensional) matrices $\vA = (A_{ij})_{i,j\geq 0}$ and $\vB = (B_{ij})_{i,j\geq 0}$  ,  
\begin{equation}
\label{eqn:matrix-def}
    \vA=\begin{pmatrix}
0 & 0 &0 & \cdots \\
p_0& p_1 & p_2 &\cdots \\
0 & p_0 & p_1&\cdots\\
0& 0 &p_0 &\cdots\\
0 &\cdots
\end{pmatrix}, \qquad \vB =\begin{pmatrix}
0 & 0 &0 & \cdots \\
c_0& c_1 & c_2 &\cdots \\
0 & p_0 & p_1&\cdots\\
0& 0 &p_0 &\cdots\\
\cdots
\end{pmatrix},
\end{equation}
where $c_i=\sum_{k=i}^\infty p_k$. In particular, $c_0=1$. 
The following is easy to check from the evolution dynamics. We omit the proof.

\begin{lemma}
\label{lem:inf-ODE}
For $t\geq 0$, 
\[\frac{d}{dt}\E[\bcP(t)]=\vA\cdot \E[\bcP(t)], \qquad \frac{d}{dt}\E[\tilde{\bcP}(t)]=\vB\cdot \E[\tilde{\bcP}(t)]. \]
\end{lemma}

Now we will come to the first connection between the tree valued process and random walk $\{S_n:n\geq 0\}$ as in \eqref{eqn:rw-def} with increments distributed as $Z-1$. Recall the hitting times of zero started from location $k$ namely $\barT_k$ in \eqref{eqn:hit-rw}. \tcr{The next lemma intuitively follows from solving the first differential equation above and noting that $\vA$ is the transition matrix for the Markov chain $\{S_n:n\geq 0\}$.}
\begin{lemma}\label{Pk}
For all $k\geq 1$,
$$\E[\cP_k(t)]=\sum_{i=0}^\infty \frac{t^i}{i!}\pr(\barT_k=i).$$
\end{lemma}
\begin{proof}
From Lemma \ref{lem:inf-ODE}, we have, 
\begin{equation}
    \label{eqn:536}
    \frac{d}{dt}\E[\cP_k(t)]=\sum_{j=0}^\infty A_{kj}\E[\cP_j(t)].
\end{equation}
Let $f_k(t):=\sum_{i=0}^\infty \frac{t^i}{i!}\pr(\barT_k=i)$ for $k\geq 1$ and write $f_0(t)\equiv 1$. Recalling the matrix $\vA = (A_{kj})_{k\geq 0, j\geq 0}$ from \eqref{eqn:matrix-def}, note that 
\begin{align*}
f_k'(t)&=\sum_{i=0}^\infty \frac{t^{i}}{i!}\pr(\barT_k=i+1)=\pr(\barT_k=1)+\sum_{i=1}^\infty \frac{t^i}{i!}\sum_{j=1}^\infty A_{kj}\pr(\barT_j=i)\\
&=A_{k0}+\sum_{j=1}^\infty A_{kj}f_j(t)=\sum_{j=0}^\infty A_{kj}f_j(t).
\end{align*}
We then compare this system of ODEs with \eqref{eqn:536}. Let $h_k(t)=f_k(t)-\E[\cP_k(t)]$ for $t\geq 0$. We have $h_k(0)=0$ and $h'_k(t)=\sum_{j=0}^\infty A_{kj}h_j(t)$ for all $k\geq 1$. Thus, 
\begin{align*}
|h_k(t)|&\leq \sum_{j=0}^\infty A_{kj}\int_0^t |h_j(s)| ds.
\end{align*}
Let $h^*(t)=\sum_{k=0}^\infty |h_k(t)|$. It is not difficult to check $h^*(\cdot)$ is continuous. 
Hence,
$$h^*(t)\leq \sum_{j=0}^\infty( \sum_{k=0}^\infty  A_{kj}) \int_0^t |h_j(s)|ds\leq\int_0^t h^*(s)ds.$$

By Gr\"onwall's inequality, $h^*(t)\equiv 0$. This proves the lemma.
\end{proof}

\mn
Given the connection between random walks and the tree evolution in Lemma \ref{Pk}, it is clear that spectral properties of $\vA$ and $\vB$ are key to understanding the evolution of tree functionals. The next result derives some properties. In all the ensuing results in this Subsection, we will {\bf always} make the Assumptions \ref{ass:mean} and \ref{ass:pmf}.

\begin{prop}[Spectral properties of $\vA, \vB$]
\label{prop:spectral-prop}
\begin{enumeratea}
\item For any positive $s$ such that $f(s)<\infty$, the vector $\vv_s = (v_i(s):i\geq 0)$ with $v_i(s) = s^{-i}$, $i \ge 0$, is a non-negative left sub-invariant eigenvector of $\vA$ for eigenvalue $f(s)/s$.
\item For any $s\geq 1$ such that $f(s) < \infty$, the vector $\vv_s = (v_i(s):i\geq 0)$ as before is a non-negative left sub-invariant eigenvector of $\vB$ for eigenvalue $f(s)/s$.
\item Recall $q_*$ from Definition \ref{def:R-q-s0}. When $\E[Z]>1$, the vector $\vu = (u_i:i\geq 0)$ with $u_0= 0$ and for $i\geq 1$ $u_i = q_*^{i-1}$  is a non-negative right eigenvector of $\vB$. 

\end{enumeratea}
\end{prop}
\begin{proof}
(a) For $j\geq 0$,
\begin{align*}
\sum_{i=1}^\infty s^{-i}A_{ij}&=\sum_{i=1}^{j+1}s^{-i}p_{j+1-i}=s^{-(j+1)}\sum_{i=1}^{j+1}s^{j+1-i}p_{j+1-i}
\leq s^{-j} \frac{f(s)}{s}.
\end{align*}

(b) Recall $c_j=\sum_{l=j}^{\infty}p_l$. For $j\geq 0$,
\begin{align*}
\sum_{i=1}^\infty s^{-i}B_{ij}&=\sum_{i=1}^{j+1}s^{-i}p_{j+1-i}+s^{-1}c_{j+1}=s^{-(j+1)}\left( \sum_{i=1}^{j+1}s^{j+1-i}p_{j+1-i}+s^{j}c_{j+1}\right)\\
&\leq s^{-(j+1)}\left( \sum_{i=0}^{j}s^{i}p_{i}+ \sum_{i=j+1}^{\infty} s^{i}p_{i}\right) = s^{-j} \frac{f(s)}{s},
\end{align*}
where, in the first inequality, we have used $s \ge 1$. 

(c) Note that since $f(q_*)=q_*$,
\begin{align*}
(\vB \vu)_1&=\sum_{k=0}^\infty c_ku_k=\sum_{k=1}^\infty \left(1-\sum_{l=0}^{k-1} p_l\right) q_*^{k-1}=\left(\frac{1}{1-q_*}-\sum_{l=0}^\infty p_l\sum_{k=l+1}^\infty q_*^{k-1}\right)\\
&=\frac{1}{1-q_*}-\frac{f(q_*)}{1-q_*}
=1=u_1.
\end{align*}
For $j\geq 2$, 
\begin{align*}
(\vB \vu)_j&=\sum_{k=0}^\infty p_ku_{k+j-1}=\sum_{k=0}^\infty p_k q_*^{k+j-2}=q_*^{j-2}f(q_*)=q_*^{j-1} = u_j.
\end{align*}
Clearly, $(\vB\vu)_0 = 0 = u_0$.

\end{proof}

\subsection{Urn models and multitype branching processes}
\label{sec:urn-model-def}
This section connects the evolution of $\bcP$ with urn models and eventually to finite dimensional multitype branching processes. 
We will consider a truncated version of $\bcP$ where we keep track of vertices at distance at most $k \ge 1$ from the root for some fixed $k$. Recall the matrix $\vA$ from \eqref{eqn:matrix-def} and let $\vA_k$ denote a $k\times k$ matrix such that $(\vA_k)_{ij}=A_{ij}$ for $1\leq i,j\leq k$. Let $\ve_j^{(k)}$ denote the unit basis vector in $\bR^{k}$ with one in the $j$-th co-ordinate and zero elsewhere. 

\begin{defn}[Urn model encoding distance from the root]
\label{defn:urn-model}
Fix $k\geq 1$. Consider a generalized multitype Polya-urn process with types $\{1,\dots,k\}$ starting with a single ball of type $1$. When a ball of type $i$ is drawn, it is returned along with $\mvxi_i=(\xi_{i1},\dots,\xi_{ik})$ other balls, where 
$$\begin{cases}
\pr(\mvxi_i=\textbf{e}^{(k)}_j)=A_{ji}  & (1\leq j\leq k) \\
\pr(\mvxi_i=\mvzero)=1-\sum_{j=1}^k A_{ji} 
\end{cases}
$$
It is easy to see $\E(\xi_{ij})=A_{ji}$. 
\end{defn}
A key tool in studying such urn processes is the so-called Athreya-Karlin embedding in finite multitype branching processes \cite{athreya1968embedding, janson2004functional}, which in this case corresponds to the following: we start with one individual (ball) at time zero of type 1. Each individual of type $i$ lives for an exponential mean one unit of time and upon dying gives birth to an offspring of type $i$ and possibly another vertex whose type is determined by an i.i.d. sample of $\mvxi_i$ (only one offspring of type $i$ is produced if the sampled $\mvxi_i=0$). In this continuous time process let $\bar{\bcP}^k(t) = (\bar{\cP}_i(t) :1\leq i\leq k)$ denote the individuals (balls) of various types alive at time $t$. 

\begin{lem}
\label{lem:urn-mod-stoch-dom}
Let $\sigma_1$ denote the first time that the root reproduces a vertex at distance one in $\cT^*$ so that $\cP_1(\sigma_1) = 1 $. We can couple $\bar{\bcP}^k(\cdot) $ with the process $\bcP(\cdot)$ defined in Section  \ref{sec:quasi}  such that $\cP_1(t+\sigma_1) \geq \bar{\cP}_1(t)$ for all $t$. 
\end{lem}

\begin{proof}
We describe an explicit coupling: 
\begin{enumeratei}
\item  If at time $t+\sigma_1$, a vertex in $\cT^*$ at distance $1\leq i\leq k$ from the root reproduces to produce a vertex at distance $1 \leq j\leq k$, a ball of type $i$ is removed from the urn at time $t$ and is returned along with a ball of type $j$.
\item If the new vertex is at distance greater than $k$ from the root, or if the root or a vertex of distance greater than $k$ reproduces, then no change is made to the urn model.
\end{enumeratei}
It is clear from our construction that $\cP_1(t+\sigma_1)\geq \bar{\cP}_1(t)$ for all $t\geq 0$. 
\end{proof}

The following is obtained by a direct application of \cite[Theorem 3.1]{janson2004functional}. Define $k_0 := \inf\{k \ge 1: p_k >0\}$. Note that $k_0< \infty$ as $p_0 < 1$.
\begin{prop}
\label{prop:finite-urn}
For $k \ge k_0$, $\vA_k$ possesses a positive largest (Perron-Frobenius) eigenvalue $\alpha_k$ and a strictly positive right eigenvector $\bar{\vv}_k = (\bar{\vv}_k(i):1\leq i\leq k)$ such that, 
$$e^{-\alpha_k t} \bar{\cP}_1^k(t)\overset{a.s.}{\longrightarrow} W \bar{\vv}_k(1)$$
for a strictly positive random variable $W$.
\end{prop}
\begin{proof}
Recalling $p_0+ p_1<1$ and $p_0 >0$, it follows that $\vA_k$ is irreducible in the sense of \cite{janson2004functional} for every $k \ge k_0$ and one easily sees that assumptions (A1)-(A6) in \cite{janson2004functional} are satisfied, required for the application of \cite[Theorem 3.1]{janson2004functional}. The strict positivity of $W$ follows from the fact that extinction is impossible in our case (see \cite[Lemma 2.1]{janson2004functional}).
\end{proof}
Thus, for $k \ge k_0$, $\alpha_k$ plays a crucial role in the growth of $\bar{\cP}_1^k(\cdot)$. The next result describes its asymptotics as $k\uparrow \infty$.  

\begin{prop}\label{Perroneigen}
$\alpha_k \uparrow 1/R$ as $k\to\infty$.
\end{prop}
\begin{rem}
In principle we are asserting that the maximal eigenvalue of the submatrix $\vA_k$ of the infinite dimensional operator $\vA$ converges in the limit $k \rightarrow \infty$ to the maximal eigenvalue of $\vA$. If $\vA$ was a compact operator then this would follow from standard function analytic methods, however it can be checked that $\vA$ is \emph{not a compact operator} and thus we give a proof relying on the specific probabilistic interpretation of $\vA$ and the corresponding random walk. 
\end{rem}

\begin{rem}
In Section \ref{sec:inf-dim-mtbp} we will construct an infinite dimensional multitype branching process, driven by the matrix $\vB$ to track the {\bf entire} height profile of $\cT$. The analogous argument for $\cT^*$ does not work, since the corresponding process with $\vB$ replaced by matrix $\vA$ is {\bf not} $\alpha$-recurrent (Lemma \ref{lemma:appen-A-trans}) and thus necessitates the truncation scheme here.
\end{rem}
\begin{proof}
Let $\{Z_i\}_{i\geq 1}$ be a collection of i.i.d. random variables distributed as $\vp$ and recall that $S_n=S_0+\sum_{i=1}^n (Z_i-1)$. Define the stopping times, 
\begin{align*}
T=\inf\{ n\geq 0: S_n= 0\}, \quad
\tau=\inf\{ n\geq 1: S_n=1\}, \quad
\tau_k=\inf\{ n\geq 0: S_n>k\}, \, k \ge 1.
\end{align*}
In this proof, for $j\in \bZ$, let $\pr_j$ and $\E_j$  denote the probability and expectation operators for the walk started from $S_0 =j$. For $r\geq 0$, define $\Phi(r)=\E_1[ e^{r\tau} \cdot 1_{\{\tau<T\}}]$. The proof of the Proposition hinges on the following lemma, whose proof is postponed to the end of this Section.
\begin{lemma}\label{Phi_func}
$\Phi(\log R)< 1$ and $\Phi(r)=\infty$ when $r>\log R$.
\end{lemma}

Choose any $k \ge k_0$. As $\vA_k$ is substochastic, $\alpha_k \le 1$ by standard Perron-Frobenius theory. Write $\alpha_k=e^{-\theta_k}$ for $\theta_k \ge 0$. We claim that $\theta_k>0$ and
\beq\label{prop_tau}
\E_1[ e^{\theta_k \tau} 1_{\{ \tau<T\wedge \tau_k\}}]=1.
\eeq
To see this, consider the right eigenvector  $\bar{\vv}_k$ of $\vA_k$ for $k\geq 1$. Then, $\vP_{ij}=\frac{A_{ij} \bar{\vv}_k(j)}{\alpha_k \bar{\vv}_k(i)}$, $1 \le i,j \le k$ defines a probability transition matrix. Denote by $\mathbb{Q}_1$ the law of the Markov chain associated with $\vP$ starting from state 1. By slight abuse of notation, keep using $T,\tau, \tau_i$ as above, to denote the associated stopping times for this Markov chain. Then $\mathbb{Q}_1(\tau=1)=\vP_{11}=\frac{A_{11}}{\alpha_k}$. For $i\geq 2$, observe that 
\begin{align*}
\mathbb{Q}_1(\tau=i)&=\sum_{ \{z_j\}_{ 1\leq j \leq i-1}: 2\leq z_j\leq k} \frac{1}{\alpha_k^i} A_{1 z_1}A_{z_1z_2}\cdots A_{z_{i-1}1}\\
&=e^{i\theta_k}\cdot \pr_1(\tau=i, i<T\wedge \tau_k ).
\end{align*}
Since $\vP$ is irreducible (and finite),
$$1=\mathbb{Q}_1(\tau<\infty)=\E_1\left[ e^{\theta_k \tau}1_{ \{ \tau<T\wedge \tau_k\}}\right],$$
i.e.,  \eqref{prop_tau} holds.  \tcr{Since $p_0 >0$,  $\pr_1(\tau<T\wedge \tau_k) < 1$ and this implies that $\theta_k >0$.}

Note that as 
$$1=\E_1\left[ e^{\theta_{k+1} \tau}1_{ \{ \tau<T\wedge \tau_{k+1}\}}\right]=\E_1\left[ e^{\theta_k \tau}1_{ \{ \tau<T\wedge \tau_k\}}\right]\leq \E_1\left[ e^{\theta_k \tau}1_{ \{ \tau<T\wedge \tau_{k+1}\}}\right],$$
we have $\theta_{k+1}\leq \theta_k$ for any $k\geq 1$ and thus $\theta_k \downarrow \theta^*$ for some $\theta^*\geq 0$.

It remains to show that $\theta^*=\log R$, which implies  $\alpha_k=e^{-\theta_k}\uparrow e^{-\theta^*}=1/R$ as $k\to\infty$. 
First we prove $\theta^*\geq \log R$. It follows from the monotonicity of $\{\theta_k\}$ and \eqref{prop_tau} that
\begin{align*}
\Phi(\theta_k)&=\E_1[e^{\theta_k\tau}1_{\{\tau<T\}}]=\lim_{i\to\infty} \E_1[e^{\theta_k\tau}1_{\{\tau<T\wedge \tau_i \}}]\geq \lim_{i\to\infty} \E_1[e^{\theta_i\tau}1_{\{\tau<T\wedge \tau_i \}}]=1.
\end{align*}
$\Phi(\cdot)$ is strictly increasing (as $\pr_1(\tau< T)>0$) and $\Phi(\log R) < 1$ by Lemma \ref{Phi_func}, implying that $\theta_k > \log R$. Therefore, $\theta^*=\lim_k \theta_k \geq \log R$.

Next, to prove $\theta^*\leq \log R$, note that 
\begin{align*}
\Phi(\theta^*)&=\E_1[e^{\theta^*\tau}1_{\{\tau<T\}}]=\lim_{i\to\infty} \E_1[e^{\theta^*\tau}1_{\{\tau<T\wedge \tau_i \}}]\leq \lim_{i\to\infty} \E_1[e^{\theta_i\tau}1_{\{\tau<T\wedge \tau_i \}}]=1.
\end{align*}
As $\Phi(r)=\infty$ for all $r>\log R$ by Lemma \ref{Phi_func}, the above implies that $\theta^*\leq \log R$. The result follows.
\end{proof}

\mn
\textit{Proof of Lemma \ref{Phi_func}.} 

%
Define $\chi(u)=\sum_{j=0}^\infty \pr_1(T=j)u^j$ for $u\geq 0$. 
Note that $\E_1[e^{rT}]=\chi(e^r)$. Defining $\kappa=\inf\{ n\geq 0: S_n=1\}$, Strong Markov property implies that 
\begin{align}\label{Phi_chi}
\Phi(r)&=e^r \sum_{j=1}^\infty p_j \E_j[e^{r\kappa}]=e^r \sum_{j=1}^\infty p_j \left(\E_1[e^{rT}]\right)^{j-1}=e^r \sum_{j=1}^\infty p_j \chi(e^r)^{j-1}\nonumber\\
&=e^r \left( \frac{f(\chi(e^r))-p_0}{\chi(e^r)}\right)<\infty
\end{align}
for all $r$  such that $f(\chi(e^r))< \infty$ and $\Phi(r)=\infty$ when $\chi(e^r)< \infty$ but $f(\chi(e^r))= \infty$. We claim that
\beq\label{chi_func}
\chi(R)=s_0 \quad \text{ and } \quad \chi(u)=\infty \quad  \text{ when }u> R.
\eeq
The desired result now follows from \eqref{Phi_chi}:
\begin{align*}
\Phi(\log R)&=R \cdot \frac{f(s_0)-p_0}{s_0}=\frac{s_0}{f(s_0)}\left(\frac{f(s_0)}{s_0}-\frac{p_0}{s_0}\right)=1-\frac{p_0}{f(s_0)}<1,
\end{align*}
recalling $f(s_0)< \infty$, and $\Phi(r)=\infty$ when $r>\log R$. 

It remains to prove the claim in \eqref{chi_func}. Lemma 1 in  \cite{pakes1973conditional} implies that 
$$\lim_{n\to\infty} \frac{\pr_1(T=n+1)}{\pr_1(T=n)}=\frac{1}{R}.$$
By the ratio test for power series, $\chi(u)<\infty$ if $s<R$ and $\chi(u)=\infty$ if $s>R$. 

In the cases where $R=1$, by Lemma \ref{lemma:prop-R-s0} we have $s_0=1$ and $\E[Z]\leq 1$. It follows from \cite[Lemma 11.3]{gut2009stopped} that, 
 $$\chi(1)=\sum_{j=1}^\infty \pr_1(T=j)=\pr_1(T<\infty)=1.$$

When $R>1$, to show $\chi(R)=s_0$, observe that for $u<R$,
$$\chi(u)=p_0u+\sum_{j=1}^\infty p_ju(\chi(u))^j=uf(\chi(u)).$$
As $\chi(u)$ is strictly increasing on $[0,R)$, it has an inverse function $\chi^{-1}(\cdot)$ given by
$$\chi^{-1}(z)=\frac{z}{f(z)}$$
for $z\in [0,\chi(R))$, i.e., $\frac{z}{f(z)}\in [0,R)$. Let $\{s_l:l\geq 1\}$ be a sequence of positive numbers such that $\lim_{l\to\infty}s_l=s_0$ and $s_l<s_0, \frac{s_l}{f(s_l)}\in [0,R)$ for all $l\in\mathbb{N}$. Hence, 
$$\lim_{l\to\infty} \chi\left(\frac{s_l}{f(s_l)}\right)=\lim_{l\to\infty} s_l=s_0.$$
The final step is now to show $\chi(R) = \lim_{l\to\infty} \chi\left(\frac{s_l}{f(s_l)}\right)$. We define a series of functions $\{ \chi_k(u): k\geq 1\}$ with $\chi_k(u)=\sum_{j=1}^k \pr_1(T=j)u^j$. For every $k\geq 1$,  $\chi_k(u)$ is continuous and $\{ \chi_k(u): k\geq 1\}$ is a non-decreasing sequence for all $u\geq 0$. As $\chi(\cdot)$ is the limit of a monotone increasing sequence of continuous functions, it is lower semicontinuous, i.e., $\chi(R)\leq \liminf_{s\uparrow R} \chi(u)$.
Therefore,
$$\chi(R)\leq \liminf_{s\uparrow R} \chi(u) \leq \lim_{l\to\infty} \chi\left(\frac{s_l}{f(s_l)}\right)=\lim_{l\to\infty} s_l=s_0.$$
Further, by the monotonicity of $\chi(\cdot)$, $\chi(R)\geq \chi(\frac{s_l}{f(s_l)})$ for any $l\in \mathbb{N}$ and hence $\chi(R)\geq \lim_{l\to\infty}\chi(s_l/f(s_l))= s_0$. This completes the proof.
\qed

\subsection{Infinite dimensional multitype branching processes}
\label{sec:inf-dim-mtbp}

In this Section we introduce a specific class of  continuous time multitype branching processes $\{\MBP(t): t \ge 0\}$ studied in full generality in  \cite{jagers1989general,jagers1996asymptotic}. The process $\MBP(\cdot)$ will once again track distances in  $\TT(\cdot)$ so that  a vertex at distance $\ell$ from the root corresponds to a type $\ell$ individual in $\MBP(\cdot)$.

The Ulam-Harris set $I=\cup_{n\geq 0} \mathbb{N}^n$, with $N^{0} = \emptyset$ representing the root, is used to encode the set of all possible individuals (vertices) in $\MBP$,  and let $(S,\mathcal{S})$ denote the space of types with a countably generated $\sigma$-algebra $\mathcal{S}$, where $S=\{0,1,\dots\}$. Start with a single individual (the root) of some type $r \in S$. Any existing individual of type $r' \in S$ in the population independently reproduces according to a rate one Poisson process, and each new individual $y$ is independently assigned the type $\rho(y)=s$ with probability $B_{sr'}$, where $\vB$ is the matrix defined in \eqref{eqn:matrix-def}. The reproduction process $\xi_x$ of an individual $x$ is a measure on $S \times \mathbb{R}_+$, with $\xi_x(A \times B)$ denoting the number of children of $x$ of types in $A$ born at times in $B$. We will write $\xi$ for the reproduction process of the root.

Observe that, if we assign each vertex in our tree process $\cT(\cdot)$ the type equalling distance from the root, then $\cT(\cdot)$ has the same distribution as $\MBP(\cdot)$ with root of type $0$. In particular, for each $k\geq 0$ and all $t\geq 0$,
\beq
 \tilde{\cP}_k(\cdot) \equald | \{ x\in \MBP(\cdot): \rho(x)=k\}|.
\eeq

%


Let $\{\tau_k:k \geq 1\}$ denote the birth times of the root in $\MBP(\cdot)$. 
For $r\in S$, we will denote by $\E_r$ the expectation operator when the root is of type $r$. 
Define the reproduction kernel $\mu^\vB$,
$$\mu^\vB(r, ds\times dt)=\E_r[\xi(ds\times dt)].$$
Note that for $r,s \in S$, $\mu^\vB(r, s\times dt) := \mu^\vB(r, \{s\}\times dt) = B_{sr}dt.$
For any $\lambda>0$ define the measure associated with the Laplace transform of $\mu^\vB$,
$$\mu^\vB_{\lambda}(r,ds\times dt)= e^{-\lambda t}\mu^\vB(r,ds\times dt).$$
Define
$$ \mu^\vB_\lambda(r,s):= \mu^\vB_\lambda(r,s \times \RR_+) = \int_0^\infty e^{-\lambda t} \mu^\vB(r, s\times dt)=\frac{B_{sr}}{\lambda}.$$

The $n$-th convolution of $\mu^\vB_\lambda$ is denoted by $\mu_{\lambda}^{(n)}$ (suppressing dependence on $\vB$) where $\mu_{\lambda}^{(1)} := \mu^\vB_\lambda$ and for $n \ge 2$,
\begin{equation*}
\mu_\lambda^{(n)}(r,A\times B)=\int_{S\times \RR_+} \mu^{(n-1)}_\lambda(s,A\times (B-t))\mu^\vB_\lambda(r,ds\times dt).
\end{equation*}


The Malthusian rate associated with $\MBP(\cdot)$ is defined as
$$
\alpha := \inf\{\lambda>0:\sum_{n=0}^\infty \mu_{\lambda}^{(n)}(s,S \times \RR_+)<\infty \text{ for some } s \in S\}.
$$
Since $\vB^T$ is a stochastic matrix, it follows that $\alpha=1$, and $\MBP(\cdot)$ is Malthusian and supercritical as required for the results of \cite{jagers1989general,jagers1996asymptotic}. We show in Appendix \ref{apprec} that $\mu^\vB_\alpha$ is irreducible and $\alpha$-recurrent in the terminology of \cite{nummelin2004general}. Choosing $\pi(r)=(1-q_*)u_r, \, r \in S$ (note that this is {\bf non-trivial} if and only if $\E[Z] > 1$ so that $q_* < 1$), where $\vu=(u_i: i\geq 0)$ is the right eigenvector of $\vB$ in Proposition \ref{prop:spectral-prop}, and $h(\cdot)\equiv 1$, note that $\sum_{r\in S} h(r)\pi(r)=1$ and
$$\sum_{r\in S} \pi(r)\mu^\vB_\alpha(r,s)=\pi(s)\quad \text{ and }\quad \sum_{s\in S} \mu^\vB_\alpha(r,s)h(s)=h(r).$$

To apply the results in \cite{jagers1989general,jagers1996asymptotic}, we need to check that, in addition to the above, the following conditions are satisfied for $\MBP(\cdot)$. In the following, for a non-negative random variable $U$ on the probability space on which $\MBP$ is defined, write $\E_{\pi}(U) := \int_S \E_s(U)\pi(ds)$. Although $\alpha=1$ in our case, we retain this notation to highlight the dependence on $\alpha$ in potentially more general applications.
\begin{enumeratei}
\item \label{it:one} The reproduction kernel $\mu^\vB$ is non-lattice, Malthusian, supercritical and satisfies
$\sup_s \mu(s,S\times[0,\ep])<1$
for some $\ep>0$. 
\item \label{it:two} The homogeneity assumption on $h$, i.e., $\inf_{s\in S} h(s)>0$, is satisfied.
\item \label{it:three} The kernel $\mu^\vB$ has strong $\alpha$-recurrence in the sense that 
$$0<\beta :=\int_{S\times S\times \RR_+} te^{-\alpha t}h(s)\mu^\vB(r,ds\times dt) \pi(dr)<\infty,$$
\item \label{it:four}The $x\log x$-condition is satisfied, i.e., 
$\E_\pi [\bar\xi \log^+\bar \xi]<\infty$, where  $\bar \xi=\int_{S\times \RR_+} e^{-\alpha t}h(s) \xi(ds\times dt)$. 

\end{enumeratei}

For Condition \eqref{it:one}, the Malthusian and supercritical property was discussed before. The remaining conditions are immediate from the fact that births happen according to a rate one Poisson process. Condition \eqref{it:two} is trivially satisfied due to our choice of $h\equiv 1$. Condition \eqref{it:three} can be checked by computing,
\begin{align*}
\beta&=\sum_{r=0}^\infty \sum_{s=0}^\infty \int_0^\infty te^{-\alpha t} h(s) \mu^\vB(r,s\times dt) \pi(r) dt=\sum_{r=0}^\infty \sum_{s=1}^\infty \int_0^\infty te^{-\alpha t} B_{sr}\pi(r) dt\\
&= \left(\int_0^\infty te^{-\alpha t} dt\right) \cdot \sum_{s=1}^\infty \sum_{k=-1}^\infty B_{s,s+k} \pi(s+k) \\
&=\frac{1-q_*}{\alpha^2} \cdot \left(\sum_{k=1}^\infty c_kq_*^{k-1} +\sum_{s=2}^\infty \sum_{k=-1}^\infty p_{k+1}q_*^{s+k-1}\right)=\frac{1}{\alpha^2} \cdot (1-q_*)\left(1+\frac{q_*}{1-q_*}\right)=\frac{1}{\alpha^2}.
\end{align*}
To check Condition \eqref{it:four}, note that since $h\equiv 1$, $\bar \xi=\sum_{k\geq 1} e^{-\alpha \tau_k}$. Let $\FF_{(k)}$ denote the natural filtration up to the birth of the $k$-th child in $\xi$. Then
\begin{align}\label{xlogx}
\nonumber \E_\pi [\bar\xi \log^+\bar \xi] &\leq \E_\pi[\bar \xi^2] =\sum_{k,l \geq 1} \E[ e^{-\alpha (\tau_k+\tau_l)}]\\
\nonumber &=\sum_{k\geq 1} \E[e^{-2\alpha \tau_k}]+2\sum_{ k<l} \E[ \E[e^{-\alpha(\tau_k+\tau_l)}|\FF_{(k)}]]\\
&=\sum_{k\geq 1} \left(\frac{1}{1+2\alpha}\right)^k +2\sum_{ k<l } \left(\frac{1}{1+\alpha}\right)^{l-k} \left(\frac{1}{1+2\alpha}\right)^k<\infty.
\end{align}
\textbf{$\mathbb{L}^1$-convergence in the non-fringe regime.} 
Let $\sigma_x$ denotes the birth time of vertex $x$ in the multitype branching process $\MBP(\cdot)$. Consider a \emph{characteristic} $\psi: S\times [0,\infty)\to [0,\infty)$, which is a random c\`adl\`ag measurable function, defined on the probability space on which $\MBP$ evolves, thought of as giving a `score' to the root at time $t$ based on its type and its genealogical tree. For $x \in I$ and $t \ge \sigma_x$, write $\psi_x(\rho(x), t - \sigma_x)$ for the corresponding score computed by evaluating the characteristic $\psi$ on the subtree rooted at $x$, when this vertex is of age $t-\sigma_x$. See \cite[Section 7]{jagers1989general} for a more formal treatment. Define the cumulative $\psi$-score
$$Z^\psi(t)=\sum_{x:\sigma_x\leq t}\psi_x(\rho(x),t-\sigma_x).$$


We will write $\E_s[\psi(t)]=\E[\psi(s,t)]$. The following theorem gives convergence of expectations for normalized cumulative $\psi$-scores.

\begin{theorem}[Theorem 1 in \cite{jagers1996asymptotic}]\label{jagersthm1}
Consider a non-lattice strictly Malthusian, supercritical branching population, counted with a bounded characteristic $\psi$ such that the function $t \mapsto e^{-\alpha t}\E_s[\psi(t)]$ is directly Riemann integrable with respect to $\pi$. 
Then, for $\pi$-almost every $s\in S$, 
$$ e^{-\alpha t}\E[Z^\psi(t)] \to h(s)\E_\pi[\hat{\psi}(\alpha)]/\alpha \beta.$$
\end{theorem}

\begin{corollary}\label{jagers_expectation}
In the non-fringe regime where $\E[Z]>1$,
$$\lim_{t\to\infty} e^{-t}\E[\tilde{\cP}_1(t)] = 1-q_*.$$
\end{corollary}
\begin{proof}

Let $\psi(s,t)=\mathbf{1}\{s=1,t\geq 0\}$. We can easily verify that $t \mapsto e^{-\alpha t}\E_s[\psi(t)]$ is directly Riemann integrable with respect to $\pi$. Moreover,
$$\E_\pi[\hat\psi(\alpha)]=\sum_{r=0}^\infty \pi(r) \int_0^\infty e^{-\alpha t} \E_r[\psi(t)]dt =\frac{\pi(1)}{\alpha}=1-q_*$$
as $\pi(1)=1-q_*$ and $\alpha=1$. Recall from previous section that $\beta=1/\alpha^2=1$. Applying Theorem \ref{jagersthm1} with $\psi(s,t)=\mathbf{1}\{s=1,t\geq 0\}$ and $h\equiv 1$ gives the result.

\end{proof}

Recall that $n(t)$ denotes the population size at time $t$. Let $\mathscr{F}_{t}$  be the filtration generated by the entire life histories of the first $n(t)$ vertices. Let $mx$ denote the mother of vertex $x$ and define for any $t,c\geq 0$,
$$\mathscr{I}(t)=\{ x; \sigma_{mx}\leq t<\sigma_x<\infty\},\quad \mathscr{I}(t,c)=\{ x; \sigma_{mx}\leq t, t+c<\sigma_x<\infty\}.$$

Define $w_t=\sum_{x\in \mathscr{I}(t)} e^{-\sigma_x}$ and $w_{t,c}=\sum_{x\in \mathscr{I}(t,c)} e^{-\sigma_x}$.  We collect some useful properties of $w_t$ and $w_{t,c}$  in the following lemma. Recall from Lemma \ref{lem:yule-asymp} that $W\sim \mathrm{Exp}(1)$ is the almost sure limit of $e^{-t}n(t)$ as $t \rightarrow \infty$.

\begin{lemma}\label{key_martingale}
(i) $\{w_t\}$ is a non-negative martingale with respect to $\{\mathscr{F}_t\}$ and $\E[w_t]=1$.\\
(ii) $w_t\to W$ almost surely and in $\mathbb{L}^2$ as $t\to\infty$. \\
(iii) $\E[w_{t,c}]\to k(c)$ as $t\to\infty$, where $k(c)\downarrow 0$ as $c\to\infty$.
\end{lemma}
\begin{proof}
(i) Follows from Proposition 2.4 and (2.17) in \cite{ nerman1981convergence}.
The $\mathbb{L}^2$ convergence in (ii) follows from Theorem 4.1 in \cite{jagers1984growth} using the fact that $\E_v[\bar \xi^2]< \infty$ derived in \eqref{xlogx}. Further, as $w_t$ is the reproductive martingale (see equation (2.15) in \cite{nerman1981convergence}), the almost sure convergence to $W$ follows from \cite[Corollary 2.5 and Theorem 5.4]{nerman1981convergence}.
(iii) follows from \cite[Lemma 3.5]{nerman1981convergence}.
\end{proof}
A version of the following theorem is proved for a class of multitype branching processes in \cite{jagers1996asymptotic}. However, it requires the finiteness of $\xi(S \times \RR_+)$, which is not satisfied in our case. We thus give a direct proof which bypasses some technicalities in \cite{jagers1996asymptotic} introduced by their generality of hypotheses.
\begin{theorem}\label{root_condensation}
In the non-fringe regime where $\E[Z]>1$, as $t\to\infty$,
$$e^{-t} \tilde{\cP}_1(t) \overset{ \mathbb{L}^1}{\longrightarrow}  (1-q_*)W,$$
where $W$ is as defined in Lemma \ref{key_martingale}.
\end{theorem}
\begin{proof}
Recall that $\alpha=\beta=1$.
Let $\psi(s,t)=\mathbf{1}\{s=1,t\geq 0\}$ so that $Z^\psi(t)=\tilde{\cP}_1(t)$. For any $M>0$, write $\psi=\psi_M+\psi'_M$ where $\psi_M(s,t)=\psi(s,t)\mathbf{1}\{t\leq M\}$ and $\psi'_M(s,t)=\psi(s,t)\mathbf{1}\{t> M\}$. 
To simplify notation, let 
$$\gamma:=\E_v[\hat{\psi}(\alpha)]/\alpha \beta=1-q_* \quad \text{ and }\quad \gamma_M:=\E_v[\hat{\psi}_M(\alpha)]/\alpha \beta=(1-e^{-M})(1-q_*).$$

Note that 
\begin{align}\label{errs}
\nonumber\E|e^{-t}Z^\psi(t)-\gamma W|&\leq e^{-t}\E|Z^{\psi}(t)-Z^{\psi_M}(t)|\\
&\quad +\E|e^{-t}Z^{\psi_M}(t)-\gamma_M W|+|\gamma_M-\gamma|\E[W].
\end{align}
The first term is equal to $e^{-t}\E[Z^{\psi'_M}(t)]$. Theorem \ref{jagersthm1} implies that
$$\lim_{t\to\infty} e^{-t}\E[Z^{\psi'_M}(t)]=(1-q_*)e^{-M}.$$
Hence, $e^{-t}\E[Z^{\psi'_M}(t)]$ can be made arbitrarily small by taking both $t$ and $M$ to be large. The third term vanishes as $M\to\infty$ since $\E[W]=1$ and $\lim_{M\to\infty}\gamma_M=\gamma$.
 
 It remains to deal with the second term in \eqref{errs}. Let $\zeta_t=e^{-t}Z^{\psi_M}(t)$ and $m_s(t)=\E_s[\zeta_t]$. By Theorem \ref{jagersthm1} we have 
\beq\label{m_convergence}
\lim_{t\to\infty} m_s(t)= \E_v[\hat{\psi}_M(\alpha)]/\alpha \beta=:\gamma_M
\eeq
for $v$-almost sure $s\in S$.

Observe that $\psi_M(u)=0$ for all $u\geq M$. Hence, for $t'\geq M$ we have
\begin{align*}
\zeta_{t+t'}&= \sum_{x\in \mathscr{I}(t)}e^{-\sigma_x}\zeta_{t+t'-\sigma_x}\circ \rho(x),
\end{align*}
where $\zeta_{t+t_0-\sigma_x}\circ \rho(x)$ denotes the normalized $\psi_M$ score for the vertex $x$, whose type is $\rho(x)$, at time $t+t'-\sigma_x$.
For $t,c\geq 0$ and $t'\geq c$, write $\eta(t+t'-\sigma_x)\circ \rho(x)=\zeta_{t+t'-\sigma_x}\circ \rho(x)-m_{\rho(x)}(t+t'-\sigma_x)$ and define
$$X(t,t',c)=\sum_{x\in \mathscr{I}(t)\backslash \mathscr{I}(t,c)} e^{-\sigma_x}\eta(t+t'-\sigma_x)\circ \rho(x).$$

By triangle inequality we have
\begin{align}\label{fourerrs}
\nonumber | \zeta_{t+t'}-\gamma_M W|&\leq  |X(t,t',c)| +\left|\sum_{x\in \mathscr{I}(t,c)} e^{-\sigma_x} \left(\zeta_{t+t'-\sigma_x}\circ \rho(x)-\gamma_M\right)\right|\\
&\quad + \gamma_M| w_{t}-W| +\sum_{x\in \mathscr{I}(t)\backslash \mathscr{I}(t,c)}e^{-\sigma_x}\left|m_{\rho(x)}(t+t'-\sigma_x)-\gamma_M\right|.
\end{align}
Fix any $\ep>0$. Our goal is to show that there exist $t, c$ (depending on $\ep$) sufficiently large and $t'_0=t'_0(t,c) \in \RR_+$ such that for all $t' \ge t'_0$, $\E| \zeta_{t+t'}-\gamma_M W|\leq \ep$. To do this we will deal with each term in \eqref{fourerrs} separately.

Note that $X(t,t',c)$ is a sum of independent random variables each having expectation zero under the conditional law $\pr(\cdot| \mathscr{F}_t)$. Hence, $\E[X(t,t',c)|\mathscr{F}_t]=0$ and
\begin{align*}
\E[X(t,t',c)^2|\mathscr{F}_t]&=\operatorname{Var}[X(t,t',c)|\mathscr{F}_t]\\
&=\sum_{x\in \mathscr{I}(t)\backslash \mathscr{I}(t,c)} e^{-2 \sigma_x} \operatorname{Var}[\eta(t+t'-\sigma_x)\circ  \rho(x) |\mathscr{F}_t]\\
&\leq \sum_{x\in \mathscr{I}(t)\backslash \mathscr{I}(t,c)} e^{-2 \sigma_x}  \E[ (\eta(t+t'-\sigma_x)\circ  \rho(x))^2|\mathscr{F}_t]
\end{align*}

Since $\zeta_t\leq e^{-t}n(t)$, it follows from Lemma \ref{mart_power} that for all $t\geq 0$ and $s\in S$,
$$\E_s[\zeta^2_t]\leq \E[\E[e^{-2t}n(t)^2|\FF_0]] \le 2.$$  
Note that for any $s\in S$, $m_s=\E_s[\zeta_t]\leq e^{-t}\E[n(t)]=1$. Then, for any $x\in \mathscr{I}(t)\backslash \mathscr{I}(t,c)$,  
\begin{align*}
\E[ (\eta(t+t'-\sigma_x)\circ  \rho(x))^2|\mathscr{F}_t]\leq& 2\E[(\zeta_{t+t'-\sigma_x}\circ \rho(x))^2+m_{\rho(x)}(t+t'-\sigma_x) ^2|\mathscr{F}_t]\\
\leq &  2(2+1)=6.
\end{align*}
Therefore, 
\begin{align*}
\E[X(t,t',c)^2]&=\E\left[\E[X(t,t',c)^2|\mathscr{F}_t]\right]\\
& \leq 6 \E\left[\sum_{x\in \mathscr{I}(t)\backslash \mathscr{I}(t,c)} e^{-2 \sigma_x}\right]\leq C'e^{-t}\E[w_t] = 6e^{-t}.
\end{align*}
Hence,
\beq\label{err1}
\E|X(t,t',c)|\leq \E[X(t,t',c)^2]^{1/2}\leq \sqrt{6}e^{- t/2}.
\eeq
There exists some $N_1\in \RR_+$ such that for $t\geq N_1$, and any choice of $t',c$, $\E|X(t,t',c)|\leq \ep/4$.
%
%

To address the second term in \eqref{fourerrs}, simply observe that 
\begin{align*}
&\E\left[\left|\sum_{x\in \mathscr{I}(t,c)} e^{-\sigma_x} \left(\zeta_{t+t'-\sigma_x}\circ \rho(x)-\gamma_M\right)\right|\big|\mathscr{F}_t\right]\\
\leq &\sum_{x\in \mathscr{I}(t,c)} e^{-\sigma_x}( m_{\rho(x)}(t+t'-\sigma_x)+\gamma_M)\leq (1+\gamma_M)w_{t,c}.
\end{align*}
Lemma \ref{key_martingale}(iii) implies that there exists some $N_2\in \RR_+$ so that if $c,t\geq N_2$ then, for any choice of $t'$,
\beq\label{err3}
\E\left|\sum_{x\in \mathscr{I}(t,c)} e^{-\sigma_x} \left(\zeta_{t+t'-\sigma_x}\circ \rho(x)-\gamma_M\right)\right| \leq (1+\gamma_M)\E[w_{t,c}]\leq \ep/4.
\eeq

For the third term, since $w_t$ converges to $W$ in $\mathbb{L}^1$, there exists some $N_3\in \RR_+$ so that 
\beq\label{err4}
\E[\gamma_M|w_t-W|]\leq \ep/4
\eeq
for all $t\geq N_3$. 

Finally we upper bound the expectation of the fourth term in \eqref{fourerrs}. Let $\mathcal{E}_t=\{\mathcal{H}_{\TT(t)}<2\kappa_0 t \}$. It follows from Proposition \ref{prop:ht-cts} (whose proof is independent of the tools using $\MBP$) that $\lim_{t\to\infty}P(\mathcal{E}^c_t)=0$. Note that on $\mathcal{E}_t$, the possible types for $x\in \mathscr{I}(t)\backslash \mathscr{I}(t,c)$ in $\MBP(\cdot)$ are in the set $[2\kappa_0 t]:=\{0,1,\dots, \lceil 2\kappa_0 t\rceil\}$.
%
Observe that 
\begin{align*}
&\E_0\left[\sum_{x\in \mathscr{I}(t)\backslash \mathscr{I}(t,c)}e^{-\sigma_x}\left|m_{\rho(x)}(t+t'-\sigma_x)-\gamma_M\right|\right]\\
\leq &\E_0\left[\sum_{x\in \mathscr{I}(t)\backslash \mathscr{I}(t,c)}e^{-\sigma_x}\sup_{s\in [2\kappa_0 t]}\left|m_{s}(t+t'-\sigma_x)-\gamma_M\right|; \mathcal{E}_t\right]\\
&\quad +\E_0\left[ \sum_{x\in \mathscr{I}(t)\backslash \mathscr{I}(t,c)}e^{-\sigma_x}(1+\gamma_M); \mathcal{E}^c_t\right].
\end{align*}
There exists some $N_4\in\RR_+$ so that for all $t\geq N_4$, and any choice of $t',c$, 
\begin{align}\label{err2_1}
\nonumber\E_0\left[ \sum_{x\in \mathscr{I}(t)\backslash \mathscr{I}(t,c)}e^{-\sigma_x}(1+\gamma_M); \mathcal{E}^c_t\right]&\leq (1+\gamma_M)\E[w_t; \mathcal{E}^c_t]\\
&\leq(1+\gamma_M)\left( \E[|w_t-W|]+\E[W; \mathcal{E}^c_t]\right)\leq \ep/8,
\end{align}
where the last inequality follows from Lemma \ref{key_martingale}.

Choose and fix any $t \ge \max_{1\leq i\leq 4} \{ N_i\}$ and $c \ge N_2$.
Since $x\in \mathscr{I}(t)\backslash \mathscr{I}(t,c)$, we have $t+t'-\sigma_x\geq t'-c$. Thus, there exists $t'_0=t'_0(t,c) \in \RR_+$ such that for all $t' \ge t'_0$, 
$$\sup_{s\in[ 2\kappa_0 t]} |m_{s}(t+t'-\sigma_x)-\gamma_M|\leq \ep/8.$$
Hence, for all $t' \ge t'_0$,
\beq\label{err2_2}
\E_0\left[\sum_{x\in \mathscr{I}(t)\backslash \mathscr{I}(t,c)}e^{-\sigma_x}\sup_{s\in [2\kappa_0 t]}\left|m_{s}(t+t'-\sigma_x)-\gamma_M\right|\right]\leq (\ep/8) \E[w_t]=\ep/8.
\eeq
Combining \eqref{err2_1} and \eqref{err2_2} gives for all $t' \ge t'_0$,
\beq\label{err2}
\E_0\left[\sum_{x\in \mathscr{I}(t)\backslash \mathscr{I}(t,c)}e^{-\sigma_x}\left|m_{\rho(x)}(t+t'-\sigma_x)-\gamma_M\right|\right]\leq \ep/4.
\eeq

Using \eqref{err1}, \eqref{err3}, \eqref{err4}, \eqref{err2} in \eqref{fourerrs}, we conclude that, for any given $\ep>0$, there exists $t_0(\ep) \in \RR_+$ such that for all $t \ge t_0(\ep)$,
$$\E| \zeta_{t}-\gamma_M W|\leq \ep.$$
Thus we have established the convergence of $e^{-t}Z^\psi(t)$ to $\gamma W$ in $\mathbb{L}^1$. The proof is complete.

\end{proof}

\section{Proofs: Local weak convergence}
\label{sec:fringe-proof}

The goal of this Section is to prove Theorem \ref{thm:fringe}. Recall the continuous time version of the process $\cT(\cdot, \vp)$ in Definition \ref{defn:cts-time}; here we label vertices as $\set{v_i:i\geq 0}$ in the order they enter the system starting with $v_0$ and recall that $\sigma_i$ denotes the birth time of vertex $v_i$. For the rest of the Section we will suppress dependence on $\vp$. We will start by studying asymptotics of empirical functionals of the fringe distribution of this process as $t\to\infty$ and then leverage these results to prove Theorem \ref{thm:fringe}. 

Let $\bcF:=\set{\FF_t:t\geq 0}$ denote the natural filtration of the process $\set{\cT(t): t\geq 0}$.  Recall the space $\bT$ of finite rooted trees from Section \ref{sec:fri-decomp}.  For any $\TT\in \mathbb{T}$ and vertex $v\in \mathbb{V}(\TT)$, we will denote by $\TT_v$ the fringe at vertex $v$ namely the subtree rooted at $v$, consisting of all vertices in $\TT$ whose path to the root of $\TT$ passes through $v$. Let $\phi:\mathbb{T} \to \RR$ denote a non-negative bounded measurable function.   For any $i\geq 0$, define $\phi_i:\RR_+ \to \RR$ by,
$$\phi_i(u)=
\begin{cases}
\phi(\TT_{v_i}(u+\sigma_i)), & u\geq 0,\\
0, &s <0.
\end{cases}
$$
Thus the stochastic process $\set{\phi_i(u):u\geq 0}$ tracks the evolution of the ``score'' of the fringe tree below $v_i$ as the age of $v_i$ increases.  Now note that for any time $s >0$ and for $v_i$ such that $\sigma_i\leq s$, the age of vertex $v_i$ at time $s$ is $s-\sigma_i$. 
Write $\cZ^\phi(s)=\sum_{i:\sigma_i\leq s}\phi_i(s-\sigma_i)$; in words we are aggregating the scores of the fringe trees of vertices born before time $s$.  For the rest of this section, fix any $\bcF$ adapted process $\set{a(t):t\geq 0}$  with $a(t)\convas \infty$, $a(t)/t\convas 0$. We will write $y(t) = t+a(t)$.  Recall the process $\cT^*(\cdot, \vp)$ in Definition \ref{def:fringe-limit}. When $\vp$ is clear from context, we will write $\cT^*(\cdot)$ for $\cT^*(\cdot, \vp)$.

\begin{theorem}\label{thm:local_limit-cts-time}
Let $\phi:\mathbb{T} \to \RR$ be a bounded measurable function. Assume that $s \to \E[\phi(\TT^*(s))]$ is Lipschitz on $[0,\infty)$. Then
$$e^{-y(t)}\cZ^\phi(y(t))\xrightarrow{P} W\int_0^\infty e^{-s}\E[\phi(\TT^*(s))]ds \quad \text{ as } t\to\infty,
$$
where $W\sim \mathrm{Exp}(1)$ is the almost sure limit of $e^{-t}n(t)$ as $t \rightarrow \infty$.
\end{theorem}
The main tool to prove this Theorem is the following Proposition. 

\begin{prop}[Conditional moments of $\cZ^\phi(t)$]\label{prop:moments}
In the setting of Theorem \ref{thm:local_limit-cts-time},  as $t\to\infty$,
\begin{align}
\label{first_moment}e^{-y(t)}\E\left[ \cZ^\phi(y(t))|\FF_t\right]&\overset{a.s.}{\longrightarrow} W\int_0^\infty e^{-u}\E[\phi(\TT^*(u))]du,\\
\label{variance}\var\left[ e^{-y(t)}\cZ^\phi(y(t))\big|\FF_t\right]&\overset{a.s.}{\longrightarrow} 0.
\end{align}
\end{prop}

\begin{proof}[{\bf Proof of Theorem \ref{thm:local_limit-cts-time} assuming Proposition \ref{prop:moments}:}]
For any $\delta>0$, write
\begin{align*}
 &\pr\left[ \left|e^{-y(t)}\cZ^\phi(y(t))-W\int_0^\infty e^{-s}\E[\phi(\TT^*(s))]ds\right|>\delta\right]\\
 \leq & \E\left[\pr\left[\left|e^{-y(t)}\cZ^\phi(y(t))-\E[e^{-y(t)}\cZ^\phi(y(t))|\FF_t]\right|>\delta/2\bigg|\FF_t\right]\right]\\
 &\qquad +\pr\left[\left| \E[e^{-y(t)}\cZ^\phi(y(t))|\FF_t]-W\int_0^\infty e^{-s}\E[\phi(\TT^*(s))]ds\right|>\delta/2\right].
 \end{align*}
We assert that each of these two terms converge to zero. For the first term, this follows by applying Chebyshev's inequality with \eqref{variance}, along with the bounded convergence theorem. For the second term this follows from \eqref{first_moment}.
\end{proof}

\subsection{{\bf Proof of Proposition \ref{prop:moments} }:}
We start with the following technical Lemma providing $\cF_t$ measurable approximations of birth times of new individuals after large $t$. Throughout recall the notation $n(t) = |\cT(t)|$. As this Lemma lies at the heart of our local limit computations, we provide an intuitive explanation first. Suppose we know the population size $n(t)$ at time $t$ and want to `guess' the birth time $\alpha_i(t)$ of the $(n(t) + i)$-th individual based on this information. From Lemma \ref{lem:yule-prop}, we know that the population size (which is a Yule process) grows approximately exponentially with time. Therefore, a good guess is $\alpha_i(t)$ satisfying $n(t)e^{\alpha_i(t)-t} = n(t)+i$. The following lemma shows that, for sufficiently large $t$, this is a uniformly good guess for all birth times after $t$.

\begin{lemma}\label{birthtime}
Let $\{\tilde{\sigma}_{t,i} : i\geq 1\}$ denote the birth times after time $t$, i.e., $\tilde{\sigma}_{t,i}=\sigma_{i+n(t)}$. Define $\alpha_i(t):=t+\log\left(1+\frac{i}{n(t)}\right)$ for $i\geq 1$. Then
\beq\label{birthtime_as}
\sup_{i\geq 1} |\tilde{\sigma}_{t,i}-\alpha_i(t)| \overset{a.s.}{\longrightarrow}0 \quad\text{ as }\quad t\to\infty.
\eeq
Moreover,
\beq\label{birthtime_exp}
\E\left[ \sup_{i\geq 1} |\tilde{\sigma}_{t,i}-\alpha_i(t)| \big|\FF_t\right] \overset{a.s.}{\longrightarrow}0 \quad\text{ as }\quad t\to\infty.
\eeq
\end{lemma}

\begin{proof}

\eqref{birthtime_as} follows from \cite[Lemma 3.4]{garavaglia2018trees}.

It remains to prove the convergence in expectation. As before, let $\{E_j:j\geq 0\}$ denote an i.i.d. sequence of exponential rate one random variables independent of $\FF_t$. The dynamics of the Yule process implies that conditional on $\cF_t$, all the joint distributions can be constructed simultaneously as
\[\tilde{\sigma}_{t,i} = t+\sum_{j=0}^{i-1} \frac{E_j}{n(t)+j}, \qquad i\geq 1.   \]
Thus, 
$$|\tilde{\sigma}_{t,i}-\alpha_i(t)|\leq \bigg| \sum_{j=0}^{i-1} \left(\frac{E_j}{n(t)+j}-\frac{1}{n(t)+j}\right)\bigg|+\bigg |\sum_{j=0}^{i-1} \frac{1}{n(t)+j}-\log\left(\frac{n(t)+i}{n(t)}\right)\bigg|.$$
It is easy to see that there exists some constant $C>0$ such that the second term is upper bounded by $\frac{C}{n(t)}$ for all $i \ge 1$ and all $t\geq 0$. Since $\{E_j\}_{j\geq 0}$ are independent of $\FF_t$, defining $\mathcal{M}_i:=\sum_{j=0}^{i-1} \frac{E_j-1}{n(t)+j}$, the sequence $\set{\mathcal{M}_i:i\geq 1}$ conditioned on $\FF_t$ is a martingale. By Cauchy-Schwarz inequality and Doob's $L^2$ inequality, for any $k \ge 1$,
\begin{align*}
\E\left[ \sup_{1\leq i\leq k} \bigg| \sum_{j=0}^{i-1} \frac{E_j-1}{n(t)+j} \bigg|\big|\FF_t\right]&=\E\left[\sup_{1\leq i\leq k}|\mathcal{M}_i|\big|\FF_t\right]\leq C' \left(\E[\mathcal{M}_k^2|\FF_t]\right)^{1/2}\\
&=C'\left( \sum_{j=0}^{k-1} \frac{1}{(n(t)+j)^2}\right)^{1/2}\leq C'\left( \sum_{j=0}^{\infty} \frac{1}{(n(t)+j)^2}\right)^{1/2}\\
&\leq \frac{C''}{\sqrt{n(t)}}
\end{align*}
for some positive constants $C',C''$ not depending on $k$. Therefore,
$$\E\left[ \sup_{i\geq 1} |\tilde{\sigma}_{t,i}-\alpha_i(t)| \big|\FF_t\right]\leq \frac{C}{n(t)}+\frac{C''}{\sqrt{n(t)}}\overset{a.s.}{\longrightarrow}0 \quad \text{as }t\to\infty.$$
\end{proof}

We will now commence on the proof of Proposition \ref{prop:moments}. 

\begin{proof}[{\bf Proof of \eqref{first_moment} }]
For the rest of this proof, let $m(s):=\E[\phi(\TT^*(s))]$ for $s\geq 0$ and $m(s)=0$ for $s<0$. By boundedness assumption on $\phi$ and Lipschitz continuity of $m(\cdot)$, there exists constants  $M, L <\infty$ such that 
\begin{equation}
    \label{eqn:lip-L-M}
  |\phi(\cdot)|\leq M, \qquad  |m(s) - m(t)|\leq L|t-s|. 
\end{equation}
Note that
$\cZ^\phi(y(t))=\sum_{i: \sigma_i\leq t} \phi_i(y(t)-\sigma_i)+ \sum_{i: \sigma_i\in (t,y(t)]} \phi_i(y(t)-\sigma_i).$ To simplify notation, we will write the second term as $\tilde{\cZ}^{\phi}(y(t)):= \sum_{i: \sigma_i\in (t,y(t)]} \phi_i(y(t)-\sigma_i)$. By \eqref{eqn:lip-L-M}, $|\phi_i(\cdot)|\leq M$ for all $i\geq 0$. It is straightforward then, observing from Lemma \ref{lem:yule-prop} that $e^{-t}n(t)$ converges almost surely to a finite random limit,
\beq\label{err_beforet}
e^{-y(t)}\E\left[  \sum_{i: \sigma_i\leq t} \phi_i(y(t)-\sigma_i)\big |\FF_t\right]\leq M\cdot \E[ e^{-y(t)} n(t)|\FF_t] \overset{a.s.}{\longrightarrow} 0 \quad \text{as }t\to\infty.
\eeq
To address the second term, recalling from Lemma \ref{birthtime} that $\tilde{\sigma}_{t,i}=\sigma_{i+n(t)}$, write
\beq\label{cal1}
\E[ \tilde{\cZ}^\phi(y(t))|\FF_t]=\E\left[ \sum_{i: \tilde{\sigma}_{t,i}\leq y(t)} m(y(t)-\tilde{\sigma}_{t,i})|\FF_t\right].
\eeq
Applying Lemma \ref{birthtime} we can approximate \eqref{cal1} by using the $\cF_t$-measurable approximations $\{\alpha_i(t)\}_{i \geq 1}$ of the birth times $\{\tilde{\sigma}_{t,i}\}_{i \ge 1}$, 
\begin{align}\label{diff0}
\nonumber & e^{-y(t)} \bigg|\E\left[ \sum_{i: \tilde{\sigma}_{t,i}\leq y(t)} m(y(t)-\tilde{\sigma}_{t,i})-\sum_{i: \alpha_i(t)\leq y(t)} m(y(t)-\tilde{\sigma}_{t,i})\big|\FF_t\right] \bigg|\\
\nonumber \leq &  e^{-y(t)} \cdot M\E\left[\sum_i 1_{ \{ \tilde{\sigma}_{t,i}\leq y(t), \alpha_i(t)> y(t)\}}|\FF_t\right]\\
 \nonumber \leq & M\cdot e^{-y(t)} \E\left[ |n(y(t))-\lfloor n(t)(e^{a(t)}-1)\rfloor|\big|\FF_t\right]\\
\nonumber \leq & M \cdot \left( \E\left[ | e^{-y(t)}n(y(t))- e^{-t}n(t)|\big|\FF_t\right] + e^{-y(t)}(n(t)+1)\right)\\
\leq &M \left(\E\left[ ( e^{-y(t)}n(y(t))- e^{-t}n(t))^2\big|\FF_t\right]\right)^{1/2}+Me^{-y(t)}(n(t)+1).
\end{align}

Since $\set{e^{-s}n(s):s\geq 0}$ is an $\bL^2$ bounded martingale we can compute 
$$\E\left[ ( e^{-y(t)}n(y(t))- e^{-t}n(t))^2\big|\FF_t\right]=\E\left[ (e^{-y(t)}n(y(t)))^2|\FF_t\right]-(e^{-t}n(t))^2\leq e^{-2t}n(t),$$
where the last inequality follows from Lemma \ref{mart_power}. Therefore,
\beq\label{diff1}
\eqref{diff0}\leq 2M e^{-t}\sqrt{n(t)}+2Me^{-y(t)}(n(t)+1). 
\eeq

Now we estimate
\begin{align}\label{diff2}
\nonumber &e^{-y(t)}\bigg| \E\left[ \sum_{i: \alpha_i(t)\leq y(t)} m(y(t)-\tilde{\sigma}_{t,i})-\sum_{i: \alpha_i(t)\leq y(t)} m(y(t)-\alpha_i(t)) \big|\FF_t\right] \bigg|\\
\leq& e^{-y(t)} n(t)e^{a(t)} \cdot L \E\left[ \sup_{i\geq 1} |\tilde{\sigma}_{t,i}-\alpha_i(t)| \big|\FF_t\right]=e^{-t}n(t)\cdot  L \E\left[ \sup_{i\geq 1} |\tilde{\sigma}_{t,i}-\alpha_i(t)| \big|\FF_t\right],
\end{align}
where $L$ is as in \eqref{eqn:lip-L-M}. Next, let $\mathcal{M}^\phi(y(t)):=\sum_{i: \alpha_i(t)\leq y(t)}m(y(t)-\alpha_i(t))$. Combining \eqref{diff1} and \eqref{diff2} yields
\begin{align}\label{diff3}
\nonumber &e^{-y(t)}\bigg| \E[ \tilde{\cZ}^\phi(y(t))|\FF_t]-\E[\mathcal{M}^\phi(y(t))|\FF_t]\bigg|\\
\leq &2M(e^{-t}\sqrt{n(t)}+ e^{-y(t)}(n(t)+1))+e^{-t} n(t)\cdot L \E\left[ \sup_{i\geq 1} |\tilde{\sigma}_{t,i}-\alpha_i(t)| \big|\FF_t\right].
\end{align}
The part $e^{-t}\sqrt{n(t)}+ e^{-y(t)}(n(t)+1)$ vanishes almost surely as $t\to\infty$. Then applying Lemma \ref{birthtime} gives that \eqref{diff3} goes to 0 almost surely as $t\to\infty$. It remains to evaluate the almost sure limit of $e^{-y(t)}\E[\mathcal{M}^\phi(y(t))|\FF_t]$. Observe that, using the boundedness and Lipschitz continuity of $m(\cdot)$, 
\begin{align}\label{Mphi}
\nonumber &e^{-y(t)}\E[\mathcal{M}^\phi(y(t))|\FF_t]\\
\nonumber=& e^{-y(t)} \sum_{i: \alpha_i(t)\leq y(t)}m(y(t)-\alpha_i(t))= e^{-y(t)} \sum_{i=1}^{ \lfloor n(t)e^{a(t)}\rfloor} m\left( a(t)-\log\left( \frac{n(t)+i}{n(t)}\right)\right) +o_{a.s.}(1)\\
\nonumber=&e^{-y(t)} \sum_{i=1}^{\lfloor e^{a(t)}\rfloor} \sum_{j=1}^{n(t)} m\left( a(t)-\log\left(i+\frac{j}{n(t)}\right)\right)+o_{a.s.}(1)=e^{-y(t)} \sum_{i=1}^{\lfloor e^{a(t)}\rfloor} n(t)m(a(t)-\log i)+o_{a.s.}(1)\\
\nonumber=&e^{-y(t)}n(t) \int_1^{e^{a(t)}} m(a(t)-\log x) dx+o_{a.s.}(1)=e^{-y(t)}n(t)  \int_0^{a(t)} m(u)e^{a(t)-u} du +o_{a.s.}(1)\\
=&e^{-t}n(t)\int_0^{a(t)}m(u)e^{-u} du+o_{a.s.}(1) \overset{a.s.}{\longrightarrow} W\int_0^\infty e^{-u} m(u)du.
\end{align}
This proves \eqref{first_moment}.

\end{proof}

\begin{proof}[{\bf Proof of \eqref{variance} }]
Note that 
\begin{align*}
\E\left[ (\cZ^\phi(y(t)))^2\big|\FF_t\right]&=\E\left[\sum_{i,j:\sigma_i\leq t, \sigma_j\leq t} \phi_i(y(t)-\sigma_i)\phi_j(y(t)-\sigma_j)\big|\FF_t\right]\\
&\quad +2\E\left[\sum_{i,j: \sigma_i\leq t< \sigma_j\leq y(t)}  \phi_i(y(t)-\sigma_i)\phi_j(y(t)-\sigma_j)\big| \FF_t\right]\\
&\quad + \E\left[ \sum_{i,j} \phi_i(y(t)-\tilde{\sigma}_{t,i})\phi_j(y(t)-\tilde{\sigma}_{t,j}) \big|\FF_t\right]\\
&=:T_1(t)+2T_2(t)+T_3(t),
\end{align*}
where for the rest of this Section,  the sum in $T_3(t)$ is implicitly over $i,j$ with $\sigma_i, \sigma_j \in (t,y(t)]$ but we have suppressed this to ease notation. Now,  it can be readily checked that 
$$e^{-2y(t)}T_1(t)\leq M^2e^{-2y(t)} (n(t))^2 \overset{a.s.}{\longrightarrow} 0\quad \text{ as }t\to\infty,$$
and 
$$e^{-2y(t)}T_2(t)\leq Me^{-y(t)}n(t)\cdot e^{-y(t)} \E\left[\tilde{\cZ}^\phi(y(t))\big|\FF_t\right]\overset{a.s.}{\longrightarrow} 0\quad \text{ as }t\to\infty,$$
where the second line follows from combining Lemma \ref{lem:yule-asymp} with \eqref{first_moment}.  Thus $\E\left[ (e^{-y(t)}\cZ^\phi(y(t)))^2\big|\FF_t\right]=e^{-2y(t)}T_3(t)+o_{a.s.}(1)$.
We write (again recalling that all the ensuing sums are over vertices born in $(t, y(t)]$), 
\begin{align*}
T_3(t)&=\E\left[ \sum_i \phi_i^2(y(t)-\tilde{\sigma}_{t,i}) \big|\FF_t\right]+ 2\E\left[ \sum_{i<j}  \phi_i(y(t)-\tilde{\sigma}_{t,i})\phi_j(y(t)-\tilde{\sigma}_{t,j}) \big|\FF_t\right] = \eps(t) + \tilde{T}_3(t).
\end{align*}
We restrict our attention to the second term since the first term is $o_{a.s.}(e^{-2y(t)})$: 
\beq\label{T3err0}
e^{-2y(t)}\eps(t):= e^{-2y(t)}\E\left[ \sum_i \phi_i^2(y(t)-\tilde{\sigma}_{t,i}) \big|\FF_t\right]\leq M^2e^{-2y(t)} n(t)e^{a(t)}  \overset{a.s.}{\longrightarrow} 0\quad \text{ as }t\to\infty.
\eeq

For $i<j$, we write $j\to i$ if $v_j$ is a descendant of $v_i$ (and write $j\nrightarrow i$ otherwise). By convention we have $i\to i$. Then for $i<j$,
\begin{align}\label{T3decomp}
\nonumber&\E\left[\phi_i(y(t)-\tilde{\sigma}_{t,i})\phi_j(y(t)-\tilde{\sigma}_{t,j}) \big|\FF_t\right]\\
\leq &M^2\pr\left[ j\to i, \tilde{\sigma}_{t,j}\leq y(t)\big|\FF_t\right]+\E\left[ \phi_i(y(t)-\tilde{\sigma}_{t,i})\phi_j(y(t)-\tilde{\sigma}_{t,j})\ind_{\{j\nrightarrow i\}}\big|\FF_t\right].
\end{align}
Recall that, for the rest of the argument, we are only interested in pairs born in the interval $(t, y(t)]$. For any fixed time $T$, vertex $v_i$ with $\sigma_i\leq T$, let $\cD_v(T)$ denote the number of descendants of $v$ by time $T$. Summing the bound \eqref{T3decomp} over all pairs of vertices born in the interval $(t,y(t)]$, one gets, 
\begin{align}
    \tilde{T}_3(t) \leq 2M^2 \E[\sum_{i: \,\sigma_i\in (t,y(t)]} & \cD_{v_i}(y(t))|\cF_t] + 2\E[\sum_{\substack{i< j: \, \sigma_i,   \sigma_j \in (t, y(t)]}} \phi_i(y(t)-\tilde{\sigma}_{t,i})\phi_j(y(t)-\tilde{\sigma}_{t,j})\ind_{\{j\nrightarrow i\}}|\cF_t]\nonumber \\
    & := \tilde{T}_{3,1}(t) + \tilde{T}_{3,2}(t) \label{eqn:T3-final-dcomp}.
\end{align}
The following lemma completes the proof of \eqref{variance}.
\begin{lemma}
\label{lem:t3-analy}
As $t\to\infty$, 
\begin{enumeratea}
\item $e^{-2y(t)} \tilde{T}_{3,1}(t) \convas 0$. 
\item $e^{-2y(t)} \tilde{T}_{3,2}(t) \leq ( e^{-y(t)}\E[\cZ^\phi(y(t))|\FF_t])^2+o_{a.s.}(1)$.
\end{enumeratea}

\end{lemma}
\begin{proof}[Proof of Lemma \ref{lem:t3-analy}(a) ]
Note that, when a new vertex is attached to some existing vertex $v$, the number of descendants of $v$ and all its ancestor vertices increases by one. Hence, we can dominate $\sum_{i: \sigma_i\in (t,y(t)]} \cD_{v_i}(y(t))$ pathwise by $\height(\cT(y(t),\mathbf{p}))n(y(t))$ where recall that $\height(\cdot)$ denotes the height of the associated tree. Moreover, observe that, conditionally on $\cF_t$, $\height(\cT(y(t),\mathbf{p})) \le \height(\cT(t,\mathbf{p})) + \operatorname{max}_{i \le n(t)}h_i(t,y(t))$, where $h_i(t,y(t))$ denotes the height of the (maximal) tree rooted at the $i$th vertex formed entirely by its descendants that arrived in the time interval $(t,y(t)]$. Further, note that there exists a collection 
$\{h^*_i(a(t)): i \le n(t)\}$, distributed as the heights of $n(t)$ independent Yule trees run till time $a(t)$, independent of $\cF_t$, so that we can couple to get $h_i(t,y(t)) \le h^*_i(a(t))$ for every $i \le n(t)$. Consequently, using Lemma \ref{diam_tail} with $\beta = e$,
\begin{align*}
&\E\left((\operatorname{max}_{i \le n(t)}h_i(t,y(t)))^2 \big| \cF_t\right) \le n(t) \E((h^*_1(a(t))^2) \le n(t)\left(9e^2a(t)^2 + \int_{9e^2a(t)^2}^{\infty}\prob(h^*_1(a(t)) \ge \sqrt{x})dx\right)\\
&\le n(t)\left(9e^2a(t)^2 + \int_{9e^2a(t)^2}^{\infty} 2e^{2ea(t)} e^{-\sqrt{x}}dx\right) = 9e^2a(t)^2n(t) + 2n(t)(1+ 3ea(t))e^{-ea(t)}.
\end{align*}
This implies that $e^{-2y(t)}\E\left((\operatorname{max}_{i \le n(t)}h_i(t,y(t)))^2 \big| \cF_t\right) \rightarrow 0$ almost surely as $t \rightarrow \infty$. 
Also, by Lemmas \ref{lem:yule-asymp} and \ref{mart_power},
$$
e^{-2y(t)}\E(n(y(t))^2 \big| \cF_t) \le e^{-2t}(n(t)^2 + n(t)) \convas W^2
$$
as $t \rightarrow \infty$. Moreover, using the monotonicity of $\height(\cT(t,\mathbf{p}))$ in $t$ and Lemma \ref{diam_tail}, it follows that $e^{-y(t)}\height(\cT(t,\mathbf{p})) \rightarrow 0$ almost surely as $t \rightarrow \infty$. Therefore, using Cauchy-Schwarz inequality,
\begin{align*}
e^{-2y(t)}\tilde{T}_{3,1}(t)& \le e^{-2y(t)}\E[\height(\cT(t,\mathbf{p}))n(y(t))|\cF_t]+e^{-2y(t)}\E[\operatorname{max}_{i \le n(t)}h_i(t,y(t)))^2 n(y(t))|\cF_t]\\
&\le e^{-2y(t)}\height(\cT(t,\mathbf{p}))n(t)e^{a(t)}\\
&\qquad + \left(e^{-2y(t)}\E\left((\operatorname{max}_{i \le n(t)}h_i(t,y(t)))^2 \big| \cF_t\right)\right)^{1/2}\left(e^{-2y(t)}\E(n(y(t))^2 \big| \cF_t)\right)^{1/2},
\end{align*}
which converges almost surely to $0$ as $t \rightarrow \infty$.

\end{proof}
\begin{proof}[Proof of Lemma \ref{lem:t3-analy}(b) ]
For $i < j$, let $\FF^*_{i,j}$ be the $\sigma$-field generated by the tree process up to the $(n(t)+j)$-th birth time, and the birth times and attachment locations of all vertices that are descendants of $v_i$. Note that $\phi_i(y(t)-\tilde{\sigma}_{t,i})1_{\{j\nrightarrow i\}}$ is $\FF^*_{i,j}$-measurable and $\phi_j(y(t)-\tilde{\sigma}_{t,j})$ is independent of $\FF^*_{i,j}$ on the event $\{j\nrightarrow i\}$. Hence first conditioning on $\FF^*_{i,j}$ and then using the tower property for conditional expectations we get,
\begin{align}
\E\left[\phi_i(y(t)-\tilde{\sigma}_{t,i})\phi_j(y(t)-\tilde{\sigma}_{t,j}) \ind_{\{j\nrightarrow i\}}\big|\FF_t\right]&=\E\left[\phi_i(y(t)-\tilde{\sigma}_{t,i})\ind_{\{j\nrightarrow i\}}m(y(t)-\tilde{\sigma}_{t,j}) \big|\FF_t\right] \nonumber\\
&\leq \E\left[m(y(t)-\tilde{\sigma}_{t,i})m(y(t)-\tilde{\sigma}_{t,j}) \big|\FF_t\right]. \label{eqn:549}
\end{align}
The last inequality above follows from a similar conditioning by a sigma field containing information about birth times and attachment locations of all individuals except the descendants of $i$ (excluding $i$).

By the boundedness and Lipschitz assumption on the mean functional $m(\cdot)$, for $i<j$, where $j$ satisfies $\max\{ \tilde{\sigma}_{t,j}, \alpha_j(t)\}\leq y(t)$, there exists some $C>0$ such that
$$|m(y(t)-\tilde{\sigma}_{t,i})m(y(t)-\tilde{\sigma}_{t,j})-m(y(t)-\alpha_i(t))m(y(t)-\alpha_j(t))|\leq C( |\tilde{\sigma}_{i,t}-\alpha_i(t)|+|\tilde{\sigma}_{j,t}-\alpha_j(t)|).$$
Hence (and writing $\E_{\cF_t}(\cdot) = \E(\cdot|\cF_t)$),
\begin{align}
\label{eqn:813} &e^{-2y(t)}\Bigg| \E_{\cF_t}\left[ \sum_{\substack{i< j: \\ \tilde{\sigma}_{t,i},   \tilde{\sigma}_{t,j} \in (t, y(t)]}} m(y(t)-\tilde{\sigma}_{t,i})m(y(t)-\tilde{\sigma}_{t,j})\right.\nonumber\\
&\qquad\qquad\qquad\qquad\left. -\sum_{\substack{i< j: \\ {\alpha}_{i}(t), {\alpha}_{j}(t) \in (t, y(t)]}} m(y(t)-\alpha_i(t))m(y(t)-\alpha_j(t))\right]\Bigg|\\
\label{T3err2}\leq & e^{-2y(t)}\cdot C\E_{\cF_t}\left[ \sum_{i<j: \max\{ \tilde{\sigma}_{t,j}, \alpha_j(t)\}\leq y(t)} ( |\tilde{\sigma}_{i,t}-\alpha_i(t)|+|\tilde{\sigma}_{j,t}-\alpha_j(t)|)\right]\\
\label{T3err3}&+ e^{-2y(t)} \cdot M^2 \E_{\cF_t}\left[ \sum_{i<j} 1_{\{ \tilde{\sigma}_{t,j}\leq y(t)<\alpha_j(t)\}}\right] \\
\label{T3err4} &+ e^{-2y(t)} \cdot M^2 \E_{\cF_t}\left[ \sum_{i<j} 1_{\{\alpha_j(t)\leq y(t) < \tilde{\sigma}_{t,j}\}}\right].
\end{align}
By Lemma \ref{birthtime},
\begin{align*}
\eqref{T3err2}&\leq e^{-2y(t)}\cdot 2C \E\left[ \sum_{i<j: \max\{ \tilde{\sigma}_{t,j}, \alpha_j(t)\}\leq y(t)} \sup_{\ell\geq 1}|\tilde{\sigma}_{t,\ell}-\alpha_\ell(t)| \big|\FF_t\right]\\
&\leq 2C e^{-2y(t)} (n(t))^2 e^{2a(t)} \E\left[ \sup_{\ell\geq 1}|\tilde{\sigma}_{t,\ell}-\alpha_\ell(t)| \big|\FF_t\right]\\
&=2C (e^{-t}n(t))^2 \E\left[ \sup_{\ell\geq 1}|\tilde{\sigma}_{t,\ell}-\alpha_\ell(t)| \big|\FF_t\right] \overset{a.s.}{\longrightarrow}0.
\end{align*}
As for \eqref{T3err3}, observe that 
\begin{align*}
\eqref{T3err3} & \leq M^2 e^{-2y(t)} \E\left[n(y(t))|n(y(t))-\lfloor n(t)(e^{a(t)}-1)\rfloor|\, \big| \, \FF_t\right]\\
&\leq M^2 e^{-2y(t)}n(t)e^{a(t)} \E\left[|n(y(t))-\lfloor n(t)(e^{a(t)}-1)\rfloor|\, \big| \, \FF_t\right]\\
&\qquad +
 M^2 e^{-2y(t)} \E\left[(n(y(t))-\lfloor n(t)(e^{a(t)}-1)\rfloor)^2\, \big| \, \FF_t\right],
\end{align*}
which goes to zero almost surely as $t \rightarrow \infty$ following the same reasoning as in \eqref{diff1}. Similarly we have 
\begin{align*}
    \eqref{T3err4} &\leq M^2 e^{-2y(t)}\E\left[n(t)(e^{a(t)}-1)\max\{ n(t)(e^{a(t)}-1)-n(y(t)),0\}\big|\FF_t\right]\\
    &\leq M^2e^{-2y(t)}n(t)e^{a(t)}\E\left[|n(y(t))- n(t)(e^{a(t)}-1)|\,\big| \, \FF_t\right] \overset{a.s.}{\longrightarrow} 0.
\end{align*}

Noting that the processes $\alpha_{i}(t)$ are $\cF_t$ adapted so that the conditional expectation of the second sum in \eqref{eqn:813} is itself,  the last step is to estimate 
\begin{align*}
\mathcal{M}^{\phi, (2)}(y(t)):= \sum_{\substack{i<j: \\ {\alpha}_{i}(t), {\alpha}_{j}(t) \in (t, y(t)]}} m(y(t)-\alpha_i(t))m(y(t)-\alpha_j(t))\le \frac{1}{2} \left(  \sum_{i: \alpha_{i}(t)\le  y(t)} m(y(t)-\alpha_i(t))\right)^2.
\end{align*}
Recall the definition of $\mathcal{M}^\phi(y(t))=\sum_{i: \alpha_i(t)\leq y(t)}m(y(t)-\alpha_i(t))$ from the proof of the first moment convergence in \eqref{first_moment}. Then, from \eqref{diff3} and \eqref{err_beforet},
\begin{align}\label{Mphi2}
\nonumber 2e^{-2y(t)}\mathcal{M}^{\phi, (2)}(y(t))&\le \left( e^{-y(t)}\mathcal{M}^\phi(y(t))\right)^2\\
\nonumber &=\left( e^{-y(t)}\E[\tilde{Z}^\phi(y(t))|\FF_t] \right)^2+o_{a.s.}(1)\\
&=\left( e^{-y(t)}\E[\cZ^\phi(y(t))|\FF_t] \right)^2+o_{a.s.}(1).
\end{align}
Using \eqref{eqn:549}, asymptotics for the terms in \eqref{T3err2}, \eqref{T3err3}, \eqref{T3err4} and \eqref{Mphi2} completes the proof.

\end{proof}
This completes the proof of the second moment namely \eqref{variance}. 
\end{proof}

\subsection{Completing the proof of Theorem \ref{thm:fringe}: } The goal now is to transfer the continuous time embedding asymptotics in Theorem \ref{thm:local_limit-cts-time} to the discrete time process $\set{\cT_n(\vp):n\geq 1}$. In Lemma \ref{lem:cts-disc}, recall the stopping times $T_n=\inf\{ t\geq 0: n(t)=n+1\}$, connecting the embedding of the discrete process in continuous time. 

\begin{theorem}\label{fringe}
Let $\phi$ be a non-negative functional on $\bT$ satisfying the assumptions in Theorem \ref{thm:local_limit-cts-time}. Let $v^{(n)}$ be a uniformly chosen vertex in the graph $\cT_n(\vp)$. Let $\TT_{v^{(n)}}$ denote the fringe tree of $v^{(n)}$. Then
$$\E[\phi(\TT_{v^{(n)}})|\cT_n]:= \frac{1}{n+1}\sum_{i=0}^n \phi_i(T_n-\sigma_i) \to  \int_0^\infty e^{-s}\E[\phi(\TT^*(s))]ds \quad \text{ in probability as }n\to\infty.$$
\end{theorem}

\begin{proof}
 Write 
$$\E[\phi(\TT_{v^{(n)}})|\cT_n]=\frac{1}{n+1}\sum_{i=0}^n \phi_i(T_n-\sigma_i):=\frac{1}{n+1}\cZ^\phi(T_n).$$
Let $\Theta(t):=e^{-t}n(t)$. Note that, by Lemma \ref{lem:yule-prop}, $\Theta(t) \convas W$ and thus, 
\begin{equation}
    \label{eqn:tn-as}
T_n-\log n+\log W\overset{a.s.}{\longrightarrow} 0.    
\end{equation}
 However, working the limit random variable $W$ for finite $t$ approximations of the embedding is difficult so we will work with approximations of the limit $W$.  Define $t_n=\log n-\sqrt{\log n}$, $\tilde{T}_n=\log n-\log \Theta({t_n})$. Using \eqref{eqn:tn-as} and $\Theta(t_n) \convas W$, we get,
 $$|\tilde{T}_n-T_n|\leq |\log n-\log W - T_n|+|\log W-\log \Theta(t_n)|\overset{a.s.}{\longrightarrow}0.$$
  For $\ep>0$, define  $T^{\pm}_n(\ep):=\tilde{T}_n\pm \ep$. For any $\ep>0$ there exists a random $n_{\ep}$ such that for all $n\geq n_\ep$, $T^-_n(\ep)\leq T_n \leq T^+_n(\ep)$ almost surely. Since $\phi$ is non-negative, for $n > n_\eps$, 
\beq\label{sandwiching}
e^{-T^+_n(\ep)}\cZ^\phi(T^-_n(\ep))\leq e^{-T_n}\cZ^\phi(T_n)\leq e^{-T^-_n(\ep)}\cZ^\phi(T^+_n(\ep)).
\eeq
Recall the limit random variable in Theorem \ref{thm:local_limit-cts-time} and to simplify notation let $X:=W\int_0^\infty e^{-s}\E[\phi(\TT^*(s)]ds$. Also recall that $W\sim \mathrm{Exp}(1)$. Fix $\delta>0$. For any $\eta>0$ we can take $\ep>0$ such that,
\begin{align}\label{Xcomparison}
\nonumber&\pr\left( (X+\delta)e^{-2\ep}>X+\delta/2, (X-\delta)e^{2\ep}<X-\delta/2\right)\\
=&\pr\left( X<\frac{\delta (e^{2\ep}-1/2)}{e^{2\ep}-1}\wedge \frac{\delta (e^{-2\ep}-1/2)}{1-e^{-2\ep}}\right)\geq 1-\eta.
\end{align}
Write $A_{\delta,\ep}=\{ (X+\delta)e^{-2\ep}>X+\delta/2, (X-\delta)e^{2\ep}<X-\delta/2\}$.
\begin{align}\label{limsup_comp}
\nonumber &\limsup_{n\to\infty} \pr(e^{-T_n}\cZ^\phi(T_n)>X+\delta)\\
\nonumber\leq &\limsup_{n\to\infty} \pr( e^{-T^-_n(\ep)}\cZ^\phi(T^+_n(\ep))>X+\delta)+\limsup_{n\to\infty} \pr(e^{-T_n}\cZ^\phi(T_n)> e^{-T^-_n(\ep)}\cZ^\phi(T^+_n(\ep)))\\
\nonumber =&\limsup_{n\to\infty} \pr( e^{2\ep}e^{-T^+_n(\ep)}\cZ^\phi(T^+_n(\ep))>X+\delta)\\
\leq &\limsup_{n\to\infty} \pr( e^{-T^+_n(\ep)}\cZ^\phi(T^+_n(\ep))>X+\delta/2,A_{\delta,\ep})+\eta,
\end{align}
where the third line comes from \eqref{sandwiching} and the last line follows from \eqref{Xcomparison}. Note that $T_n^\pm(\ep)=t_n+(\sqrt{\log n}-\log\Theta(t_n)\pm \ep)$ is measurable with respect to $\FF_{t_n}$ and 
$$\sqrt{\log n}-\log\Theta({t_n})\pm \ep \overset{a.s.}{\longrightarrow}\infty \quad\text{and}\quad \frac{\sqrt{\log n}-\log\Theta({t_n})\pm \ep}{t_n}\overset{a.s.}{\longrightarrow}0.$$
Thus applying Theorem \ref{thm:local_limit-cts-time}, we get that for any $\delta,\ep>0$, as $n\to\infty$,
$$e^{-T^\pm_n(\ep)}\cZ^\phi(T^\pm_n(\ep)) \to X \quad \text{in probability}.$$
Applying this back to \eqref{limsup_comp} shows that $\limsup_{n\to\infty} \pr(e^{-T_n}\cZ^\phi(T_n)>X+\delta)\leq \eta$ for arbitrary $\eta>0$, i.e., $\limsup_{n\to\infty} \pr(e^{-T_n}\cZ^\phi(T_n)>X+\delta)=0$. Following a similar argument we have $\limsup_{n\to\infty} \pr(e^{-T_n}\cZ^\phi(T_n)<X-\delta)=0$, thus establishing 
\[e^{-T_n}\cZ^\phi(T_n) \probc X.\]
Since $e^{-T_n}\cdot (n+1) \convas W$, combining this with the above equation gives, 
$$\E[\phi(\TT_{v^{(n)}})|\cT_n]=\frac{1}{n+1}\cZ^\phi(T_n)=\frac{\cZ^\phi(T_n)}{e^{T_n}}\cdot \frac{e^{T_n}}{n+1}\to  \int_0^\infty e^{-s}\E[\phi(\TT^*(s))]ds \quad \text{in probability}.$$
\end{proof}

\begin{proof}[{\bf Proof of Theorem \ref{thm:fringe}}]
To complete the proof of part (a), namely, the convergence in probability in the fringe sense, it suffices to show that for any fixed finite rooted tree $\mathbf{s}_0$, the function $m(u) := \pr\left(\cT^*(u) = \mathbf{s}_0\right), \, u \ge 0,$ is Lipschitz continuous in $u$. But this follows upon noting that there exists a finite positive constant $C(|\mathbf{s}_0|)$ depending only on the size of $\mathbf{s}_0$ such that, for any $u>0$, $\delta>0$,
\begin{align*}
&\left|\pr\left(\cT^*(u+\delta) = \mathbf{s}_0\right) - \pr\left(\cT^*(u) = \mathbf{s}_0\right) \right|\\
&\le \sum_{\mathbf{s} \subseteq \mathbf{s}_0}\pr(\text{There is a birth in the Yule process with initial population size } |\mathbf{s}| \text{ before time } \delta)\\
&\le C(|\mathbf{s}_0|) \delta,
\end{align*}
where the sum above is over all rooted subtrees of $\mathbf{s}_0$.

Next we prove part (c). Observe that, from part (a), $\E[D] = \E\left(\cP_1(\tau)\right)$, where $\cP_1(\cdot)$ is defined in Section \ref{sec:quasi} and $\tau$ is an independent $\mathrm{Exp}(1)$ random variable. Using Lemma \ref{Pk},
$$
\E[D] = \E(\cP_1(\tau)) = \sum_{i=0}^{\infty}\frac{\E(\tau^i)}{i!}\pr(\barT_1 = i) = \sum_{i=0}^{\infty}\pr(\barT_1 = i) = \pr(\barT_1 < \infty),
$$
where we have used $\E[\tau^i]=i!$ to obtain the second equality. By standard results on recurrence of random walks, eg. see the Remark after Lemma 1 in \cite{pakes1973conditional}, the last term above is $1$ if and only if $\E[Z] \le 1$. This proves the claimed assertions on $\E[D]$. Further, again using Lemma \ref{Pk},
$$
\E(|\cT^*(\tau,\vp)|) = \sum_{k=0}^{\infty}\E(\cP_k(\tau)) = \sum_{k=0}^{\infty}\sum_{i=0}^{\infty}\frac{\E(\tau^i)}{i!}\pr(\barT_k = i) = \sum_{k=0}^{\infty}\pr(\barT_k < \infty) = 1 + \sum_{k=1}^{\infty}(\pr(\barT_1 < \infty))^k.
$$
The right hand side above is finite if and only if $\E[D] = \pr(\barT_1 < \infty) <1$, proving the assertions on expected tree size.

Part (b) now follows from Theorem \ref{thm:aldous} (a).
\end{proof}

\section{Proofs: Degree distribution asymptotics}
\label{sec:proof-deg-dist}
In this section, we prove Theorem \ref{thm:deg-dist}. The high correlation in the evolution of degrees of different vertices renders conventional tools inapplicable and one has to develop new stochastic analytic techniques to track the degree evolution. The proofs of the degree distribution upper bounds in Theorem \ref{thm:deg-dist} rely crucially on the asymptotics of a weighted linear combination of vertex counts at different distances from the root of $\cT^*$, summarized in Theorem \ref{updegree}. This theorem also plays a key role in subsequent sections involving fixed vertex degree asymptotics and PageRank asymptotics. The lower bounds in Theorem \ref{thm:deg-dist} rely on a softer analysis involving approximation by multitype branching processes with finitely many types using tools developed in Section \ref{sec:urn-model-def}.

Recall the process $\bcP(\cdot) = (\cP_i(\cdot):i\geq 0)$ from Section \ref{sec:quasi} keeping track of the number of vertices at various levels in the process $\cT^*$. Let $\tau\sim \mathrm{Exp}(1)$ be an independent random variable. The limit in Theorem \ref{thm:fringe} now results in the following description of the limit degree distribution. 

\begin{corollary}\label{corr: limit_degree}
Let $D$ be as in Theorem \ref{thm:fringe}. Then $D\overset{d}{=} \cP_1(\tau)$.
\end{corollary}
Thus understanding the evolution of $\bcP$ will play a key role in the proof of Theorem \ref{thm:deg-dist}. We will begin by stating Theorem \ref{updegree}. Assuming this Theorem, we will prove Theorem \ref{thm:deg-dist}. The rest of the Section compartmentalized in Section \ref{sec:proof-updegree} is then devoted to the proof of Theorem \ref{updegree}. Recall the matrix $\vA = (A_{ij})$ as in \eqref{eqn:matrix-def}. With $s_0$ as in Definition \ref{def:R-q-s0} let,  
\begin{equation}
    \label{eqn:cpst}
    \cpst_s(t)=\sum_{i=1}^\infty s^{-i}\cP_i(t), \qquad \mbox{for } t\geq 0,s>0, \qquad \cpst(t) \equiv \cpst_{s_0}(t). 
\end{equation}
Recall the p.g.f of $\vp$, $f(\cdot)$. 
\begin{theorem}\label{updegree}
\begin{enumeratei}
\item For any $s>0$ such that $f(s)<\infty$ and for all $t\geq 0$,
\beq\label{Ptub}
\E[\cpst_s(t)]\leq \frac{p_0}{f(s)} e^{\frac{f(s)}{s}t}. 
\eeq
\item When $\E[Z]\leq 1$, for any $s\in [1, s_0], \theta\geq 1$, there exists constant $C_{\theta,s} < \infty$ such that,
\beq\label{moment_fringe}
\E[(\cpst_s(t))^\theta]\leq C_{\theta,s}e^{\frac{f(s)}{s}\theta t}, \qquad \forall t> 0. 
\eeq
\item When $\E[Z]>1$ so that $s_0<1$ by Lemma \ref{lemma:prop-R-s0}, let $\alpha^*(\theta):=\max\{ \frac{f(s_0^\theta)}{s_0^\theta},\frac{\theta}{R}\}$. For any $\theta\geq 1$, there exists some constant $C_\theta>0$ such that for all $t\geq 0$,
\begin{align}
\label{moment_nonfringe}\E[(\cpst(t))^\theta]&\leq C_\theta (1+t^\theta)e^{\alpha^*(\theta)t}.
\end{align}
\end{enumeratei}

\begin{corollary}\label{updegree_root}
Define $\tilde\cpst_s(t):=\sum_{i=1}^{\infty} s^{-i}\tilde\cP_i(t)$. 
\begin{enumeratei}
\item
For all $t\geq 0$ and $s\geq 1$ such that $f(s)<\infty$,
$$\E[\tilde\cpst_s(t)]\leq \frac{p_0}{f(s)} e^{\frac{f(s)}{s}t}.$$
\item When $\E[Z]\leq 1$, for any $s\in [1, s_0], \theta\geq 1$, there exists constant $C_{\theta,s} < \infty$ such that,
\beq\label{moment_nonfringe}
\E[(\tilde\cpst_s(t))^\theta]\leq C_{\theta,s}e^{\frac{f(s)}{s}\theta t}, \qquad \forall t> 0. 
\eeq
\end{enumeratei}
\end{corollary}
\begin{proof}
The result follows from the same proof as of Theorem \ref{updegree} upon noting that $(s^{-i}:i\geq 0)$ is a left subinvariant eigenvector of $\vB$ associated with the eigenvalue $f(s)/s$ when $s\geq 1$ (see Proposition \ref{prop:spectral-prop}(b)).
\end{proof}
\end{theorem}
The proof of Theorem \ref{updegree} is deferred to the end of this Section. The above bounds coupled with the preliminary analysis of $\cP_1$ in Section \ref{sec:urn-model-def} immediately lead to the following two Corollaries.

\begin{corollary}
\label{corr:cp1-infty}
For any $\delta>0$, we have the following limit
\beq\label{lower}
\lim_{t\to\infty} e^{-\left(\frac{1}{R}-\delta\right)t}\cP_1(t)=\infty \quad a.s.
\eeq
\beq\label{upper}
\lim_{t\to\infty} t^{-(1+\delta)}e^{-t/R}\cP_1(t)=0 \quad a.s.
\eeq
\end{corollary}
\begin{proof}
 Fix $\delta\in (0,1/R)$. Using Proposition \ref{Perroneigen} choose $k=k_\delta\in \mathbb{N}$ large enough such that the Perron-Frobenius eigenvalue $\alpha_k$ of the $k\times k$ principal submatrix $A_k$ of $A$ satisfies 
$$\frac{1}{R}-\frac{\delta}{2}< \alpha_k \leq \frac{1}{R}.$$
By the stochastic domination in Lemma \ref{lem:urn-mod-stoch-dom} and the limit result for finite urns in Proposition \ref{prop:finite-urn}, 
\begin{align*}
\lim_{t\to\infty} e^{-\left(\frac{1}{R}-\delta\right)t}\cP_1(t)
&=\lim_{t\to\infty} e^{-\left(\frac{1}{R}-\delta\right)(t+\sigma_1)}\cP_1(t+\sigma_1)\\
&\geq e^{-\left(\frac{1}{R}-\delta\right)\sigma_1} \lim_{t\to\infty} e^{-\left(\frac{1}{R}-\delta\right)t}\bar{\cP}_1(t)\\
&\geq e^{-\left(\frac{1}{R}-\delta\right)\sigma_1} \lim_{t\to\infty} e^{\delta t/2}\cdot (e^{-\alpha_k t}\bar{\cP}_1(t))=\infty \quad a.s.
\end{align*}
This proves \eqref{lower}. To prove \eqref{upper}, for any given $\ep>0$ and $N\geq 0$ we can define the event 
$$E_N=\{ \sup_{s \in [N,N+1]} s^{-(1+\delta)} e^{-s/R} \cP_1(s)>\ep\}.$$ 
By Theorem \ref{updegree}, for any $t\geq 0$
$$\E[\cP_1(t)]\leq s_0\E[\cpst(t)]\leq  \frac{p_0 s_0}{f(s_0)}e^{\frac{f(s_0)}{s_0}t} = p_0 R e^{t/R}.$$ Hence,
\begin{align*}
\pr(E_N)&\leq \pr(\cP_1(N+1)>\ep e^{N/R}N^{1+\delta})\leq \frac{\E[\cP_1(N+1)]}{\ep e^{N/R}N^{1+\delta}}\\
&\leq \frac{ p_0Re^{(N+1)/R}}{\ep e^{N/R}N^{1+\delta}}=\frac{p_0R e^{1/R}}{\ep N^{1+\delta}}.
\end{align*}
Applying Borel-Cantelli Lemma then gives $\pr(\limsup_{N\to\infty} E_N)=0$. Hence,
$$t^{-(1+\delta)}e^{-t/R}\cP_1(t) \overset{a.s.}{\longrightarrow}0,$$
proving \eqref{upper}.
\end{proof}

\begin{corollary}\label{tau_moment}
Let $\tau \sim \mathrm{Exp}(1)$ independent of $\cT^*$. 
\begin{enumeratei}
\item When $\E[Z]\leq 1$, for $\theta \in [1,\frac{s}{f(s)})$ and $s\in [1,s_0]$,
$\E[ (\cpst_s(\tau))^\theta]<\infty.$
\item When $\E[Z]>1$, let $q_*<1$ be as in Definition \ref{def:R-q-s0}. For  $\theta\in [1,R \wedge \frac{\log q_*}{\log s_0})$, $\E[ (\cpst(\tau))^\theta]<\infty.$
\end{enumeratei}

\end{corollary}
\begin{proof}
(i) Simply follows from \eqref{moment_fringe}. For $\theta\in (0,\frac{s}{f(s)})$,
$$\E[ (\cpst_s(\tau))^\theta]=\int_0^\infty e^{-t} \E[(\cpst_s(t))^\theta] dt\leq  C_{\theta,s} \int_0^\infty e^{-t}e^{\frac{f(s)}{s}\theta t}dt<\infty.$$
(ii) We first start by evaluating $\alpha^*(\cdot)$ in Theorem \ref{updegree}(iii).  For $\E[Z]>1$ and $\theta\in [1,R\wedge \frac{\log q_*}{\log s_0})$ we obviously have then $\theta/ R<1$. Further, noting that $s_0\geq  s_0^\theta > s_0^{(\log q_*)/(\log s_0)} $, Lemma \ref{lemma:prop-R-s0}(a) gives  
$$\frac{f(s_0^\theta)}{s_0^\theta} <  \frac{ f(s_0^{(\log q_*)/(\log s_0)})}{ s_0^{(\log q_*)/(\log s_0)}}=\frac{f(q_*)}{q_*}=1.$$
Hence $\alpha^*(\theta)<1$. Using \eqref{moment_nonfringe} gives
$$
\E[ (\cpst(\tau))^\theta]=\int_0^\infty e^{-t} \E[(\cpst(t))^\theta] dt\leq  C_{\theta} \int_0^\infty e^{-t}(1+t^\theta)e^{\alpha^*(\theta) t}dt<\infty,$$
from which the result follows.
\end{proof}

\begin{proof}[{\bf Proof of Theorem \ref{thm:deg-dist} assuming Theorem \ref{updegree}}]
Throughout we use the representation in Corollary \ref{corr: limit_degree}. 

\noindent {\bf Upper bound for the tail exponent:}  
Note that by construction, for any $s>0$, $\cP_1(t) \leq s \cpst_s(t)$.  We start with the regime $\E[Z]\leq 1$. Note that $s/f(s)\uparrow R$ as $s\uparrow s_0$.  Corollary \ref{tau_moment}(i) gives that for any $\delta\in(0,R)$, there exists some $s_\delta>0$ such that $$\E(\cP_1(\tau)^{R-\delta})\leq s_\delta \E(\cpst_{s_\delta}(\tau)^{R-\delta})\leq s_\delta C_\delta.$$
It follows that
\begin{align}
\pr(D\geq k)&=\pr(\cP_1(\tau)\geq k)\leq \frac{\E[ (\cP_1(\tau))^{R-\delta}]}{k^{R-\delta}}\leq \frac{s_\delta C_\delta}{k^{R-\delta}} \label{eqn:802},
\end{align}
 for finite constant $C_\delta$. This implies that, 
$$\limsup_{k\to\infty}  \frac{\log \pr(D\geq k)}{\log k}\leq \limsup_{k\to\infty} \frac{ \log (s_\delta C_\delta)-(R-\delta)\log k}{\log k}=-R+\delta.$$
As $\delta>0$ can be chosen arbitrarily small, this completes the upper bound for the tail exponent when $\E[Z] \leq 1$.

The case $\E[Z] >1$ follows the exact same argument but using Corollary \ref{tau_moment}(ii) with the exponent $R$ in \eqref{eqn:802} replaced by $R \wedge \frac{\log q_*}{\log s_0}$.

\noindent{\bf Lower bound for the tail exponent:} Here we want to show that $R$ is a lower bound on the tail exponent in all regimes. Fix $0< \delta <R$.   Note that for any fixed $k\geq 1$,
\begin{align}\label{lld}
\pr(D\geq k)&=\int_0^\infty e^{-s} \pr(\cP_1(s)\geq k) ds= \int_0^\infty e^{-s} \pr\left(e^{-\frac{s}{R+\delta}}\cP_1(s)\geq e^{-\frac{s}{R+\delta}}k\right) ds\notag\\
&\geq \int_{\set{s:e^{-\frac{s}{R+\delta}}k \leq 1}} e^{-s}\pr\left(e^{-\frac{s}{R+\delta}}\cP_1(s)\geq 1\right) ds  = \int_{(R+\delta)\log k}^\infty e^{-s}\pr\left(e^{-\frac{s}{R+\delta}}\cP_1(s)\geq 1\right) ds.
\end{align}
By \eqref{lower} in Corollary \ref{corr:cp1-infty}, $e^{-\frac{s}{R+\delta}}\cP_1(s) \to \infty$ a.s. Thus  there exists $k_0 \geq 1$ such that 
$\pr\left(e^{-\frac{s}{R+\delta}}\cP_1(s)\geq 1\right) \geq 1/2$ for all $s\geq (R+\delta)\log k_0$. Thus, for $k\geq k_0$, we have
$$\pr(D\geq k)\geq  \frac{1}{2} \int_{(R+\delta)\log k}^\infty e^{-s}ds=\frac{1}{2 k^{R+\delta}},$$
which leads to
$\liminf_{k\to\infty}  {\log \pr(D\geq k)}/{\log k}\geq -(R+\delta)$
for arbitrarily small $\delta>0$.
\end{proof}

\subsection{{\bf Proof of Theorem \ref{updegree}}} 
The following lemma gives a tractable formulation for the expectation of powers of $\cpst$. This will be used to set up differential equations involving $\E[(\cpst(\cdot))^\theta]$ for $\theta \ge 1$ whose analysis will lead to the proof of Theorem \ref{updegree}. \color{black}In the following, a crucially used object will be the generator $\bar \LL$ of the continuous time Markov process $\bcP(\cdot)$ taking values in 
$$
\bar \cS := \{\vx = (x_i : i \ge 0) \in (\mathbb{N}_0)^{\mathbb{N}_0} : x_0=1, \, \exists \, l_{\vx} \in \mathbb{N} \text{ such that } x_l>0 \, \forall\, l \le l_{\vx} \text{ and } x_l=0 \, \forall\, l>l_{\vx}\}.
$$ 
For any function $f: \bar \cS \rightarrow \mathbb{R}$, define the action of the generator $\bar \LL$ on $f$ as the function $\bar \LL f : \bar \cS \rightarrow \mathbb{R}$ given by
$$
\bar \LL f(\vx) = \sum_{i=1}^{\infty}\left[f(\vx + \ve_i) - f(\vx)\right]\sum_{j=0}^{\infty}A_{ij}x_j, \  \ \vx \in \bar \cS,
$$
whenever the right hand side above is well-defined (here $\ve_i$ is the i-th coordinate unit vector). For notational convenience, for $g: \bar \cS \rightarrow \mathbb{R}$, we will write $\LL[g(\bcP(t))] := \bar \LL g (\bcP(t)), \, t \ge 0$.
\color{black}
\begin{lemma}\label{generator}
For any $s>0, t\geq 0$ and $\theta\geq 1$,
\beq
\label{eqn:1036}
\E[ (\cpst_s(t))^\theta]=\int_0^t \E[\LL[(\cpst_s(r))^\theta]]dr<\infty,
\eeq
where
\begin{equation}
    \label{eqn:full-gene}
    \LL[(\cpst_s(t))^\theta]=\sum_{i=1}^\infty \left[ (\cpst_s(t)+s^{-i})^\theta-(\cpst_s(t))^\theta\right] \sum_{j=0}^\infty A_{ij}\cP_j(t).
\end{equation}
\end{lemma}
\begin{proof}
Let $\yu(\cdot)$ denote a rate one Yule process as in Lemma \ref{diam_tail}, let $Y(t) = |\yu(t)|$ denote the size and $\hght(t)$ the corresponding height of the genealogical tree at time $t$. Note that the height of $\TT^*(t)$ is stochastically dominated by the height $\hght(t)$ of the Yule process. Thus for any $\theta \geq 1$ using Cauchy-Schwartz inequality, 
\begin{align*}
\E[ (\cpst_s(t))^\theta]&\leq \E[  (\sup_{i\leq \hght(t)} s^{-i}\cdot Y(t))^\theta]\leq \E[ (1\vee s^{-\theta \hght(t)}) (Y(t))^\theta]\\
&\leq \sqrt{\E[(1\vee s^{-2\theta \hght(t)})]}\sqrt{ \E[ (Y(t))^{2\theta}]}<\infty,
\end{align*}
where the finiteness of the first term follows from Lemma \ref{diam_tail} and finiteness of the second term follows from distributional identity $Y(t)\sim \text{Geometric}(e^{-t})$. This proves $\E[ (\cpst_s(t))^\theta]< \infty$.  To prove \eqref{eqn:1036}, first consider the truncation $\cpst_{s,N}(t)=\sum_{i=1}^N s^{-i} \cP_i(t)$. Applying the generator of the Markov process $\bcP(\cdot)$ gives
$$\E[ (\cpst_{s,N}(t))^\theta]=\int_0^t \E \LL[ (\cpst_{s,N}(r))^\theta]dr$$
where,
$$\LL[ (\cpst_{s,N}(t))^\theta]:=\sum_{i=1}^\infty \ind_{\{ i\leq N\}}  \left[ (\cpst_{s,N}(t)+s^{-i})^\theta-(\cpst_{s,N}(t))^\theta\right] \sum_{j=0}^\infty A_{ij}\cP_j(t).$$
Note that for each $i\geq 1$, 
$$\ind_{\{ i\leq N\}}  \left[ (\cpst_{s,N}(t)+s^{-i})^\theta-(\cpst_{s,N}(t))^\theta\right] \nearrow \left[ (\cpst_s(t)+s^{-i})^\theta-(\cpst_s(t))^\theta\right]$$ as $N\to\infty$.
By monotone convergence theorem,
$$\E[ (\cpst_s(t))^\theta]=\lim_{N\to\infty} \E[ (\cpst_{s,N}(t))^\theta]=\int_0^t \E[\LL[(\cpst_s(r))^\theta]]dr.$$
\end{proof}

\begin{rem}\label{rem:mar}
By the same argument conditional on $\mathcal{F}_u$ for any fixed $u \ge 0$, it follows that $\E[ (\cpst_s(t+u)^\theta -(\cpst_s(u))^\theta \, \vert \, \mathcal{F}_u]=\int_u^{t+u} \E[\LL[(\cpst_s(r))^\theta] \, \vert \, \mathcal{F}_u]dr, \ t \ge 0$. This implies that the process $\left(\cpst_s(t)^\theta - \int_0^{t} \LL[(\cpst_s(r))^\theta]dr\right)_{t \ge 0}$ is a martingale. 
\end{rem}

\ \\ 
\label{sec:proof-updegree}\noindent {\bf Proof of Theorem \ref{updegree}(i):}
We begin by proving \eqref{Ptub} namely $\theta=1$ case. Using \eqref{eqn:full-gene} and the form of $\vA$ gives, 
\begin{align*}
    \cL(\cpst_s(t)) &= \sum_{i=1}^\infty s^{-i} \sum_{j=i-1}^{\infty} p_{j-i+1} \cP_j(t) = \sum_{j=0}^{\infty}\left[\sum_{i=1}^{j+1} s^{-i}p_{j-i+1}\right]\cP_j(t),\\
    &= \sum_{j=0}^{\infty}\left[\sum_{i=1}^{j+1} s^{j-i+1} p_{j-i+1}\right] s^{-(j+1)} \cP_j(t) \leq \frac{p_0}{s} + \frac{f(s)}{s} \sum_{j=1}^\infty s^{-j}\cP_j(t)
    = \frac{p_0}{s} + \frac{f(s)}{s} \cpst(t).
\end{align*}
Here we have essentially re-derived Proposition \ref{prop:spectral-prop}(a) on $(s^{-i}:i\geq 0)$ being a sub-invariant eigenvector with eigenvalue $f(s)/s$. Using Lemma \ref{generator} gives
\[\frac{d}{dt} \E(\cpst_s(t)) \leq \frac{p_0}{s} + \frac{f(s)}{s}\E(\cpst_s(t)).  \]
Integrating completes the proof. 

\noindent
{\bf Proof of Theorem \ref{updegree}(ii):} To prove \eqref{moment_fringe}, we will first start by assuming that $\theta$ is an integer and argue by induction. After completing the proof for integer $\theta$, we will extend the proof to general $\theta$.  By \eqref{Ptub}, the assertion is true when $\theta=1$. We will use $A_{ij}$ instead of $p_{j-i+1}$ for ease of notation.  Suppose  \eqref{moment_fringe} holds for $\theta\leq k-1$. 
\begin{align*}
\LL[ (\cpst_s(t))^k]&=\sum_{i=1}^\infty \left[ (\cpst_s(t)+s^{-i})^k-(\cpst_s(t))^k\right] \sum_{j=0}^\infty A_{ij}\cP_j(t)\\
&=\sum_{i=1}^\infty \sum_{l=1}^k {k \choose l} (\cpst_s(t))^{k-l} s^{-il}\sum_{j=0}^\infty A_{ij}\cP_j(t)\leq \sum_{i=1}^\infty \sum_{l=1}^k {k \choose l} (\cpst_s(t))^{k-l} s^{-i}\sum_{j=0}^\infty A_{ij}\cP_j(t) \\
&= \sum_{l=1}^k{k \choose l} (\cpst_s(t))^{k-l} \sum_{j=0}^\infty \left( \sum_{i=1}^\infty s^{-i} A_{ij}\right)\cP_j(t)\leq \sum_{l=1}^k{k \choose l} (\cpst_s(t))^{k-l} \left( s^{-1}p_0+ \frac{f(s)}{s} \cpst_s(t)\right)\\
&=\frac{kf(s)}{s} (\cpst_s(t))^k+\frac{f(s)}{s} \sum_{l=2}^k{k \choose l} (\cpst_s(t))^{k-l+1}+ s^{-1}p_0 \sum_{l=1}^k{k \choose l} (\cpst_s(t))^{k-l}.
\end{align*}
Here the key inequality in line two follows from assuming that $s\ge 1$ so that $s^{-il} \leq s^{-i}$, whilst the inequalities in the ensuing lines mimic calculations in the $\theta =1$ case. By the induction hypothesis, for $1\leq i\leq k-1$, $\exists $ finite constants $C_{i,s}$ such that $\E((\cpst_s(t))^{i}) \leq C_{i,s}\exp(\frac{f(s)}{s} i\cdot t)$. Let $\tilde C_{k,s}:=\max_{1\leq i\leq k-1}\{ C_{i,s}\}$. Using Lemma \ref{generator} gives, 
\begin{align*}
\frac{d}{dt} \left[ e^{-k\frac{f(s)}{s}t}\E[(\cpst_s(t))^k]\right]&\leq e^{-k\frac{f(s)}{s}t} \left[ \frac{f(s)}{s} \sum_{l=2}^k{k \choose l} \E[(\cpst_s(t))^{k-l+1}]+ s^{-1}p_0 \sum_{l=1}^k{k \choose l} \E[(\cpst_s(t))^{k-l}]\right]\\
&\leq \tilde C_{k,s}   e^{-\frac{f(s)}{s}t}\cdot \left[\frac{f(s)}{s}\sum_{l=2}^k{k \choose l}  e^{-(l-2)\frac{f(s)}{s}t}+s^{-1}p_0\sum_{l=1}^k{k \choose l} e^{-(l-1)\frac{f(s)}{s}t} \right]\\
&\leq \tilde C_{k,s}e^{-\frac{f(s)}{s}t}\cdot 2^k \left(\frac{f(s)}{s}+s^{-1}p_0\right).
\end{align*}
Therefore, there exists $C_{k,s}>0$ so that
$$\E[(\cpst_s(t))^k]\leq  \tilde C_{k,s}e^{k\frac{f(s)}{s}t}\int_0^t e^{-\frac{f(s)}{s}r}dr \cdot 2^k \left(\frac{f(s)}{s}+s^{-1}p_0\right)\leq C_{k,s} e^{k\frac{f(s)}{s}t},$$
proving \eqref{moment_fringe} for $\theta=k$. To extend to $\theta\in (k-1,k]$, we apply Jensen's inequality to obtain,
$$\E[(\cpst_s(t))^\theta ]\leq \E[(\cpst_s(t))^k]^{\frac{\theta}{k}}\leq C_{k,s}^{\frac{\theta}{k}} e^{\theta \frac{f(s)}{s}t}=C_{\theta,s} e^{\theta \frac{f(s)}{s}t}.$$ 

\noindent
{\bf Proof of Theorem \ref{updegree}(iii):} It remains to prove \eqref{moment_nonfringe}  when $\E[Z]>1$ so that $s_0< 1$.  Recall that $\cpst = \cpst_{s_0}$.  Note that for any $\theta>1, i\geq 1$,
\begin{align*}
(\cpst(t)+s_0^{-i})^\theta-\cpst(t)^\theta& \leq \theta s_0^{-i} \cpst(t)^{\theta-1}+\frac{\theta(\theta-1)}{2} s_0^{-2i}\left[ \cpst(t)^{\theta-2}\vee (\cpst(t)+s_0^{-i})^{\theta-2}\right].
\end{align*}
It follows that 
\begin{align}\label{generator-2}
\nonumber \LL[ (\cpst(t))^\theta]&=\sum_{i=1}^\infty \left[ (\cpst(t)+s_0^{-i})^\theta-(\cpst(t))^\theta\right] \sum_{j=0}^\infty A_{ij}\cP_j(t)\\
\nonumber &\leq \sum_{i=1}^\infty \theta s_0^{-i} \cpst(t)^{\theta-1} \sum_{j=0}^\infty A_{ij}\cP_j(t)\\
&\quad +\sum_{i=1}^\infty \frac{\theta(\theta-1)}{2} s_0^{-2i}\left[ \cpst(t)^{\theta-2}\vee (\cpst(t)+s_0^{-i})^{\theta-2}\right] \sum_{j=0}^\infty A_{ij}\cP_j(t).
\end{align}
Since $\sum_{j=0}^\infty A_{ij}\cP_j(t)=\sum_{j=i-1}^\infty p_{j+1-i}\cP_j(t)$, it is easy to see $\sum_{j=0}^\infty A_{ij}\cP_j(t)>0$ only when $\cP_{i-1}(t)\geq 1$, which implies $\cpst(t)\geq s_0^{-(i-1)}$. 

Hence for $\theta\geq 2$,
$$\left[ \cpst(t)^{\theta-2}\vee (\cpst(t)+s_0^{-i})^{\theta-2}\right] \sum_{j=0}^\infty A_{ij}\cP_j(t)\leq (1+1/s_0)^{\theta-2}\cpst(t)^{\theta-2}\sum_{j=0}^\infty A_{ij}\cP_j(t).$$
Taking into account the case where $\theta\in (1,2)$, we have
$$\left[ \cpst(t)^{\theta-2}\vee (\cpst(t)+s_0^{-i})^{\theta-2}\right] \sum_{j=0}^\infty A_{ij}\cP_j(t)\leq \left(1\vee(1+1/s_0)^{\theta-2}\right)\cpst(t)^{\theta-2}\sum_{j=0}^\infty A_{ij}\cP_j(t).$$
Plugging this back into \eqref{generator-2} gives
\begin{align*}
\LL[ (\cpst(t))^\theta]&\leq  \theta \cpst(t)^{\theta-1} \sum_{j=0}^\infty \left(\sum_{i=1}^\infty s_0^{-i} A_{ij}\right)\cP_j(t)\\
&\quad + \frac{\theta(\theta-1)}{2} \left(1\vee(1+1/s_0)^{\theta-2}\right) \cpst(t)^{\theta-2} \sum_{j=0}^\infty \left(\sum_{i=1}^\infty s_0^{-2i}A_{ij}\right) \cP_j(t)\\
&\leq \frac{\theta p_0}{s_0}\cpst(t)^{\theta-1}+\frac{\theta}{R}\cpst(t)^\theta+ \frac{\theta(\theta-1)}{2} \left(1\vee(1+1/s_0)^{\theta-2}\right) \frac{p_0}{s_0^2} \cpst(t)^{\theta-2}\\
&\quad+\frac{\theta(\theta-1)}{2} \left(1\vee(1+1/s_0)^{\theta-2}\right)\frac{f(s_0^2)}{s_0^2} \cpst(t)^{\theta-2}\sum_{j=1}^{\infty}s_0^{-2j}\cP_j(t).
\end{align*}
Define $d(t)=\sup\{i\geq 0: \cP_i(t)\geq 1\}$. Observe that, since $s_0\in (0,1)$, $\sum_{j=1}^\infty s_0^{-2j}\cP_j(t)\leq s_0^{-d(t)}\cpst(t)$. Hence,
\begin{align}\label{induct_1}
\nonumber\LL[ (\cpst(t))^\theta]&\leq \left(\frac{\theta p_0}{s_0}+\frac{\theta(\theta-1)}{2} \left(1\vee(1+1/s_0)^{\theta-2}\right)\frac{f(s_0^2)}{s_0^2} s_0^{-d(t)}\right)\cpst(t)^{\theta-1}\\
\nonumber&\quad+\frac{\theta}{R}\cpst(t)^\theta+ \frac{\theta(\theta-1)}{2} \left(1\vee(1+1/s_0)^{\theta-2}\right) \frac{p_0}{s_0^2} \cpst(t)^{\theta-2}\\
&\leq \tilde C_\theta s_0^{-d(t)}\cpst(t)^{\theta-1}+\frac{\theta}{R}\cpst(t)^\theta,
\end{align}
for some finite constant $\tilde C_\theta$ that depends on $\theta,s_0$ and $p_0$. Notice that the term $\cpst(t)^{\theta-2}$ can be upper bounded by $\cpst(t)^{\theta-1}$ for all $\theta\geq 1$ since $\cpst(t)\geq 1$. For any  $\theta\geq 1$, 
\begin{align*}
\frac{d}{dt} \E[\cpst(t)^\theta]&=\E\left[\LL[ (\cpst(t))^\theta]\right]\leq \tilde C_\theta \E[s_0^{-d(t)}\cpst(t)^{\theta-1}]+\frac{\theta}{R}\E[\cpst(t)^\theta]\\
&\leq \tilde C_\theta \left( \E[\cpst(t)^\theta]\right)^{\frac{\theta-1}{\theta} }\left( \E[s_0^{-\theta d(t)}]\right)^{\frac{1}{\theta}}+\frac{\theta}{R}\E[\cpst(t)^\theta].
\end{align*}
By the definition of $\alpha^*(\theta)$ and \eqref{Ptub},
$$\E[s_0^{-\theta d(t)}]\leq \E\left[ \sum_{j=1}^\infty s_0^{-j\theta}\cP_j(t)\right]\leq \frac{p_0}{f(s_0^\theta)}e^{\frac{f(s_0^\theta)}{s_0^\theta}t}\leq \frac{p_0}{f(s_0^\theta)}e^{\alpha^*(\theta) t}.$$
Let $g(t)=e^{-\alpha^*(\theta) t}\E[\cP^*(t)^\theta]$ and $C'_\theta= \tilde C_\theta \left(\frac{p_0}{f(s_0^\theta)}\right)^{1/\theta}$. Then
\beq \label{g_equation}
g'(t)\leq C'_\theta (g(t))^{\frac{\theta-1}{\theta}}.
\eeq
Define $h(t)=1+g(0)^{1/\theta}+\frac{2C'_\theta}{\theta}t-g(t)^{1/\theta}$. Then \eqref{g_equation} implies that,
$$h'(t)=\frac{2C'_\theta}{\theta}-\frac{1}{\theta}g'(t) (g(t))^{\frac{1}{\theta}-1}\geq \frac{C'_\theta}{\theta}.$$
Since $h(0)>0$ we have $h(t)>0$. It then follows that there exists some $C_\theta>0$ so that $g(t)\leq C_\theta(1+t^{\theta})$ for all $t\geq 0$, i.e., 
$$\E[\cpst(t)^\theta]\leq C_\theta(1+t^{\theta})e^{\alpha^*(\theta) t},$$
which is assertion (iii). \qed

\section{Proofs: Condensation and fixed vertex degree asymptotics}
\label{sec:proof-fixed-vertex}
We abbreviate $d_n(v_k) = \deg(v_k,n), \, k \ge 0, n \ge k$. We will first prove the non-root fixed vertex asypmtotics.
\begin{proof}[{\bf Proof of Theorem \ref{thm:fixed-vertex}}]
This follows from a direct application of Corollary \ref{corr:cp1-infty}. Using the continuous time embedding, note that for any $i \ge 1$, 
$$\{d_n(v_i): n\geq i\}\overset{d}{=} \{ \cP^{v_i}_1(T_n - \sigma_i): n\geq i\},$$
where as before $T_n=\inf\{ t\geq 0: |\TT(t)|=n+1\}$, $\sigma_i$ is the birth time of $v_i$ and $\cP^{v_i}_1(t)$ denotes the number of children of vertex $v_i$ in $\TT(t + \sigma_i)$.
Observe that $\cP^{v_i}_1(\cdot)$ has the same distribution as the process $\cP_1(\cdot)$. By Lemma \ref{lem:yule-asymp}, $\frac{e^{T_n}}{n} \overset{a.s.}{\longrightarrow} \frac{1}{W}$ where $W\sim \mathrm{Exp}(1)$. It follows from \eqref{lower} that 
$$\frac{d_n(v_i)}{n^{1/R-\delta}}=\frac{\cP^{v_i}_1(T_n - \sigma_i)}{n^{1/R-\delta}}=\frac{\cP^{v_i}_1(T_n - \sigma_i)}{e^{(1/R-\delta)(T_n-\sigma_i)}} \cdot \left( \frac{ e^{(T_n-\sigma_i)}}{n}\right)^{1/R-\delta} \overset{a.s.}{\longrightarrow} \infty.$$
Similarly, by \eqref{upper},
\begin{align*}
\frac{ d_n(v_i)}{n^{1/R}(\log n)^{1+\delta}}&=\frac{\cP^{v_i}_1(T_n - \sigma_i)}{n^{1/R}(\log n)^{1+\delta}}\\
&=\frac{\cP^{v_i}_1(T_n-\sigma_i)}{e^{(T_n-\sigma_i)/R}\cdot (T_n-\sigma_i)^{1+\delta}} \cdot \left(\frac{e^{(T_n-\sigma_i)}}{n}\right)^{1/R}\cdot \left( \frac{T_n - \sigma_i}{\log n}\right)^{1+\delta} 
\overset{a.s.}{\longrightarrow} 0.
\end{align*}

\end{proof}
Now we proceed to root degree asymptotics.   

\begin{proof}[{\bf Proof of Theorem \ref{thm:max-degree}(a)}]Recall that $T_n$ is the stopping time when $n(t)=|\TT(t)|$ first becomes $n+1$. Standard properties of the Yule process (see Lemma \ref{lem:yule-prop}) imply that 
$$\pr(|T_n-\log n|\geq M)\leq 2e^{-M},$$
where $M>0$ will be chosen later.

For any fixed $\ep>0$, it follows from Theorem  \ref{root_condensation} that there is some $t_0(\ep)>0$ such that for $t\geq t_0$,
\beq\label{bound_thm5.26}
\E| e^{-t}\tilde{\cP}_1(t)-W_\infty|\leq \ep^3,
\eeq
where $W_\infty:=(1-q_*)W$.

For any $n\geq e^{t_0+M}$ and  some $\delta>0$ to be chosen later, let $a_{n,0}=\log n-M$ and define $a_{n,i}=a_{n,i-1}+\delta$ for $i\geq 1$. Then we can observe that 
\begin{align}\label{two_errs}
\nonumber\E|e^{-T_n}\tilde\cP_1(T_n)-W_\infty|&\leq\sum_{i=0}^{\lceil2M/\delta \rceil} \E\left[|e^{-T_n}\tilde\cP_1(T_n)-W_\infty|\cdot 1_{\{ a_i\leq T_n <a_{i+1}\}}\right]\\
&\quad + \E\left[|e^{-T_n}\tilde\cP_1(T_n)-W_\infty| \cdot 1_{\{|T_n-\log n|\geq M\}}\right].
\end{align}
To address the second term in \eqref{two_errs}, note that 
\begin{align*}
&\E\left[|e^{-T_n}\tilde\cP_1(T_n)-W_\infty|\cdot 1_{\{|T_n-\log n|\geq M\}}\right]\\
\leq & \left(\E[(e^{-T_n}\tilde\cP_1(T_n)-W_\infty)^2]\right)^{1/2}\cdot \pr(|T_n-\log n|\geq M)^{1/2}\\
\leq &  \sqrt{2}\left( 2(n+1)^2\E[ e^{-2T_n}]+ 2\E[W_\infty^2] \right)^{1/2} e^{-M/2}.
\end{align*}
To upper bound $(n+1)^2\E[ e^{-2T_n}]$, write $X_n= (n+1)^2 e^{-2T_n}$. Observe that 
\begin{align*}
\E[X_{n+1}]&=(n+2)^2 \E[e^{-2(T_{n+1}-T_n)}e^{-2T_n}]\\
&=\E[e^{-2T_n}] (n+2)^2 \cdot \frac{n+1}{n+3}=\E[X_n]\cdot \frac{ (n+2)^2}{(n+1)(n+3)},
\end{align*}
where we used the observation $T_{n+1} - T_n \sim \mathrm{Exp}(n+1)$ and is independent of $T_n$.
Hence, for any $n\geq 1$,
\begin{align*}
\E[X_{n+1}]&=\E[X_n]\left( 1+\frac{1}{(n+1)(n+3)}\right)\leq \E[X_n](1+\frac{1}{n^2})\\
&\leq \E[X_0]\prod_{k=1}^n e^{1/k^2}\leq C_0
\end{align*}
for some constant $C_0>0$. This combined, along with the explicit form of $W_{\infty}$, shows that there exists some constant $C>0$ so that 
$ \sqrt{2}\left( 2\E[ e^{-2T_n}n(T_n)^2]+ 2\E[W_\infty^2] \right)^{1/2}\leq C$. Taking $M=2\log(2C/\ep)$ then gives
\beq\label{second_err}
\E\left[|e^{-T_n}\tilde\cP_1(T_n)-W_\infty| \cdot 1_{\{|T_n-\log n|\geq M\}}\right]\leq \ep/2.
\eeq

For $t\in [a_i,a_{i+1})$, we have $e^{-a_{i+1}}\tilde\cP_1(a_i)\leq e^{-t}\tilde\cP_1(t) \leq e^{-a_i}\tilde\cP_1(a_{i+1})$. Hence,
\begin{align*}
&\E\left[|e^{-T_n}\tilde\cP_1(T_n)-W_\infty|\cdot1_{\{ a_i\leq T_n <a_{i+1}\}}\right]\\
\leq & \E\left[\max\{ |e^{-a_{i+1}}\tilde\cP_1(a_i)-W_\infty|, |e^{-a_i}\tilde\cP_1(a_{i+1})-W_\infty|\}\cdot1_{\{ a_i\leq T_n <a_{i+1}\}}\right]\\
\leq & \E\left[ \max\{ e^{-\delta}|e^{-a_i}\tilde\cP_1(a_i)-W_\infty|, e^\delta|e^{-a_{i+1}}\tilde\cP_1(a_{i+1})-W_\infty|\}\cdot1_{\{ a_i\leq T_n <a_{i+1}\}}\right]\\
&\quad+\max\{ e^\delta-1, 1-e^{-\delta}\}\E[W_\infty\cdot1_{\{ a_i\leq T_n <a_{i+1}\}}]\\
\leq & e^{-\delta}\E|e^{-a_i}\tilde\cP_1(a_i)-W_\infty|+e^\delta \E|e^{-a_{i+1}}\tilde\cP_1(a_{i+1})-W_\infty|+2\delta\E[W_\infty\cdot1_{\{ a_i\leq T_n <a_{i+1}\}}]
\end{align*}
where the last term follows from choosing $\delta\in(0,1)$ to be sufficiently small so that $\max\{ e^\delta-1, 1-e^{-\delta}\}\leq 2\delta$. Since $n\geq e^{t_0+M}$, for all $i\geq 0$ we have $a_i\geq t_0$. Applying \eqref{bound_thm5.26} to $t=a_i$ and $t=a_{i+1}$ yields
\begin{align*}
\E\left[|e^{-T_n}\tilde\cP_1(T_n)-W_\infty|\cdot1_{\{ a_i\leq T_n <a_{i+1}\}}\right]
\leq & (e^\delta+e^{-\delta})\ep^3+2\delta\E[W_\infty\cdot1_{\{ a_i\leq T_n <a_{i+1}\}}]\\
\leq &4\ep^3+2\delta\E[W_\infty\cdot1_{\{ a_i\leq T_n <a_{i+1}\}}].
\end{align*}
Hence, the first term in \eqref{two_errs} satisfies
\begin{align*}
\sum_{i=0}^{\lceil2M/\delta \rceil} \E\left[|e^{-T_n}\tilde\cP_1(T_n)-W_\infty|\cdot1_{\{ a_i\leq T_n <a_{i+1}\}}\right]&\leq \sum_{i=0}^{\lceil2M/\delta \rceil} \left(4\ep^3+2\delta\E[W_\infty\cdot1_{\{ a_i\leq T_n <a_{i+1}\}}]\right)\\
&\leq (\lceil2M/\delta \rceil+1)4\ep^3+2\delta\E[W_\infty].
\end{align*}
Take $\delta=\ep/8$. It follows from our choice of $M= 2\log(2C/\ep)$ that there exists $\ep_0>0$ such that for all $\ep \in (0,\ep_0)$ is sufficiently small we have $(\lceil2M/\delta \rceil+1)4\ep^3\leq \ep/4$. Since $\E[W_\infty]\leq 1$,
\beq\label{first_err}
\sum_{i=0}^{\lceil2M/\delta \rceil} \E\left[|e^{-T_n}\tilde\cP_1(T_n)-W_\infty|\cdot1_{\{ a_i\leq T_n <a_{i+1}\}}\right]\leq \ep/4+2(\ep/8)=\ep/2.
\eeq
Collecting \eqref{first_err} and \eqref{second_err} in \eqref{two_errs} gives that for any $\ep \in (0,\ep_0)$, there exist some $t_0,M>0$ depending on $\ep$ so that for all $n\geq e^{t_0+M}$,
$$
\E|e^{-T_n}\tilde\cP_1(T_n)-W_\infty|\leq \ep.
$$
Thus we have established the $\mathbb{L}^1$-convergence of $e^{-T_n}\tilde\cP_1(T_n)$ to $W_\infty$. By Lemma \ref{lem:yule-asymp}, $e^{T_n}/(n+1) \overset{a.s.}{\longrightarrow} 1/W$ where $W \sim \mathrm{Exp}(1)$. Hence, 
\beq\label{tnconvp}
\tilde \cP_1(T_n)/(n+1) \overset{\prob}{\longrightarrow} W_\infty/W = 1-q_*.
\eeq
The result now follows upon noting that $\{\frac{d_n(v_0)}{n+1} : n \ge 0\} \overset{d}{=} \{\frac{  \tilde\cP_1(T_n)}{n+1} : n \ge 0\}$.

\end{proof}

\begin{proof}[{\bf Proof of Theorem \ref{thm:max-degree}(b)}] 
Let $\delta>0$ be given. Note that $\{d_n(v_0): n \ge 1\} \overset{d}{=}\{\tilde{\cP}_1(T_n) : n \ge 1\}$. It is easy to see $\tilde{\cP}_1(\cdot)$ dominates $\cP_1(\cdot)$ and hence it follows from the proof of Theorem \ref{thm:fixed-vertex} that
$$\frac{d_n(v_0)}{n^{1/R-\delta}}\overset{a.s.}{\longrightarrow} \infty.$$

In the fringe regime $\E[Z]\leq 1$ we have $s_0\geq 1$. Using Corollary \ref{updegree_root}(i) with $s=s_0\geq 1$ to replace the estimate from Theorem \ref{updegree}(i) in the proof of \eqref{upper} in Corollary \ref{corr:cp1-infty} yields 
\beq\label{root_upper}
\lim_{t\to\infty} t^{-(1+\delta)}e^{-t/R}\tilde{\cP}_1(t)=0 \quad a.s.
\eeq
The same argument as in the proof Theorem \ref{thm:fixed-vertex} then gives
$$\frac{ d_n(v_0)}{n^{1/R}(\log n)^{1+\delta}}\overset{a.s.}{\longrightarrow} 0.$$
\end{proof}

\section{Proofs: Height}
\label{sec:proof-height}

\color{black}

Recall the processes and notation of Lemma \ref{lem:ht-domination} as well as $\kappa_0, s^*$ from Definition \ref{kappa}. The main goal of this section is to show the following: 

\begin{prop}
\label{prop:ht-cts}

	\begin{enumeratei}
	 \item {\bf Fringe regime:} When $\E[Z]\leq 1$, 
	$$\cH_{\cT(t)}/t \convas \kappa_0 \quad \text{ as }t\to\infty.$$
	 \item {\bf Non-Fringe regime:} When $\E[Z]> 1$, as $t\to\infty$,
 $$
	\frac{\cH_{\cT(t)}}{t}\overset{a.s.}{\longrightarrow}
	 \begin{cases}
	  \kappa_0 &\quad \text{ if } s^*\in (0,q_*],\\
	  \frac{1}{\log(1/q_*)} &\quad \text{ if } s^* \in (q_*,s_0],
	 \end{cases}
	 $$
	 where $q_*$ is defined as in Definition \ref{def:R-q-s0} (b). 
\end{enumeratei}
\end{prop}

\noindent {\bf Proof of Theorem \ref{thm:height} assuming Proposition \ref{prop:ht-cts}: } By the continuous time embedding, $\cH_n = \cH_{\cT(T_n)}$ where $T_n$ as before is the time for $\cT$ to get to size $n+1$. By Lemma \ref{lem:yule-asymp}(b) $T_n/\log{n} \convas 1$. Combining this with Proposition \ref{prop:ht-cts} completes the proof. \qed


The rest of this Section is devoted to the proof of Proposition \ref{prop:ht-cts}. Recall from Lemma \ref{lemma:prop-R-s0} that under Assumptions \ref{ass:mean} and \ref{ass:pmf}, $q_* < s_0 < 1$ when $\E[Z]>1$. 
 Recall that $\kappa(s):=\frac{f(s)}{s\log(1/s)}, \,  s \in (0,1)$.

\begin{lemma}\label{kappa_properties}
\begin{enumeratei}
\item Suppose $\E[Z]>1$. For $x>0$, the infimum of $s \mapsto x \log s +f(s)/s$ is uniquely attained at some $s^*_x\in (0,s_0]$ that satisfies $x=f(s^*_x)/s^*_x-f'(s^*_x)$. In particular, $s^*=s^*_{\kappa_0}\in (0,s_0]$.
\item The infimum of $s \mapsto \kappa(s)$ is uniquely attained at $s^*$.
\end{enumeratei}
\end{lemma}
\begin{proof}

To prove part (i), for $x>0$, write $h_x(s)=x\log s+f(s)/s, \, s \in (0,1)$. As both $\log s$ and $f(s)/s$ are increasing on $(s_0,1)$, we see that 
$$\inf_{s\in (0,1)}h_x(s)=\inf_{s\in(0,s_0]} h_x(s).$$
Differentiating $h_x(s)$ leads to $h'_x(s)=\frac{x-(f(s)/s-f'(s))}{s}$, which implies that the infimum is attained at some $s^*_x\in (0,s_0]$ satisfying the equation $x=f(s^*_x)/s^*_x-f'(s^*_x)$. As $f(s)/s$ and $-f'(s)$ are both strictly decreasing on $(0,s_0]$, the solution $s^*_x$ to this equation is unique.

For part (ii), it is easy to see that $\kappa(s)$ is increasing on $(s_0,1)$ so we can restrict our attention to $(0,s_0]$. Differentiating $\kappa(s)$ leads to 
$$\kappa'(s)=\frac{ \log(1/s) (sf'(s)-f(s))+f(s)}{ (s\log(1/s))^2}=\frac{ sf'(s)-f(s)+\frac{f(s)}{\log(1/s)}}{s^2 \log(1/s)}.$$
It is straightforward to check that both $sf'(s)-f(s)$ and $\frac{f(s)}{\log(1/s)}$ are strictly increasing on $(0,s_0]$. Hence, the infimum of $\kappa(\cdot)$ is uniquely attained at $s$ such that $\kappa'(s)=0$, i.e.,
$$f(s)/s-f'(s)=\frac{f(s)}{s\log(1/s)}.$$
By part (i) and (ii), we have $f(s^*)/s^*-f'(s^*)=\kappa_0=\frac{f(s^*)}{s^*\log(1/s^*)}$, i.e., $s^*$ is the unique point where $\kappa(\cdot)$ is minimized.
\end{proof}

\begin{proof}[{\bf Proof of Proposition \ref{prop:ht-cts}}]

\mn
\textbf{Fringe regime.} Suppose $\E[Z]\leq 1$. Lemma \ref{lem:ht-domination} and Proposition \ref{prop:bt-limit} give $\liminf_{t\to\infty} \cH_{\cT(t)}/t \geq \kappa_0$ almost surely. It remains to prove a corresponding upper bound. 

Let $s\in (0,1)$ and recall that we defined $\cpst_s(t)=\sum_{i=1}^\infty s^{-i}\cP_i(t)$.  Using \eqref{Ptub} in Theorem \ref{updegree} gives for any  $s\in(0,1)$, non-negative $x,t$ and $t' \in (0,t)$,
\begin{align}\label{hi}
\pr(\cH_{\cT^*(t-t')}\geq xt)&=\pr(s^{-\cH_{\cT^*(t-t')}}\geq s^{-xt})\leq \pr( \cpst_s(t-t')\geq s^{-xt})\nonumber\\
&\leq  s^{xt}\E[\cpst_s(t-t')]\leq \frac{p_0}{f(s)}\exp \left(t (x\log s+f(s)/s) - t'\frac{f(s)}{s}\right).
\end{align}
By definition, $\kappa_0$ is the infimum of $x$ such that $\inf_{s\in(0,1)} \{x\log s+f(s)/s\}<0$. It follows from Lemma \ref{kappa_properties}(i) that for any $x>\kappa_0$, there exists $s^*_x \in (0,s_0]$ such that
\beq\label{delta_dev}
-\delta:=\inf_{s\in(0,1)}\{ x\log s+f(s)/s\}=x\log s^*_x+f(s^*_x)/s^*_x<0.
\eeq
Note that in $\TT(\cdot)$, all subtrees rooted at level one evolve as $\TT^*(\cdot)$. In order for $\TT(t)$ to have height larger than $\lceil xt\rceil$, one of the subtrees rooted at level one need to achieve height at least $xt$.

With $R$ as in Definition \ref{def:R-q-s0}, we partition the time interval $[0,t]$ into segments $\{ [t_i,t_{i+1}] :i\geq 0\}$ where $t_0=0$ and $t_{i+1}=(t_i+mR)\wedge t$ for some positive constant $m$ to be chosen later. Let $\cS_i=\{ v_{ij}: j\geq 0\}$ denote the set of vertices at level one that arrived during $(t_i, t_{i+1}]$. Fix any $\ep>0$. For $i\geq 0$, define the event
\beq\label{event_Ei}
E_i(t)=\{ \text{one of the subtrees $\TT_{v_{ij}}$ rooted at $v_{ij}\in \cS_i$ has height at least $(\kappa_0+\ep)t$ at time $t$}\}.
\eeq
In the fringe regime, using Corollary \ref{updegree_root}(i) with $s=s_0\geq 1$ gives
$$\E(|\cS_i|) \leq \E(\tilde\cP_1(t_{i+1})) \leq p_0 R e^{t_{i+1}/R}.$$ 
Recall that by Lemma \ref{lemma:prop-R-s0}(a), $f(s)/s\geq 1/R$ for $s\in (0,1)$. Using \eqref{hi} with $t'=t_i$, $x=\kappa_0 + \ep$ and $s=s^*_x$ as in \eqref{delta_dev} gives, writing $\delta := \delta_x$, 
\begin{align}\label{prob_Ei}
\nonumber \pr(E_i(t))&\leq \E(|\cS_i|) \cdot \pr( \cH_{\cT^*(t-t_i)}\geq xt)\\
\nonumber &\leq \frac{p_0^2R}{f(s^*_x)}\exp\left(\frac{t_{i+1}}{R}+ t( x\log s^*_x+f(s^*_x)/s^*_x)- \frac{f(s^*_x)}{s^*_x}t_i\right)\\
&\leq \frac{p_0^2R}{f(s^*_x)}\exp( -\delta t +\frac{t_{i+1}-t_i}{R})\leq p_0 R\, e^{-\delta t+m}.
\end{align}
Let $i_0=\min \{i\geq 0:  t_{i+1}\geq t\}$. Note that $i_0\leq t/(mR)$. It follows that 
$$\pr(\cH_{\cT(t)}> \lceil xt\rceil)\leq \pr( \cup_{i=0}^{i_0} E_i(t))\leq \left(\frac{t}{mR}+1\right)p_0 R\, e^{-\delta t+m}.$$
Choosing $m=(\delta t)/2$ we have 
\beq\label{height_upper}
\pr(\cH_{\cT(t)} > \lceil xt\rceil)\leq (2/(\delta R)+1)p_0 R\,e^{-\delta t/2}.
\eeq
Given any $\ep>0$ and $N\geq 0$ we can define the event 
$$E_N=\big\{ \exists t\in [N,N+1]: \cH_{\cT(t)}  >\lceil(x+\ep)t\rceil\big\}.$$ 
Let $d_N=(x+\ep)\frac{N}{N+1}$. When $N$ is sufficiently large, $d_N>x$. Then \eqref{height_upper} gives
\begin{align*}
\pr(E_N)&\leq \pr(\cH_{\cT(N+1)} >\lceil(x+\ep)N\rceil)=\pr(\cH_{\cT(N+1)} >\lceil d_N(N+1)\rceil)\\
&\leq (2/(\delta R)+1)p_0 R\,e^{-\delta (N+1)/2}.
\end{align*}
Applying Borel-Cantelli Lemma then gives $\pr(\limsup_{N\to\infty} E_N)=0$. As $x=\kappa_0 + \ep$ and $\ep>0$ is arbitrary, 
$$\limsup_{t\to\infty} \frac{\cH_{\cT(t)}}{t} \leq \kappa_0 \quad \text{ almost surely}.$$
Combining this with the lower bound completes the proof for the case $\E[Z] \le 1$.

\mn
\textbf{Non-fringe regime with $s^*\in (0,q_*]$.} Recall from Lemma \ref{kappa_properties} that $\kappa_0=f(s^*)/s^*-f'(s^*)$. Using Lemma \ref{lemma:prop-R-s0}(a), $f(s)/s-f'(s)$ is a strictly decreasing function on $(0,s_0]$ and $f(s)/s-f'(s)\uparrow \infty$ as $s\downarrow 0$, for any $x>\kappa_0$, there exists some $s^*_x<s^*$ such that 
$$x=f(s^*_x)/s^*_x-f'(s^*_x).$$ 
As before, fix any $\ep>0$ and set $x=\kappa_0 + \ep$. It follows from Lemma  \ref{kappa_properties}(i) that the infimum of $s \mapsto  x\log s+f(s)/s$ on $(0,1)$ is uniquely achieved at $s^*_x$ and by the definition of $\kappa_0$ we see that 
$$-\delta:=\inf_{s\in (0,1)} \{ x\log s+f(s)/s\}=x\log(s^*_x)+f(s^*_x)/s^*_x<0.$$

Define $\{E_i(t)\}_{i\geq 0}$ as in \eqref{event_Ei}. Since $s^*_x<s^*\leq q_*$ we have $f(s^*_x)/s^*_x\geq f(q_*)/q_*=1$. By a similar calculation as \eqref{prob_Ei} with $\E(|\cS_i|) \leq \E(|\cT(t_{i+1})|)=\exp(t_{i+1})$ and $s=s^*_x$, $t_{i+1}=(t_i+m)\wedge t$, we can obtain
\begin{align*}
\pr(E_i(t))&\leq \E(|\cS_i|) \cdot \pr( \cH_{\cT^*(t-t_i)}\geq (\kappa_0+\ep)t)\\
&\leq \frac{p_0}{f(s^*_x)}\exp\left(t_{i+1}+ t( (\kappa_0+\ep)\log s^*_x+f(s^*_x)/s^*_x)- \frac{f(s^*_x)}{s^*_x}t_i\right)\\
&\leq \frac{p_0}{f(s^*_x)}\exp( -\delta t +(t_{i+1}-t_i))\leq e^{-\delta t+m}.
\end{align*}
The rest of the proof follows from the same arguments as the case before by choosing $m=\delta t/2$.

\mn
\textbf{Non-fringe regime with $s^*\in (q_*,s_0]$.} We begin by proving the upper bound. Fix $\ep>0$. Let $\{E_i(t)\}_{i\geq 0}$ be events defined as in \eqref{event_Ei} with $\kappa_0 + \ep$ replaced by $(\log(1/q_*))^{-1}(1+\ep)$. Using the fact that $\E(|\cS_i|) \leq \exp(t_{i+1})$,
by a similar calculation as \eqref{prob_Ei} with $x= (\log(1/q_*))^{-1}(1+\ep)$, $s=q_*$, $t_{i+1}=(t_i+m)\wedge t$, we have
\begin{align*}
\pr(E_i(t))&\leq \E(|\cS_i|) \cdot \pr( \cH_{\cT^*(t-t_i)}\geq (\log(1/q_*))^{-1}(1+\ep)t)\\
&\leq \frac{p_0}{f(q_*)}\exp\left(t_{i+1}+ t(  (\log(1/q_*)^{-1}(1+\ep)\log q_*+f(q_*)/q_*)- \frac{f(q_*)}{q_*}t_i\right)\\
&= \frac{p_0}{f(q_*)}\exp( -\ep t +(t_{i+1}-t_i)) \le e^{ -\ep t +m}.
\end{align*}
Choosing $m=\ep t/2$ and repeating the arguments in previous cases proves the upper bound, i.e.,
$$\limsup_{t\to\infty} \frac{\cH_{\cT(t)}}{t} \leq (\log(1/q_*))^{-1} \quad \text{ almost surely}.$$

It remains to prove the matching lower bound. Let $\delta\in (0,1)$ be a constant that we will choose later. Conditional on the tree $\cT((1-\delta)t)$ at time $(1-\delta)t$, observe that the height of $\cT(t)$ is stochastically lower bounded by the maximum of $|\cT((1-\delta)t)|$ many i.i.d. random variables, each distributed as the location of the rightmost particle of a branching random walk $\BRW(\delta t)$, which has the distribution $B(\delta t)$. 

Fix any $\ep>0$. Let $x=(1-\ep)(\log(1/q_*))^{-1}$. Using the above observation,
\beq\label{height_lb}
\pr( \cH_{\cT(t)} < xt)\leq \pr( B(\delta t)< xt)^{e^{(1-\delta)t-t^{1/2}}}+\pr( |\cT((1-\delta)t)|\leq  e^{(1-\delta)t-t^{1/2}}).
\eeq
Note that $|\cT(\cdot)|$ is a rate 1 Yule process and hence by Lemma \ref{lem:yule-prop} we have for large enough $t$, 
\begin{align}\label{yule_deviation}
\nonumber\pr( |\cT((1-\delta)t)|\leq  e^{(1-\delta)t-t^{1/2}})&=\sum_{k=1}^{ \lfloor e^{(1-\delta)t-t^{1/2}}\rfloor} e^{-(1-\delta)t}(1-e^{-(1-\delta)t})^{k-1}\\
&\leq 1-(1-e^{-(1-\delta)t})^{ e^{(1-\delta)t-t^{1/2}}}\leq  1-\exp(-2e^{-t^{1/2}})\leq  2e^{-t^{1/2}}
\end{align}
by the elementary inequality that $1-2x\leq e^{-2x}\leq 1-x$ for $x\in [0,1/2]$.


It remains to estimate $\pr( B(\delta t) < xt)$. For reasons that will become clear soon, we will take  $\delta= \frac{1-\ep}{(1-f'(q_*))\log(1/q_*)}$ so that $x/\delta=1-f'(q_*)$. To verify that $\delta\in (0,1)$ it suffices to show $(1-f'(q_*))\log(1/q_*)\geq 1$. In this case, it follows from Lemma \ref{kappa_properties}(ii) that $s^*\in (q_*,s_0]$ is the unique minimizer of $\kappa(s)=\frac{f(s)}{s\log(1/s)}$, which implies that $\kappa'(q_*)\leq 0$. This, in turn, leads to $(1-f'(q_*))\log(1/q_*)\geq 1$. Moreover,
$$\kappa_0=\kappa(s^*)< \kappa(q_*)=\frac{1}{\log(1/q_*)}.$$
That is, our choice of $x$ and $\delta$ gives $x/\delta=1-f'(q_*)\geq (\log(1/q_*))^{-1}>\kappa_0$. Then we can apply Lemma \ref{BRW_deviation} to obtain
\begin{align*}
 \pr( B(\delta t)\ge xt)&= \exp( \delta t\inf_{s\in(0,1)}\{ (1-f'(q_*))\log s+f(s)/s\} + o(t)).
\end{align*}
 Lemma \ref{kappa_properties}(i) then shows that the infimum of $s \mapsto (1-f'(q_*))\log s+f(s)/s$ on $(0,1)$ is uniquely achieved at $s=q_*$, which leads to 
\beq\label{BRW_lb}
 \pr( B(\delta t)\ge xt)= \exp\left( \delta t \left((1-f'(q_*))\log q_*+1\right) + o(t)\right)=\exp( -(1-\ep - \delta)t + o(t))
\eeq
as  $\delta= \frac{1-\ep}{(1-f'(q_*))\log(1/q_*)}$. 

Combining \eqref{yule_deviation} and \eqref{BRW_lb} in \eqref{height_lb} yields for large enough $t$,
\begin{align*}
\pr( \cH_{\cT(t)}< xt)&\leq \left(1-\pr( B(\delta t)\ge xt)\right)^{e^{(1-\delta)t-t^{1/2}}}+2e^{-t^{1/2}}\\
&\leq \exp\left( -\pr( B(\delta t)\ge xt)e^{(1-\delta)t-t^{1/2}}\right)+2e^{-t^{1/2}}\\
&= \exp\left( -\exp(-(1-\ep - \delta)t +(1-\delta)t+o(t))\right)+2e^{-t^{1/2}}\\
&\leq \exp(-\exp(\ep t/2))+2e^{-t^{1/2}}.
\end{align*}
The same argument as in the previous cases with the Borel-Cantelli Lemma would lead to
$$\liminf_{t\to\infty} \frac{ \cH_{\cT(t)}}{t} \geq (\log(1/q_*))^{-1}  \quad \text{ almost surely},$$
which concludes our proof.
\end{proof}

\color{black}

\section{Proofs: PageRank asymptotics}
\label{sec:proof-pagerank}
In this section, we prove Theorems \ref{thm:page-rank-fixed} and \ref{thm:pr2}. As before, Assumptions \ref{ass:mean} and \ref{ass:pmf} continue to hold and are not explicitly stated in the results.

As in the case of degrees, our analysis will rely on continuous time versions of PageRank. Consider the PageRank of the root in $\TT^*(t)$ with damping factor $c$, namely, 
\beq\label{cpr1}
R^*_c(t)=(1-c)\left(1 + \sum_{l=1}^\infty c^l \cP_l(t)\right).
\eeq
and the Pagerank of the root in $\TT(t)$ given by 
\beq\label{cpr2}
R_c(t)=(1-c)\left(1+\sum_{l=1}^\infty c^l \tilde \cP_l(t)\right).
\eeq

We begin with a lower bound on $\E[\cpst_s(t)]$ given in Lemma \ref{plb}, which will play a key role in showing that the limiting random variable $W_{k,c}$ in Theorem \ref{thm:page-rank-fixed} is non-degenerate. This lower bound involves a `change of measure' argument which we now present.

For any $s>0$, define a probability transition kernel 
$$p^*_s(x,y)=\frac{s^{-y} p_{x+1-y}}{\alpha(s) s^{-x}}, \quad  x,y\in \mathbb{Z} \text{ and }-\infty<y\leq x+1,$$
where $\alpha(s)=\frac{f(s)}{s}$. We can check that $p^*_s$ is indeed a probability transition kernel by 
$$\sum_{y\leq x+1} \frac{s^{-y} p_{x+1-y}}{\alpha(s) s^{-x}}=\frac{s^{-(x+1)}f(s)}{\alpha(s) s^{-x}}=1.$$
Let $S^*_s$ be a discrete time random walk following the transition kernel $p^*_s$. Set $S^*_s(0)=0$ and define 
$$\tau^*_s=\inf \{ n\geq 1: S^*_s(n)\leq 0\}.$$

\begin{lemma}\label{plb}
For any $s>0$ and  $t\geq 0$,
$$\E[\cpst_s(t)]\geq \pr(\tau^*_s=\infty) e^{\frac{f(s)}{s} t}.$$
In particular, when $\E[Z]<1$ and $s_0>1$, for any $s \in [1,s_0)$ and $t\geq 0$,
$$e^{-\frac{f(s)}{s} t}\E[\cpst_s(t)]\geq \pr(\tau^*_s=\infty)>0.$$
\end{lemma}
\begin{proof}
Recall the random walk $S$ defined in \eqref{eqn:rw-def}, and recall from Lemma \ref{Pk} that for all $k \ge 1$,
\begin{equation}\label{p1}
\E[\cP_k(t)] = \sum_{i=0}^{\infty}\frac{t^i}{i!}\pr(\barT_k=i),
\end{equation}
where $\barT_k:=\inf\{ n\geq 0: S_n=0| S_0=k\}$. Consider the random walk $\tilde S$ given by $\tilde S_0=0$ and 
$$\tilde S_n := \sum_{j=1}^n (1-Z_j), \, n \ge 1.$$ 
Also, define $\tilde \tau := \inf\{j \ge 1: \tilde S_j \le 0\}.$ By a time-reversal argument, it readily follows that for any $i,k \ge 1$,
$$
\pr(\barT_k=i) = \pr(\tilde S_i = k, \tilde \tau > i).
$$
Using this and \eqref{p1}, we obtain
\begin{equation}\label{p2}
 \E[\cpst_s(t)] = \sum_{i=0}^{\infty}\frac{t^i}{i!} \sum_{k=1}^{\infty} s^{-k} \pr(\tilde S_i = k, \tilde \tau > i).
\end{equation}
The crucial elementary algebraic identity connecting $\tilde S$ to the random walk $S^*_s$ defined before the lemma is the following:
\beq\label{changekernel}
s^{-k}\pr(\tilde S_i=k, \tilde \tau > i)=(\alpha(s))^i \pr(S_s^*(i)=k, \tau^*_s> i), \ i \ge 0, k \ge 1,
\eeq
where $\alpha(s) = \frac{f(s)}{s}$. Using this observation in \eqref{p2}, we obtain
\begin{align*}
\E[\cpst_s(t)]
&=\sum_{i=0}^\infty  \frac{t^i}{i!} (\alpha(s))^i \sum_{k=1}^\infty \pr(S_s^*(i)=k, \tau^*_s>i)\\
&=\sum_{i=0}^\infty \frac{ (\alpha(s) t)^i}{i!} \pr(\tau_s^*>i)\geq e^{\alpha(s)t}\pr(\tau^*_s=\infty).
\end{align*}
This proves the first assertion in the lemma. To prove the second assertion, namely $\pr(\tau^*_s=\infty)>0$ for $s \in [1,s_0)$ in the case $\E[Z] <1$, it suffices to show $\E[S^*_s(1)]>0$ (this implies the result, e.g., by Lemma 11.3 of \cite{gut2009stopped}). We compute the probability generating function $f^*(\cdot)$ (suppressing dependence on $s$) of $S^*_s(1)$:
\begin{align*}
f^*(\psi)&:=\sum_{j=-\infty}^1 \psi^j p^*_s(0,j)=\sum_{j=-\infty}^1 \psi^j \frac{s^{-j}p_{1-j}}{\alpha(s)}=\frac{\psi f(\frac{s}{\psi})}{f(s)}, \ \psi \ge 0.\\
\end{align*}
Then, 
$$\E[S^*_s(1)] = (f^*)'(1)=\frac{1}{f(s)}\left[ f(\frac{s}{\psi})-\frac{s}{\psi}f'(\frac{s}{\psi})\right]\bigg|_{\psi=1}=\frac{f(s)-sf'(s)}{f(s)}.$$
Let $g(s)=f(s)-sf'(s)$. Since $g(1)=1 - \E[Z]>0$, by definition of $s_0$, we have $g(s)>0$ for any $s\in [1,s_0)$ and hence $\E[S^*_s(1)]>0$. This completes the proof.
\end{proof}

The following theorem proves the analogous assertions of Theorem \ref{thm:page-rank-fixed} for the continuous time versions of the PageRank defined in \eqref{cpr1} and \eqref{cpr2}. Recall $T_n=\inf\{ t\geq 0: |\TT(t)|=n+1\}$
\begin{theorem}\label{Pagerank_con}
Fix $c \in (0,1)$.
\begin{enumeratea}
\item {\bf Non-fringe regime: }When $\E[Z]>1$,
\begin{enumeratei}
\item For any $\delta>0$, there exists $\ep>0$ such that
\beq\label{PRrootnf}
\liminf_{n \rightarrow \infty}\pr\left(\frac{R_{c}(T_n)}{n} \ge c(1-c)(1-q_*) - \ep\right)\ge 1-\delta.
\eeq
\item For any $\delta>0$,
\beq\label{PR_lu_nf}
\lim_{t\to\infty} e^{-\left(\frac{1}{R}-\delta\right)t}R^*_c(t)=\infty \quad \text{and}  \quad  \lim_{t\to\infty} t^{-(1+\delta)}e^{-t/R}R^*_c(t)=0 \quad \text{a.s.}
\eeq
\end{enumeratei}
\item {\bf Fringe regime: }When $\E[Z]\leq 1$, 
\begin{enumeratei}
\item Fix any $c\in (0,s_0^{-1}]$ with $c<1$. Then for any $\delta>0$,
\begin{align}
\label{PRrootf} \lim_{t\to\infty} e^{-\left(\frac{1}{R}-\delta\right)t}R_c(t)=\infty \quad &\text{and}  \quad  \lim_{t\to\infty} t^{-(1+\delta)}e^{-t/R}R_c(t)=0 \quad\text{a.s.}\\
\label{PRvertex}\lim_{t\to\infty} e^{-\left(\frac{1}{R}-\delta\right)t}R^*_c(t)=\infty \quad &\text{and} \quad \lim_{t\to\infty} t^{-(1+\delta)}e^{-t/R}R^*_c(t)=0 \quad \text{a.s.} 
\end{align}
\item Suppose $s_0> 1$. For any $c\in( s_0^{-1},1)$, there exist non-negative random variables $W_c$ and $W^*_c$  with $\pr(W_c>0)>0$ and $\pr(W^*_c>0)>0$ so that as $t\to\infty$,
$$e^{-cf(1/c)t}R_c(t)\overset{a.s.}{\longrightarrow} W_c \quad\text{ and }\quad e^{-cf(1/c)t}R^*_c(t)\overset{a.s.}{\longrightarrow} W^*_c.$$
\end{enumeratei}
\end{enumeratea}
\end{theorem}

\begin{proof}
We first focus on the  non-fringe regime.

Denoting the size of $\TT(t)$ by $|\TT(t)|$, observe that 
$$R_c(t) \geq (1-c)c\tilde \cP_1(t).$$
\eqref{PRrootnf} then follows from \eqref{tnconvp}.

For $R^*_c(t)$ with any $c \in (0, 1)$, recalling from Lemma \ref{lemma:prop-R-s0}(d) that $s_0^{-1}>1$,
\beq\label{sand}
(1-c)c\cP_1(t)\leq R^*_c(t) \leq (1-c)\left(1 + \sum_{l=1}^\infty s_0^{-l} \cP_l(t)\right)=(1-c)(1+\cpst(t)).
\eeq
Hence, the first limit in \eqref{PR_lu_nf} follows from \eqref{lower} and the first inequality in \eqref{sand}. 

To show the second limit in \eqref{PR_lu_nf} take any $\delta>0$. For $\ep>0$ and $N\geq 0$, define the event 
$$E_N=\{ \sup_{u \in [N,N+1]} u^{-(1+\delta)} e^{-u/R} R^*_{c}(u)>\ep\}.$$ 
It follows from the upper bound in \eqref{sand} and \eqref{Ptub} that for any $c \in (0, 1)$ and $N \ge 1$,
\begin{align*}
\pr(E_N)&\leq \pr(R^*_{c}(N+1)>\ep e^{N/R}N^{1+\delta})\leq \frac{\E[R^*_{c}(N+1)]}{\ep e^{N/R}N^{1+\delta}}\\
&\leq  \frac{(1-c)(1+\E[\cpst(N+1)])}{\ep e^{N/R}N^{1+\delta}}\leq \frac{(1-c)\left(1+\frac{p_0}{f(s_0)}e^{(N+1)/R}\right)}{\ep e^{N/R}N^{1+\delta}}
\leq  \frac{C e^{1/R}}{\ep N^{1+\delta}}
\end{align*}
for some constant $C>0$. Applying Borel-Cantelli Lemma then gives $\pr(\limsup_{N\to\infty} E_N)=0$. Hence,
$$t^{-(1+\delta)}e^{-t/R}R^*_c(t) \overset{a.s.}{\longrightarrow}0,$$
proving the second limit in \eqref{PR_lu_nf}.

We now turn to the fringe regime. Define $\tilde\cpst(t) = \sum_{i=1}^\infty s_0^{-i} \tilde \cP_i(t),$
$$(1-c)c\tilde \cP_1(t)\leq  R_c(t) \leq (1-c)\left(1 + \tilde \cpst(t)\right).$$
The first limit in \eqref{PRrootf} comes from the observation that $\tilde\cP_1(t)$ dominates $\cP_1(t)$ and Corollary \ref{corr:cp1-infty}. The second limit follows from an argument that is essentially the same as that of the second limit in \eqref{PR_lu_nf} except that we apply Corollary \ref{updegree_root}(i) for the expectation of $\tilde \cpst(t)$. 

The proof of \eqref{PRvertex} is essentially the same as that of \eqref{PR_lu_nf} upon noting that the bounds in \eqref{sand} remain valid for $c \in (0, s_0^{-1}]$.

It remains to prove (ii) of part (b). For $t\geq 0$, define
$$M^*_c(t)= e^{-cf(1/c)t}R^*_c(t).$$
Using the generator expression in \eqref{eqn:full-gene} (see also Remark \ref{rem:mar}), observe that for any $t \ge 0$, 
\begin{align*}
\LL (M^*_c(t))&= -cf(1/c)e^{-cf(1/c)t}R^*_c(t)+ e^{-cf(1/c)t}(1-c) \sum_{i=1}^\infty c^i \sum_{j=i-1}^\infty A_{ij}\cP_j(t)\\
&= -cf(1/c)e^{-cf(1/c)t}R^*_c(t)+ e^{-cf(1/c)t}(1-c)\sum_{j=0}^\infty \left(\sum_{i=1}^{j+1} c^i A_{ij}\right)\cP_j(t)\\
&= -cf(1/c)e^{-cf(1/c)t}R^*_c(t)+ e^{-cf(1/c)t}(1-c)c\sum_{j=0}^\infty \left(\sum_{i=1}^{j+1} c^{i-j-1} p_{j-i+1}\right)c^j\cP_j(t)\\
&\le -cf(1/c)e^{-cf(1/c)t} R^*_c(t) + cf(1/c)e^{-cf(1/c)t}R^*_c(t)=0.
\end{align*}
Hence, $M^*_c(t)$
 is a non-negative supermartingale and thus 
 \beq\label{convergence_PR}
 M^*_c(t) \overset{a.s.}{\longrightarrow} W^*_c \quad\text{ as }t \to \infty,
 \eeq
 for some non-negative random variable $W^*_c$. 
 
 By Theorem \ref{updegree}(ii), we conclude that $\sup_{t < \infty} \E(M^*_c(t))^2<\infty$. Using this observation and Lemma \ref{plb},
 $$
 \E(W^*_c) = \lim_{t \rightarrow \infty}\E(M^*_c(t)) \ge \pr(\tau^*_{1/c}=\infty) >0.
 $$
This implies that $\pr(W^*_c>0)>0$.

The proof for the convergence of $M_c(t):= e^{-cf(1/c)t}R_c(t)$ to some $W_c$ follows similarly upon noting that when $\E[Z]\leq 1$, $(c^i: i\geq 0)$ is a sub-invariant left eigenvector of $\vB$ with eigenvalue $cf(1/c)$, see Proposition \ref{prop:spectral-prop}(b). We can show $\LL (M_c(t))\leq 0$, which implies that  $M_c(t)$ is a non-negative supermartingale that converges to a random variable $W_c$. By Corollary \ref{updegree_root}(ii),  $\sup_{t<\infty}\E(M_c(t)^2)<\infty$. Combined with the fact that  $\tilde{\cP}_i(t)$ dominates $\cP_i(t)$ for all $i\geq0,t\geq 0$, we have
$$\E(W_c)=\lim_{t\to\infty}\E(M_c(t))\geq \lim_{t\to\infty}\E(M^*_c(t))>0,$$
proving $\pr(W_c>0)>0$. This concludes the proof of (ii) of part (b), and the Theorem. 
\end{proof}

\begin{proof}[{\bf Proof of Theorem \ref{thm:page-rank-fixed}}]
Using the continuous time embedding, observe that for any $k \ge 0$,
$$\{ R_{v_k,c}(n), n\geq v\}\overset{d}{=}\{ R^\circ_{v_k,c}(T_n-\sigma_k), n\geq v\},$$
where $\sigma_k$ is the birth time of $v_k$, $T_n = \inf\{t \ge 0 : |\TT(t)| = n+1\}$ and $R^\circ_{v_k,c}(t)$ is the PageRank of $v_k$ in $\TT(t + \sigma_k)$. Part (a) and part (b)(i) of the Theorem now follow from Theorem \ref{Pagerank_con} upon noting that $\{R^\circ_{v_k,c}(t) : t \ge 0\}$ has the same distribution as $\{R_c(t) : t \ge 0\}$ for $k=0$ and $\{R^*_c(t) : t \ge 0\}$ for $k\ge 1$.

To prove (b)(ii), note that for $k \ge 1$,
\begin{align*}
\frac{R_{v_k,c}(n)}{n^{cf(1/c)}}&=\frac{R^\circ_{v_k,c}(T_n-\sigma_k)}{e^{cf(1/c)(T_n-\sigma_k)}} \cdot \left(\frac{ e^{T_n-\sigma_k}}{n}\right)^{cf(1/c)}\overset{a.s.}{\longrightarrow} e^{-\sigma_k cf(1/c)} \tilde{W}_{k,c} W^{-cf(1/c)}=: W_{k,c},
\end{align*}
where the limit $W:=\lim_{n\to\infty} ne^{-T_n}$ almost surely exists by Lemma \ref{lem:yule-asymp}. The random variable $\tilde{W}_{k,c}\overset{d}{=}W^*_c$ for $k\geq 1$ and $\tilde{W}_{k,c}\overset{d}{=}W_c$ for $k=0$, where $W^*_c$ and $W_c$ are obtained in Theorem \ref{Pagerank_con} and The result follows.
\end{proof}

\begin{proof}[{\bf Proof of Theorem \ref{thm:pr2}}]
Recall the form of the limiting PageRank distribution given in \eqref{lwlpr}, from which we conclude that $\cR_{\emptyset, c}(\infty)$ has the same distribution as $R_c^*(\tau)$ where $\tau$ is a unit rate exponential random variable independent of the random tree process $\TT(\cdot)$.

Part (a) and the assertion in part (b) for $c \in (0, s_0^{-1}]$ with $c<1$ now follow along the same argument as in the proof of Theorem \ref{thm:deg-dist} upon using the bounds in \eqref{sand}. The analogue of \eqref{eqn:802} in part (a) uses Corollary \ref{tau_moment}(ii) and that in part (b) uses Corollary \ref{tau_moment}(i). The other direction follows from \eqref{lld} and the lower bound in \eqref{sand}.

It only remains to prove the assertion in part (b) for $c \in (s_0^{-1},1)$. Observe that for any $\delta \in (0, \frac{1}{cf(1/c)})$, $r>0$, using Markov's inequality and Corollary \ref{tau_moment}(i), 
$$
r^{\frac{1}{cf(1/c)} - \delta}\pr(\cR_{\emptyset, c}(\infty)\geq r) \le \E\left[(\cR_{\emptyset, c}(\infty))^{\frac{1}{cf(1/c)} - \delta}\right]\le \E\left[\left((1-c)(1+ \cpst_{1/c}(\tau)\right)^{\frac{1}{cf(1/c)} - \delta}\right] < \infty,
$$
which implies 
$$
\limsup_{r \rightarrow \infty}\frac{\log \pr(\cR_{\emptyset, c}(\infty)\geq r)}{\log r} \le - \frac{1}{cf(1/c)}.
$$
Moreover, by the almost sure convergence in Theorem \ref{Pagerank_con} (b)(ii) to a non-degenerate non-negative random variable, there exist positive $\eta_1, \eta_2, t_0$ such that
$$
\pr\left(e^{-cf(1/c)t}R^*_c(t) \ge \eta_1\right) \ge \eta_2, \ t \ge t_0.
$$
Now, proceeding as in Theorem \ref{thm:deg-dist}, for $r \ge \eta_1e^{cf(1/c)t_0}$,
\begin{align*}
&\pr(\cR_{\emptyset, c}(\infty)\geq r) = \int_0^{\infty}e^{-s}\pr\left(R_c^*(s) \ge r\right)ds = \int_0^{\infty}e^{-s}\pr\left(e^{-cf(1/c)s}R_c^*(s) \ge e^{-cf(1/c)s}r\right)ds\\
& \ge \int_{\frac{1}{cf(1/c)}\log(r/\eta_1)}^{\infty}e^{-s}\pr\left(e^{-cf(1/c)s}R_c^*(s) \ge \eta_1\right)ds \ge \eta_2\int_{\frac{1}{cf(1/c)}\log(r/\eta_1)}^{\infty}e^{-s}ds = \eta_2\left(\frac{\eta_1}{r}\right)^{\frac{1}{cf(1/c)}}.
\end{align*}
This implies
$$
\liminf_{r \rightarrow \infty}\frac{\log \pr(\cR_{\emptyset, c}(\infty)\geq r)}{\log r} \ge - \frac{1}{cf(1/c)},
$$
completing the proof of the theorem.
\end{proof}

\section*{Acknowledgements}
Banerjee was supported in part by the NSF CAREER award DMS-2141621. Bhamidi was supported in part by NSF DMS-2113662. Research was partially supported by NSF RTG grant DMS-2134107. \tcr{The authors thank an anonymous referee and an editor whose valuable comments led to a major improvement in the paper, in particular allowing us to realize the existence of potential phase transition behavior for the height in the non-fringe regime.}

\bibliographystyle{plain}
\bibliography{persistence,pref_change_bib,scaling,pdn,CTBP}

\appendix
\section{$\alpha$-recurrence and transience of $\mu^\vB$ and $\mu^\vA$}\label{apprec}

A full account of $\alpha$-recurrence/transience for Markov chains can be found in Chapter 2 of \cite{nummelin2004general}. To be concise we only state what is needed for the rigorous definition of $\alpha$-recurrence.

 Let $\varphi$ be a $\sigma$-finite measure on $S$ and let $A\in\mathcal{S}$ be a $\varphi$-positive set (i.e., $\varphi(A)>0$).
\begin{defn}[irreducible measure]\label{irrdef}
The set $A$ is called $\varphi$-communicating for a Markov kernel $K$ on $S\times S$ if every $\varphi$-positive subset $B\subseteq A$ is attainable from $A$, i.e., $K^{(n)}(x,B)>0$ for some $n\geq 0$ for all $x\in A$ and all $\varphi$-positive $B\subseteq A$.

If the whole state space $S$ is $\varphi$-communicating, then the kernel $\mu$ is called $\varphi$-irreducible and $\varphi$ is called an irreducibility measure for $K$.
\end{defn}

It is clear that any measure $\psi$ which is absolutely continuous with respect to an irreducibility measure $\varphi$ is itself an irreducibility measure. 
\begin{prop}[Proposition 2.4 in \cite{nummelin2004general}]
Suppose kernel $K$ is $\varphi$-irreducible. Then there exists a maximal irreducibility measure $\psi$ in the sense that all other irreducibility measures are absolutely continuous with respect to $\psi$.
\end{prop}

Now we present a definition of $\alpha$-recurrence in the context of our problem. It follows from the discussion on page 194 of  \cite{jagers1989general} (see also Proposition 2.1 of \cite{niemi1986non}) that 
\begin{defn}[$\alpha$-recurrence]\label{recurrence}
Let $\varphi$ be the maximal irreducibility measure for  the kernel $\mu(\cdot,\cdot\times \RR_+)$. The kernel $\mu(\cdot,\cdot\times \RR_+)$ is said to be $\alpha$-recurrent if 
$$\sum_{n=0}^\infty \mu_{\alpha}^{(n)}(s,A\times \RR_+)=\infty$$ 
for all $s\in S$ and $A\in \mathcal{S}$ with $\varphi(A)>0$. Otherwise it is said to be $\alpha$-transient.
\end{defn}

For the rest of this section we will discuss the $\alpha$-recurrence/transience of the continuous time branching process embedding of our model.

\begin{lemma}
When $\E[Z]>1$, the kernel $\mu^\vB(s,r\times dt):=B_{rs}dt$ is $\alpha$-recurrent.
\end{lemma}
\begin{proof}
In this case the Malthusian parameter $\alpha=1$. Hence
$$\mu^\vB_{\alpha}(s,r\times \RR_+)=\sum_{n=0}^\infty \frac{ (\vB^n)_{rs}}{\alpha^n}=\sum_{n=0}^\infty (\vB^n)_{rs}.$$
Let $\hat \vB$ denote the matrix obtained by restricting $\vB$ to states $\{1,2,\dots\}$. As $\E[Z]>1$ implies $p_0 + p_1 <1$, it can be checked that $\hat \vB$ is irreducible, which follows from the fact that for any $r,s\in \mathbb{N}$, there exist $m_1,m_2\in \mathbb{N}$ such that $(\hat{\vB}^{m_1})_{r1}>0$ and $(\hat \vB^{m_2})_{1s}>0$.

First, we will show $\sum_{n=0}^\infty (\hat \vB^{n})_{11}=\infty$ which, by irreducibility, will imply $\sum_{n=0}^\infty (\hat \vB^{n})_{rs}=\infty$ for any $r, s \in \mathbb{N}$. Notice that $\hat \vB^n_{11}=((\hat \vB^T)^n)_{11}$ and $\hat \vB^T$ is a probability transition matrix. Let $\{X_n\}_{n\geq 1}$ denote the Markov chain following the transition matrix $\hat \vB^T$ started at $X_0=1$ and let $\tau=\inf\{ n\geq 1: X_n=1\}$ be the first returning time to 1. It is well known that $\sum_{n=0}^\infty ((\hat \vB^T)^{n})_{11}=\infty$ if and only if $X_n$ is recurrent. Therefore, to show  the $\alpha$-recurrence of $\mu^\vB$ we only need to show $P(\tau<\infty|X_0=1)=1$. Note that
\begin{align*}
\E[X_n-X_{n-1}|X_{n-1}=i]&=c_i+\sum_{k=0}^{i-1}(i+1-k)p_k-i\\
&=c_i+(i+1)\sum_{k=0}^{i-1}p_k-\sum_{k=0}^{i-1}kp_k-i=1-(c_ii+\sum_{k=0}^{i-1}kp_k)\\
&=1-\E[Z]+\sum_{k=i}^\infty (k-i)p_k \downarrow 1-\E[Z] \quad \text{ as  $i\to \infty$}.
\end{align*}
When $\E[Z]>1$, there exists $i_0\in \mathbb{N}$ such that for all $i\geq i_0$, $\E[X_n-X_{n-1}|X_{n-1}=i]<0$. Let $\tau_{i_0}=\inf\{ n\geq 0: X_n\leq i_0\}$. We claim that for all $j>i_0$ we have $\pr_{j}(\tau_{i_0}<\infty)=1$.  If we view the set $A=\{1,2,\dots,i_0\}$ as an absorbing state and let $Y_n$ be a Markov chain obtained by projecting $X_n$ onto $\{A,i_0+1,\dots\}$ with $Y_0=j$, then $Y_n$ is a non-negative supermartingale and has to converge almost surely. Note that $Y_n$ can only converge by getting absorbed at $A$, which implies $\pr_{j}(\tau_{i_0}<\infty)=1$. This means that $X_n$ returns to the set $A$ infinitely often almost surely and thus, by irreducibility of $\hat \vB^T$, will return to 1 in finite time almost surely.

For $r\in \mathbb{N}$ and $s=0$, note that $(\vB^n)_{r0}\geq  (\vB^{m_1})_{r1} (\vB^{n-(m_1+1)})_{11} \vB_{10}$ for $n \ge m_1+1$ and the above argument implies $\sum_{n=0}^\infty (\hat \vB^{n})_{r0}=\infty$. It follows from Definition \ref{irrdef} that the maximal irreducibility measure $\varphi$ for $\mu^\vB$ must satisfy $\varphi(\{0\})=0$. Hence we don't need to consider the case where $r=0$.
\end{proof}

\begin{lemma}
\label{lemma:appen-A-trans}
The kernel $\mu^\vA(s,r\times dt):=A_{rs}dt$ is $\alpha$-transient.
\end{lemma}
\begin{proof}
Recall the function $\chi(u)=\sum_{n=0}^\infty u^n \pr_1(T=n)$ as defined in the proof of Lemma \ref{Phi_func}.

It is easily checked that for $s,r\in \{0,1,\dots\}$,
$$\sum_{n=0}^\infty (\mu^\vA_{\alpha})^{(n)}(s,r\times \RR_+)=\sum_{n=0}^\infty \frac{ (\vA^n)_{rs}}{\alpha^n}.$$
Take $(s,r)=(0,1)$ and note that $\varphi(1)>0$. Since $(\vA^n)_{10}=\pr_1(T=n)$ we can observe that
$$\sum_{n=0}^\infty (\mu^\vA_{\alpha})^{(n)}(0,1\times \RR_+)=\sum_{n=0}^\infty \frac{ \pr_1(T=n)}{\alpha^n}=\chi(1/\alpha).$$
Note that the Perron root of $\vA$ is $1/R$ and hence the corresponding Malthusian rate is $\alpha=1/R$.
It follows from the proof of Lemma \ref{Phi_func} (see \eqref{chi_func}) that $\chi(1/\alpha)=\chi(R)<\infty$, which immediately implies the $\alpha$-transience for $\mu^\vA$.

\end{proof}

\end{document}